\documentclass{amsart}
\newtheorem{theorem}{Theorem}[section]
\newtheorem{definition}[theorem]{Definition}
\newtheorem{proposition}[theorem]{Proposition}
\newtheorem{lemma}[theorem]{Lemma}

\usepackage{hyperref,amssymb}
\usepackage{listings}
\lstdefinelanguage{GAP}{%
	morekeywords={%
		Assert,Info,IsBound,QUIT,%
		TryNextMethod,Unbind,and,break,%
		continue,do,elif,%
		else,end,false,fi,for,%
		function,if,in,local,%
		mod,not,od,or,%
		quit,rec,repeat,return,%
		then,true,until,while%
	},%
	sensitive,%
	morecomment=[l]\#,%
	morestring=[b]",%
	morestring=[b]',%
}[keywords,comments,strings]
\usepackage[T1]{fontenc}
\usepackage[variablett]{lmodern}
\usepackage{xcolor}
\begin{document}

\title[IBIS large base]{IBIS primitive groups of almost simple type}

%\author[F. Dalla Volta]{Francesca Dalla Volta}
%\address{Francesca Dalla Volta, Dipartimento di Matematica Pura e Applicata,\newline
% University of Milano-Bicocca, Via Cozzi 55, 20126 Milano, Italy} 
%\email{f.dallavolta@unimib.it}

\author[F. Mastrogiacomo]{Fabio Mastrogiacomo}
\address{Fabio Mastrogiacomo, Dipartimento di Matematica ``Felice Casorati", University of Pavia, Via Ferrata 5, 27100 Pavia, Italy} 
\email{fabio.mastrogiacomo01@universitadipavia.it}

\author[P. Spiga]{Pablo Spiga}
\address{Pablo Spiga, Dipartimento di Matematica Pura e Applicata,\newline
 University of Milano-Bicocca, Via Cozzi 55, 20126 Milano, Italy} 
\email{pablo.spiga@unimib.it}
\subjclass[2010]{primary 205}
\thanks{The authors are members of the GNSAGA INdAM research group and kindly acknowledge their support.}
	\maketitle
\begin{abstract}        Let $G$ be a finite permutation group on $\Omega$. An ordered sequence 
        $(\omega_1\ldots,\omega_\ell)$ of elements of $\Omega$ is an irredundant base for $G$ if the pointwise stabilizer is trivial and no point is fixed by the stabilizer of its predecessors. The minimal cardinality of an irredundant base is said to be the base size of $G$. If all irredundant bases of $G$ have the same cardinality, $G$ is said to be an IBIS group. 
        
In this paper, we classify the finite almost simple primitive IBIS groups whose base size is at least $6$.
\end{abstract}
\section{Introduction}\label{introduction}
A \textit{\textbf{base}} for a permutation group $G$ is a sequence $(\omega_1 , \ldots , \omega_\ell )$ whose pointwise stabilizer in $G$ is the identity. A base is
said to be \textit{\textbf{irredundant}} if no base point is fixed by the stabilizer of its predecessors. Therefore, an irredundant base $(\omega_1,\dots,\omega_\ell)$ gives
rise to a strictly decreasing sequence of stabilizers
$$
G > G_{\omega_1} > G_{\omega_1 ,\omega_2} > \cdots > 
G_{\omega_1 ,\omega_2 ,\ldots,\omega_{\ell-1}} > 
G_{\omega_1,\omega_2,\ldots,\omega_\ell}= 1.$$

The minimal cardinality of a base, called the \textit{\textbf{base size}} of $G$, is usually denoted by $b(G)$ and has played a key role in
the investigation of primitive groups.

Cameron and Fon-der-Flaass~\cite{CF} (see also~\cite[Section~4.14]{Peter}) have proved that in a finite permutation group the following conditions are equivalent:
\begin{itemize}
\item all irredundant bases have the same size;
\item the irredundant bases are invariant under re-ordering;
\item the irredundant bases are the bases of a matroid.
\end{itemize}
A permutation group satisfying one, and hence all, of these conditions is said to be \textit{\textbf{IBIS}}, that is, \textit{\textbf{Irredundant Bases of Invariant Size}}. In particular, in an IBIS group $G$, all irredundant bases have cardinality $b(G)$ and this value is sometimes referred to as the \textit{\textbf{rank}} of the IBIS group, since it is the rank of the corresponding matroid. Aside from this introductory section, we avoid using the term rank with this meaning, because it can be easily confused with other properties of permutation groups that also bear this name.

Lucchini, Morigi and Moscatiello~\cite{LuMoMo} made the first attempt of classifying finite primitive IBIS groups. Their approach is via the O'Nan-Scott classification of primitive groups (we use the terminology from~\cite{Pra90}). 
\begin{theorem}[{Theorem~1.1,~\cite{LuMoMo}}]Let $G$ be a primitive IBIS group. Then one of the following holds:
\begin{enumerate}
\item $G$ is of affine type,
\item $G$ is almost simple,
\item $G$ is of diagonal type.
\end{enumerate}
Moreover, $G$ is a primitive IBIS group of diagonal type if and only if it belongs to the infinite family of diagonal groups $\{\mathrm{PSL}_2(2^f)\times \mathrm{PSL}_2(2^f)\mid f\in\mathbb{N},f\ge 2\}$ having degree $|\mathrm{PSL}_2(2^f)|=2^f(4^f-1)$.
\end{theorem}
In light of this result, the problem of understanding finite primitive IBIS groups is reduced to affine groups and to almost simple groups. A first attempt of classifying finite primitive IBIS groups of affine type can be found in~\cite{LuMa}.

Lee and the second author~\cite{LeSp} have classified the almost simple primitive IBIS groups having socle an alternating group. In the same paper, the authors have proposed a conjecture; indeed, they have conjectured that every almost simple primitive IBIS group is one of the groups in~\cite[Table~1]{LeSp}. However, we discovered two additional families of primitive $\mathrm{IBIS}$ groups that do not appear in the table of \cite{LeSp}, namely $\Omega_d^\pm(q)$, with $q=2^f$, acting on non-singular $1$-dimensional subspaces. For completeness, we have reported the whole of ~\cite[Table~1]{LeSp} in Table~\ref{table:table} together with the two additional families we have discovered.

Later, Lee~\cite{Lee} has classified the almost simple primitive IBIS groups having socle a sporadic simple group.

There are two special circumstances where IBIS groups arise. A permutation group $G$ on $\Omega$ is said to be \textit{\textbf{sharply $k$-transitive}} if, for any two $k$-tuples of distinct elements of $\Omega$, there exists a unique element of $G$ mapping the first
$k$-tuple to the second. Observe that if $G$ is sharply $k$-transitive, then $b(G) = k$ and $G$ is an IBIS group. Moreover, a
transitive permutation group $G$ is said to be \textit{\textbf{geometric}} (or,  base-transitive, see~\cite{Robert}) if $G$ permutes its irredundant bases transitively. Clearly, every sharply $k$-transitive group is geometric and
every geometric group is IBIS.\footnote{In Table~\ref{table:table}, the third column indicates whether the group is sharply $k$-transitive, denoted by a $\checkmark$, and similarly, whether the group is geometric,  marked with a $\checkmark$ in the forth column.}  

In this paper, we deal with almost simple primitive groups having socle a simple group of Lie type.
\begin{theorem}\label{thrm:main}
Let $G$ be an almost simple primitive group having socle a simple group of Lie type $G_0$, with $b(G)\ge 6$ when $G_0$ is exceptional and $b(G)\ge 5$ when $G_0$ is classical. Then $G$ is IBIS if and only if $G$ is either $\mathrm{SL}_d(2)$ or $\mathrm{Sp}_d(2)$ in its natural action on the non-zero vectors of a $d$-dimensional vector space over the field with two elements, or $G=\Omega_d^\pm(q)$ with $d\ge 4$ and $q\ge 4$ acting on the non-singular $1$-dimensional subspaces of the $d$-dimensional vector space over the field with $q$ elements.
\end{theorem}

Our proof of Theorem~\ref{thrm:main} uses the astonishing work on the Cameron-Kantor conjecture~\cite[Theorem~$2.1$]{cameron}. Most interest in the base size of primitive groups originated from Jordan's classic results~\cite{jordan}, which bound the cardinality of a primitive group via its base size. This interest was spurred in the '90s by the Cameron-Kantor conjecture~\cite{cameron}: there exists an absolute constant $b$ with $b(G)\le b$, for every almost simple primitive group in a non-standard action. (We refer to~\cite{cameron} and to Section~\ref{largebasesize} for the definition of \textit{\textbf{standard action}}). This conjecture was settled in the positive in~\cite{BLS09}; however, the refinement of Cameron~\cite[Chapter~$4$, Page~$120$]{Peter}, asking whether one can take $b=7$ has required considerable more effort. The detailed analysis of Burness~\cite{B1,B2,B3,B4} on fixed point ratios of classical groups has resulted in a positive answer to Cameron's question in~\cite{Bur07}, together with \cite{BLS09}.

Therefore, our proof of Theorem~\ref{thrm:main} uses deatailed information on the standard actions. Indeed, we prove this stronger result.
\begin{theorem}\label{thrm:main2}
Let $G$ be an almost simple primitive group having socle a simple classical group  in a standard action. Then $G$ is IBIS if and only if one of the following holds:
\begin{enumerate}
\item $G=\mathrm{SL}_d(2)$ or $G=\mathrm{Sp}_d(2)$ in its natural action on the non-zero vectors of an $n$-dimensional vector space over the field with two elements,
\item $G$ has socle $\mathrm{PSL}_2(q)$ and $G$ is endowed of its natural action on the $q+1$ points of the projective line,
\item the socle of $G$ is $\mathrm{Sp}_4(2)'$  in its natural action on the fifteen non-zero vectors of a $4$-dimensional vector space over the field with two elements.
\item $G=\Omega_d^\pm(q)$, $d\ge 4$, $q\ge 4$ in its natural action on the non-singular $1$-dimensional subspaces. In this case, all bases have cardinality $d-1$.
\end{enumerate}
\end{theorem}
While proving Theorem~\ref{thrm:main2}, we have delved deeply into the problem of classifying almost simple primitive IBIS groups. As a result, in the future, we intend to revisit this issue and examine the case of groups having base size at most 5. Indeed, in our opinion, there is some hope
in dealing with groups having small base size. For instance, suppose that $G$ is an almost simple primitive IBIS group with $b(G) = 2$. In
particular, all irredundant bases of $G$ have cardinality $2$. Thus, for any two distinct points $\alpha$ and $\beta$ in the domain of $G$
we have $G_\alpha\cap G_\beta = 1$.  Therefore $G$ is
a Frobenius group, contradicting the fact that $G$ is almost simple. This simple argument shows that, if $G$ is an almost
simple primitive IBIS group, then $b(G) \ge 3$. Therefore, from Theorem~\ref{thrm:main}, the cases that need to be considered are when $b(G)\in \{3,4,5\}$.\footnote{Most examples in Table~\ref{table:table} have $b(G)\in \{3,4\}$.}

\begin{table}[!h]

\begin{tabular}{c|c|c|c|c|c}
Group&Degree&$b(G)$&s. $k$-t.&g.&Comments\\\hline
$\mathrm{Alt}(n)$&$n$&$n-2$&$\checkmark$&$\checkmark$&$n\ge 5$\\
$\mathrm{Sym}(n)$&$n$&$n-1$&$\checkmark$&$\checkmark$&$n\ge 5$\\\hline
%$\mathrm{PSL}_2(q)$&$q+1$&$3$&&&\\
$\mathrm{PGL}_2(q)$&$q+1$&$3$&$\checkmark$&$\checkmark$&\\
$\mathrm{PSL}_2(q)$&$q+1$&$3$&&& $q$ odd\\
$\mathrm{PSL}_2(q)\rtimes\langle\tau\rangle$&$q+1$&$4$&&& $q$ odd, $\tau$ field automorphism\\
&&&&& of prime order\\
$\mathrm{PGL}_2(q)\rtimes\langle\tau\rangle$&$q+1$&$4$&&& $q>4$, $\tau$ field automorphism\\
&&&&& of prime order\\
$\mathrm{Aut}(\mathrm{PSL}_2(4))$&$5$&$4$&$\checkmark$&$\checkmark$&\\
$\mathrm{PSL}_2(q).\langle\tau\rangle$&$q+1$&$3$&$\checkmark$&$\checkmark$& $q=p^f$ odd, $f$ even, $\tau$ not\\
&&&&&field automorphism, $G\ne\mathrm{PGL}_2(q)$\\
\hline

$M_{11}$&$11$&$4$&$\checkmark$&$\checkmark$&\\
$M_{12}$&$12$&$5$&$\checkmark$&$\checkmark$&\\
$M_{22}$&$22$&$5$&&&\\
$M_{23}$&$23$&$6$&&&\\
$M_{24}$&$24$&$7$&&&\\\hline
$\mathrm{SL}_d(2)$&$2^d-1$&$d$&&$\checkmark$& $d\ge 3$\\\hline
$\mathrm{Sp}_d(2)$&$2^d-1$&$d$&&& $d\ge 4$\\
$\mathrm{Sp}_4(2)'\cong \mathrm{Alt}(6)$&$15$&$3$&&&\\\hline

${}^2B_2(q)$&$q^2+1$&$3$&&& $q=2^{2f+1}$, $f\ge 1$\\\hline
${}^2G_2(q)$&$q^3+1$&$3$&&& $q=3^{2f+1}$, $f\ge 1$\\
${}^2G_2(3)'\cong\mathrm{PSL}_2(8)$&$28$&$3$&&&\\\hline
$\mathrm{SL}_2(q)$&$\frac{(q-1)q}{2}$&$3$&&& $q=2^f$, $f\ge 3$\\
$\mathrm{Aut}(\mathrm{SL}_2(q))$&$\frac{(q-1)q}{2}$&$3$&&& $q=2^f$, $f\ge 3$ prime \\
$\mathrm{SL}_2(4)$&$6$&$3$&&$\checkmark$&\\\hline
$\mathrm{P}\Gamma\mathrm{L}_2(q)$&$\frac{(q+1)q}{2}$&$3$&&& $q=2^p$, $p$ odd prime\\\hline
$\mathrm{SL}_2(q)$&$\sqrt{q}(q+1)$&3&&& $q=2^f$, $f\ge 4$, $f$ even\\\hline
$\mathrm{Alt}(7)$&$15$&$3$&&$\checkmark$&\\\hline

$\Omega_d^\pm(q)$&$q^{\frac{d}{2}-1}(q^{\frac{d}{2}}\mp1)$&$d-1$&&&$q\ge 4$, $q$ even, $n\ge 4$\\
&&&&&action on non-singular $1$-spaces
\end{tabular}
\caption{Examples of almost simple primitive $\mathrm{IBIS}$ groups, with the sharply $k$-transitive and	geometric actions highlighted}\label{table:table}
\end{table}

\subsection{Standard and non-standard actions}\label{largebasesize}
To start with, we use well-known
results bounding $b(G)$ to deal with the so-called \textit{\textbf{non-standard actions}}. The terminology
below follows~\cite[Definition~2.1]{Bur07}.\footnote{Originally, this notation was introduced by Liebeck and Shalev~\cite{LS99}.}

\begin{definition}\label{defintion6.2}
{\rm
 Let $G$ be an almost simple classical group over $\mathbb{F}_q$, where $q=p^f$ and $p$ is prime, with socle $G_0$ and associated natural
module $V$. A subgroup $H$ of $G$ not containing $G_0$ is a \textit{\textbf{subspace subgroup}} if, for each maximal
subgroup $M$ of $G_0$ containing $H \cap G_0$, one of the following holds:
\begin{enumerate} 
\item\label{definition6.2:1} $M$ is the stabilizer in $G_0$ of a proper nonzero subspace $U$ of $V$, where $U$ is totally
singular, non-degenerate, or, if $G_0$ is orthogonal and $p = 2$, a non-singular $1$-space ($U$ can be any subspace if $G_0 = \mathrm{PSL}(V )$).
\item\label{definition6.2:2} $G_0 = \mathrm{Sp}_{2m} (q)'$, $p = 2$, and 
$M = {\mathrm O}^\pm_{2m} (q)$.
\end{enumerate}
}
\end{definition}
A \textbf{\textit{subspace action}} of the classical group $G$ is the action of $G$ on the coset space $[G : H]$,
where $H$ is a subspace subgroup of $G$.
Note that the definition above amounts precisely to this: let $G$ be an almost simple group whose socle is a classical group $G_0$ with natural module $V$. If $G_0 = \mathrm{PSL}_n(p^f)$ or $G \leq \mathrm{PGL}(V)$, then a maximal subgroup of $G$ is a subspace subgroup if and only if it is contained in the $\mathcal{C}_1$ class or, when $G_0 = \mathrm{Sp}_n(2^f)$, the $\mathcal{C}_8$ class (using the labelling of Aschbacher classes $\mathcal{C}_1$–$\mathcal{C}_8$ in \cite{KL90}).
Otherwise, $G_0 = \mathrm{Sp}_4(2^f)$ or $G_0 = \mathrm{P}\Omega^+_8(p^f)$, and $G_0$ has some exotic automorphisms giving rise to further maximal subgroups, some of which are subspace subgroups.\\
We note that~\cite{KL90} explicitly excludes these cases, but they are included in~\cite{bhr}.

\begin{definition}\label{definition6.3}
{\rm
 A transitive action of $G$ on a set $\Omega$ is said to be standard if, up to
equivalence of actions, one of the following holds:
\begin{enumerate}
\item $G_0 = \mathrm{Alt}(m)$ and $\Omega$ is an orbit of subsets or uniform partitions of $\{1, \ldots , m\}$.
\item $G$ is a classical group in a subspace action.
\end{enumerate}
}
\end{definition}
For an almost simple primitive permutation group in a non-standard action, the base
size is bounded by an absolute constant. This was conjectured by Cameron and Kantor
(see~\cite{Cam92,cameron}) and then settled in the aﬃrmative by Liebeck and Shalev in~\cite[Theorem~1.3]{LS99}. The constant was then made explicit in subsequent works~\cite{Bur07,Bur2018,BGS11,BLS09,BOW10}. The following theorem summarizes these results.

\begin{theorem}\label{thrm:6.4} Let $G$ be a finite almost simple group in a primitive faithful non-standard
action with socle $G_0$. Then, $b(G) \le 7$, with equality if and only if $G$ is the Mathieu group $M_{24}$ in its natural action of degree $24$. Moreover, $b(G)=6$ if and only if one of the following holds:
\begin{enumerate}
\item $G=M_{23}$ in its natural action of degree $23$,
\item $G=Co_3$ in its action on the right cosets of a maximal subgroup isomorphic to $McL.2$,
\item $G=Co_2$ in its action on the right cosets of a maximal subgroup isomorphic to $\mathrm{U}_6(2).2$,
\item $G=Fi_{22}.2$ in its action on the right cosets of a maximal subgroup isomorphic to $2.\mathrm{U}_6(2).2$,
\item\label{part1} $G_0=E_7(q)$ in its action on the cosets of a parabolic subgroup labeled $P_7$,
\item\label{part2} $G_0=E_6(q)$ in its action on the cosets of a parabolic subgroup labeled $P_1$ or $P_6$.
\end{enumerate}

Moreover, if $G_0$ is a classical group, then either $b(G)\le 4$ or $G=\mathrm{U}_6(2).2$ with stabilizer $\mathrm{U}_4(3).2^2.$
\end{theorem}
The labeling of the parabolic subgroups in part~\eqref{part1} and~\eqref{part2} follows~\cite{Bourbaki}.
\subsection{Structure of the paper}\label{structure}
Let $G$ be an almost simple primitive group having socle a simple group of Lie type $G_0$. When $G_0$ is classical, let $V$ be the natural module of the covering group of $G_0$ and let $p$ be the characteristic of the underlying field.  From Definition~\ref{defintion6.2}\footnote{See also the comment following Definition~\ref{defintion6.2} and~\cite[page~549]{Bur07}.} and Theorem~\ref{thrm:6.4}, the proof of Theorem~\ref{thrm:main} reduces to the following cases:
\begin{enumerate}
\item\label{HiHo1} $G_0$ is a classical group acting on totally singular subspaces of $V$,
\item\label{HiHo3+}$G_0=\mathrm{PSL}_d(q)$ acting on pairs of subspaces of $V$, see~\eqref{psk1} and~\eqref{psk2} in Section~\ref{psll} for more details,
\item\label{HiHo4} $G_0=\mathrm{Sp}_{2m}(q)'$, $q$ is even and $G_0$ is acting on  the right cosets of $\mathrm{O}_{2m}^\pm(q)$,
\item\label{HiHo2} $G_0$ is classical acting on non-degenerate subspaces of $V$,
\item\label{HiHo3} $G_0$ is orthogonal of even characteristic and $G_0$ is acting on non-singular $1$-subspaces of $V$,

\item\label{HiHo5} $G_0=\mathrm{Sp}_4(2^a)$, $G \not\leq \mathrm{P}\Gamma \mathrm{L}(V)$, where $V$ is the natural module of $G_0$, and $G_0$ is acting on the cosets of a local subgroup of type $[q^4]:C_{q-1}^2$, see~\cite[Table~8.14]{bhr},
\item\label{HiHo6} $G_0=\mathrm{P}\Omega_8^+(q)$,  $G \not\leq \mathrm{P}\Gamma \mathrm{L}(V)$, where $V$ is the natural module of $G_0$, and $G_0$ is acting on the cosets of a local subgroup of type $$[q^{11}]:\left[\frac{q-1}{d}\right]^2\cdot\frac{1}{d}\mathrm{GL}_2(q)\cdot d^2,$$where $d=\gcd(2,q-1)$, see~\cite[Table~8.50]{bhr},
\item\label{HiHo7} $G_0=\mathrm{P}\Omega_8^+(q)$,  $G \not\leq \mathrm{P}\Gamma \mathrm{L}(V)$, where $V$ is the natural module of $G_0$, and $G_0$ is acting on the cosets of a subgroup isomorphic to $G_2(q),$ see~\cite[Table~8.50]{bhr},
\item\label{HiHo8} $G_0=E_7(q)$ in its action on the cosets of a parabolic subgroup labeled $P_7$,
\item\label{HiHo9} $G_0=E_6(q)$ in its action on the cosets of a parabolic subgroup labeled $P_1$ or $P_6$.
\end{enumerate}
The structure of the paper is straightforward.
In Section~\ref{preliminaries} we collect some basic remarks and observations. The proof of  Theorem~\ref{thrm:main} is divided by cases \eqref{HiHo1} - \eqref{HiHo9} as follow.\\
We prove case~\eqref{HiHo1} and case~\eqref{HiHo3+} in Section~\ref{psl}.
We prove case~\eqref{HiHo4} in Section~\ref{symplecticeven}. 
We prove case~\eqref{HiHo2} in Section~\ref{nondegenerate}.
We prove case~\eqref{HiHo3} in Section~\ref{1non-singular}.
We prove case~\eqref{HiHo5} in Section~\ref{Sp4graph}.
We prove case~\eqref{HiHo6} and~\eqref{HiHo7} in Section~\ref{trialies}.
We prove cases~\eqref{HiHo8} and~\eqref{HiHo9} in Section~\ref{e6e7}.\\
In the final section of the paper, we present the $\mathrm{GAP}$ (\cite{GAP}) code used for all computations. Specifically, we provide the functions along with two examples demonstrating their applications. The first example uses the permutation representation of groups available in $\mathrm{GAP}$. In the second example, where the permutation representation is too large, we use the FinInG package (\cite{FINING}) for computations involving classical groups.
\section{Preliminaries}\label{preliminaries}
Let $G$ be a permutation group on $\Omega$. Given $\omega_1,\ldots,\omega_\ell\in \Omega$, we say that the sequence $(\omega_1,\ldots,\omega_\ell)$ is \textit{\textbf{irredundant}} if
$$G>G_{\omega_1}>G_{\omega_1,\omega_2}>\cdots>G_{\omega_1,\omega_2,\ldots,\omega_\ell}.$$

\begin{lemma}\label{lemma:1}
Let $G$ be a permutation group on $\Omega$ and let  $(\omega_1,\ldots,\omega_\ell)$ and $(\omega_1',\ldots,\omega_{\ell'}')$ be irredundant  with 
$$G_{\omega_1,\ldots,\omega_\ell}=G_{\omega_1',\ldots,\omega_{\ell'}'}.$$ 
%If $G$ is $\mathrm{IBIS}$, then $\ell=\ell'$.
if $\ell \neq \ell'$, then $G$ is not $\mathrm{IBIS}$.
\end{lemma}
\begin{proof}
As $G_{\omega_1,\ldots,\omega_\ell}=G_{\omega_1',\ldots,\omega_{\ell'}'}$, we may extend the irredundant sequences $(\omega_1,\ldots,\omega_\ell)$ and $(\omega_1',\ldots,\omega_{\ell'}')$   using the same set of points to obtain irredundant bases for $G$.
\end{proof}

The following  results are just simple observations.
\begin{lemma}\label{lemma:minus1}
Let $G$ be a permutation group on $\Omega$ and let $H$ be a subgroup of $G$. If   $(\omega_1,\ldots,\omega_\ell)$  is irredundant for $H$, then it is  irredundant   for $G$.
\end{lemma}
\begin{proof}
	Let $H_i = H_{\omega_1,\dots,\omega_i}$ and $G_i = G_{\omega_1,\dots,\omega_i}$, for $i=1,2,\dots,\ell$. Then, $H_i = G_i \cap H$. Since $H_1 > H_2 > \dots > H_\ell$, we deduce $G_1 > G_2 > \dots >G_\ell$.
\end{proof}

\begin{lemma}\label{lemma:minusminus1}
Let $A,G$ and $B$ be  permutation groups on $\Omega$ with $A\le G\le B$. If $A$ has an irredundant base of length $\ell_A$ and $B$ has an irredundant base of length $\ell_B$ with $\ell_A>\ell_B$, then  $G$ is not $\mathrm{IBIS}$.
\end{lemma}
\begin{proof}
	Let $(\omega_1,\dots,\omega_{\ell_B})$ be an irredundant base for $B$. Observe that this is a base for $G$ and hence we may extract from it an irredundant base for $G$ of length at most $\ell_B$. 
	Let $(\omega_1,\dots,\omega_{\ell_A})$ be an irredundant base for $A$. Observe that this an irredundant sequence for $G$ and hence we may extend it to an irredundant base for $G$ of length at least $\ell_A$. 
\end{proof}
\section{Primitive action on totally-singular subspaces: case~\eqref{HiHo1} and case~\eqref{HiHo3+}}\label{psl}
We divide the proof of Theorem~\ref{thrm:main} for case~\eqref{HiHo1} depending on whether $G_0$ is linear, unitary, symplectic or orthogonal. There are strong similarities in the arguments we use to deal with each case. However, each family has its own peculiar behavior and hence we have preferred to divide the proofs in separate cases.

In this and in later sections, we denote the projective image of a matrix with square brackets.

\subsection{Linear groups}\label{psll}
Let $d\ge 2$ be a positive integer, let $q$ be a prime power (that is, $q=p^f$ for some prime number $p$ and some positive integer $f$), let $G_0=\mathrm{PSL}_d(q)$ and let $G$ be an almost simple group with socle $G_0$. In particular, $G_0\unlhd G\le \mathrm{Aut}(G_0)$, where 
$\mathrm{Aut}(G_0)=\mathrm{P}\Gamma\mathrm{L}_d(q)$ when $d=2$ and $$\mathrm{Aut}(G_0)=\mathrm{P}\Gamma\mathrm{L}_d(q)\rtimes\langle\iota\rangle=\mathrm{PGL}_d(q)\rtimes(\mathrm{Aut}(\mathbb{F}_q)\times\langle\iota\rangle),$$ when $d>2$, where $\iota$ is a graph automorphism of $G_0$ and $\mathrm{Aut}(\mathbb{F}_q)$ is the group of field automorphisms of $\mathbb{F}_q$.

We now describe the primitive subspace actions.
 We let $V=\mathbb{F}_q^d$ be the $d$-dimensional vector space over the finite field $\mathbb{F}_q$ of cardinality $q$ and let $k\in \{1,\ldots,d-1\}$. Then,
\begin{enumerate}
\item\label{pslk} we let $\Omega_k$ be the collection of all $k$-dimensional subspaces of $V$ for $1\leq k\leq d/2$, here $\mathrm{PSL}_d(q)\unlhd G\le\mathrm{P}\Gamma\mathrm{L}_d(q)$,
\item\label{psk1}for $1\le k<d/2$, we let $\Omega_k^1=\{\{W,U\}\mid W\in \Omega_k, U\in \Omega_{d-k}, V=W\oplus U\}$, here $\mathrm{PSL}_d(q)\unlhd G\nleq\mathrm{P}\Gamma\mathrm{L}_d(q)$,
\item\label{psk2}for $1\le k<d/2$, we let $\Omega_k^2=\{\{W,U\}\in W\in \Omega_k, U\in \Omega_{d-k}, W\le U\}$, here $\mathrm{PSL}_d(q)\unlhd G\nleq\mathrm{P}\Gamma\mathrm{L}_d(q)$.
\end{enumerate}

Since our argument is inductive, we need to discuss in detail the cases when $d$ is small\footnote{Most of the results for linear groups can be deduced from~\cite{KRD}. However, we prefer to give some independent arguments, since they highlight some general constructions in arbitrary subspace actions for classical groups.}. The case $d=2$ is quite special and it was already discussed in~\cite{LeSp}. Here, we recall some basic facts from~\cite[Example~3.1]{LeSp}. Let $G$ be an almost simple group with socle 
$\mathrm{PSL}_2(q)$ acting on the $q+1$ points of the projective line. In particular, $$\mathrm{PSL}_2(q)\unlhd G\le \mathrm{P}\Gamma\mathrm{L}_2(q).$$ Set
$r=|G:G\cap\mathrm{PGL}_2(q)|$. Recall that $q=p^f$.
\begin{lemma}[{{Example~3.1,~\cite{LeSp}}}]\label{lemma:-100}
The group $G$ is $\mathrm{IBIS}$ if and only if
$r=1$, or $r$ is prime and $r\mid f$.
\end{lemma}

Suppose now $d\ge 3$. Let $V=\mathbb{F}_q^d$ and let $e_1,\ldots,e_d$ be an $\mathbb{F}_q$-basis of $V$. 

\subsubsection{Case~\eqref{pslk}}
We first discuss the case $k=1$.
\begin{lemma}\label{prop:2}
Let $G$ be a group as in~\eqref{pslk} acting on $\Omega_k$ with $k=1$. Then $G$ is $\mathrm{IBIS}$ if and only if $q=2$.
\end{lemma}
\begin{proof}
When $q=2$, $G=\mathrm{GL}_d(2)$ and $G$ is the group described in~\cite[Example~3.3]{LeSp}; hence $G$ is IBIS. Therefore, we may suppose that $q>2$.

Let $\omega_1=\omega_1'=\langle e_1\rangle$, $\omega_2=\omega_2'=\langle e_2\rangle$, $\omega_3=\omega_4'=\langle e_3\rangle$, $\omega_3'=\langle e_1+e_2\rangle$, $\omega_4=\omega_5'=\langle e_1+e_2+e_3\rangle$.  As $q>2$, with a simple computation, we see that $(\omega_1,\ldots,\omega_4)$ and $(\omega_1',\ldots,\omega_5')$ are irredundant chains for $G_0=\mathrm{PSL}_d(q)$. Indeed, $$G_{\omega_1,\omega_2} = G_{\omega_1',\omega_2'}$$ consists of (the image of) matrices where the first two rows are zero except for the diagonal elements. By stabilizing $\omega_3'$, we force the first two diagonal elements to be equal. Subsequently, stabilizing $\omega_4'$ and $\omega_5'$, we first ensure that the third row is zero except for its diagonal element, and then we impose that the third diagonal element equals the first two.
For the first base, we essentially follow the same approach, but we enforce the equality of all three diagonal elements simultaneously by stabilizing $\omega_4$. Therefore, by Lemma~\ref{lemma:minus1}, these  are irredundant chains for $G$. 

We show that
$$G_{\omega_1,\ldots,\omega_4}=G_{\omega_1',\ldots,\omega_5'},$$
from which the proof immediately follows from Lemma~\ref{lemma:1}. 

As $\{\omega_1,\ldots,\omega_4\}\subseteq\{\omega_1',\ldots,\omega_5'\}$, we deduce $G_{\omega_1,\ldots,\omega_4}\ge G_{\omega_1',\ldots,\omega_5'}.$ Thus, let $g\in G_{\omega_1,\ldots,\omega_4}$. In particular, there exists a linear transformation $x\in\mathrm{GL}(V)$ and a Galois automorphism $\sigma\in \mathrm{Aut}(\mathbb{F}_q)$ such that $g$ is  the permutation induced by $\sigma x$ on $\Omega_1$. Thus
\begin{align*}
e_1^\sigma x&=e_1x\in \langle e_1\rangle=\omega_1,\\
e_2^\sigma x&=e_2x\in \langle e_2\rangle=\omega_2,\\
e_3^\sigma x&=e_3x\in \langle e_3\rangle=\omega_3,\\
(e_1+e_2+e_3)^\sigma x&=(e_1+e_2+e_3)x\in \langle e_1+e_2+e_3\rangle=\omega_4.
\end{align*}
In particular, there exist $\lambda_1,\lambda_2,\lambda_3,\lambda_4\in \mathbb{F}_q$ with
\begin{align*}
e_1x=&\lambda_1 e_1,\\
e_2x=&\lambda_2 e_2,\\
e_3x=&\lambda_3 e_3,\\
(e_1+e_2+e_3)x=&\lambda_4 (e_1+e_2+e_3).
\end{align*} This linear system of equations gives $\lambda_1=\lambda_2=\lambda_3=\lambda_4$. This shows that $x$ acts as a scalar on the subspace $W=\langle e_1,e_2,e_3\rangle$ of $V$. Since $\omega_1',\ldots,\omega_5'$ are $1$-dimensional subspaces of $W$ and since $e_1,e_2,e_1+e_2,e_3,e_1+e_2+e_3$ have coefficients in the prime field $\mathbb{F}_p$, we deduce $g\in G_{\omega_1',\ldots,\omega_5'}$.
\end{proof}

\begin{lemma}\label{prop:2bis}
Let $G=\mathrm{GL}_d(2)$ be acting on $\Omega_k$ with $k=2$. Then $G$ is not $\mathrm{IBIS}$.
\end{lemma}
\begin{proof}
As $k=2$, we have $d\ge 4$. We let
\begin{align*}
\omega_1&=\langle e_1,e_2\rangle,\, \omega_2=\langle e_3,e_4\rangle,\, \omega_3=\langle e_1+e_3,e_2+e_4\rangle,\,
\omega_4=\langle e_1,e_3\rangle,\, \omega_5=\langle e_2,e_4\rangle,\\
\omega_1'&=\langle e_1,e_2\rangle,\, \omega_2'=\langle e_1,e_3\rangle,\,
\omega_3'=\langle e_2,e_4\rangle,\,
\omega_4'=\langle e_3,e_4\rangle.
\end{align*} 
A computation shows that $(\omega_1,\omega_2,\omega_3,\omega_4,\omega_5)$ and $(\omega_1',\omega_2',\omega_3',\omega_4')$ are irredundant chains reaching the same stabilizer. In particular, their stabilizer is the subgroup of $\mathrm{GL}_d(2)$ of matrices of the form
\[
	\begin{bmatrix}
		I_4 & \underline{0} \\
		A & B
	\end{bmatrix},
\]
where $I_4$ is the $4 \times 4$ identity matrix, $\underline{0}$ is the $4 \times (d-4)$ zero matrix, $A$ is a $(d-4)\times 4$ matrix and $B \in \mathrm{GL}_{d-4}(2)$. Using Lemma~\ref{lemma:1}, we conclude that $G$ is not $\mathrm{IBIS}$.
\end{proof}

The action of $G=\mathrm{GL}_4(2)$ on the $2$-subspaces of $V=\mathbb{F}_2^4$ is actually very interesting; Lemma~\ref{prop:2bis} shows that $G$ has irredundant bases of size $4$ and $5$, however, it can be verified with the auxiliary help of a computer that every minimal base\footnote{Recall that a minimal base is a base for $G$, for which none of its proper subsets is again a base.} of $G$ has cardinality $4$. In~\cite{DMS}, these permutation groups were christened with MiBIS, from \textit{\textbf{Minimal Bases of Invariant Sizes}}. 

Using Lemmas~\ref{prop:2} and~\ref{prop:2bis} we can now deal with the case $k\ge 2$.
\begin{lemma}\label{prop:3}
Let $G$ be a group as in~\eqref{pslk} acting on $\Omega_k$ with $k\ge 2$. Then $G$ is not $\mathrm{IBIS}$.
\end{lemma}
\begin{proof}
When $q>2$, let $W=\langle e_1,\ldots,e_{k-1}\rangle$ and, when $q=2$, let $W=\langle e_1,\ldots,e_{k-2}\rangle$. Now, let $\Delta=\{\omega\in \Omega_k\mid W\le \omega\}$, let $G_W$ be the stabilizer of $W$ in $G$, let $G_{\Delta}$ be the setwise stabilizer of $\Delta$ in $G$ and let $H$ be the permutation group induced by $G_W$ on its action on $\Delta$. We claim that
\begin{equation}\label{equation:basta}G_W=G_{\Delta}.
\end{equation}
Indeed, if $g\in G_W$, then $g$ fixes setwise $W$ and hence $g$ permutes the $k$-dimensional subspaces of $V$ containing $W$, that is, $g$ fixes setwise $\Delta$. Thus $g\in G_\Delta$. Conversely, if $g\in G_\Delta$, then $g$ fixes setwise
$$\bigcap_{\delta\in \Delta}\delta=W$$
and hence $g\in G_W$.

Observe that when $q=2$, from Lemma~\ref{prop:2bis}, we may assume that $k\ge 3$.

When $q>2$, $H$ is an almost simple group with socle $\mathrm{PSL}_{d-k+1}(q)$ and the action of $H$ on $\Delta$ is permutation isomorphic to the natural action of $H$ on the $1$-dimensional subspaces of $V/W\cong\mathbb{F}_q^{d-k+1}$; moreover, $d-k+1\ge 3$ because $2\le k\le d/2$. When $q=2$,  $H=\mathrm{GL}_{d-k+2}(2)$ and the action of $H$ on $\Delta$ is permutation isomorphic to the natural action of $H$ on the $2$-dimensional subspaces of $V/W\cong\mathbb{F}_q^{d-k+2}$; moreover, $d-k+2\ge 4$ because $3\le k\le d/2$.

Thus, from Lemmas~\ref{prop:2} and~\ref{prop:2bis}, the action of $H$ on $\Delta$ is not IBIS. Therefore, in its action on $\Delta$ the group $H$ has two irredundant bases $(\omega_1,\ldots,\omega_\ell)$ and $(\omega_1',\ldots,\omega_{\ell'}')$ with $\ell'\ne \ell$. We claim that
\begin{equation}\label{eq:number1}
\bigcap_{i=1}^\ell\omega_i=W=\bigcap_{i=1}^{\ell'}\omega_i'.
\end{equation}
We only show that $W=\bigcap_{i=1}^\ell\omega_i$, because the argument for showing $W=\bigcap_{i=1}^{\ell'}\omega_i'$ is similar. As $\omega_i\in \Delta$ for each $i$ and as $W\le \omega_i$, we have $W\le\bigcap_{i=1}^\ell\omega_i$. Suppose first $q>2$. Since $\dim_{\mathbb{F}_q}W=k-1$ and $\dim_{\mathbb{F}_q}\omega_i=k$, we have $W\ne \bigcap_{i=1}^\ell\omega_i$ only if $\ell=1$. However, when $\ell=1$, we get a contradiction because $\mathrm{PSL}_{d-k+1}(q)$ in its action on the projective points cannot have a base of size $\ell=1$, and hence neither can $H$. Suppose now $q=2$. Let $W'=\bigcap_{i=1}^\ell \omega_i$. Since $\dim_{\mathbb{F}_q}W=k-2$ and $\dim_{\mathbb{F}_q}\omega_i=k$, we have $W\ne W'$ only when $\dim_{\mathbb{F}_q}W'=k$ or $\dim_{\mathbb{F}_q}W'=k-1$. In the first case, we have $\ell=1$ and  we get a contradiction because $H=\mathrm{GL}_{d-k+2}(2)$ in its action on the $2$-dimensional subspaces of $\mathbb{F}_q^{d-k+2}$ cannot have a base of size $\ell=1$. Assume then $\mathrm{dim}_{\mathbb{F}_q}W'=k-1$. Without loss of generality, replacing the standard $\mathbb{F}_q$-basis of $V$ if necessary, we have $W'=\langle W,e_{k-1}\rangle$. Let $\omega=\langle W,e_{k},e_{k+1}\rangle$ and let $x\in\mathrm{GL}_d(q)$ with $e_{k}x=e_{k-1}+e_{k}$ and $e_ix=e_i$, for all $i\ne k$. Since $W\le \omega$ and $\dim_{\mathbb{F}_q}\omega=k$, we have $\omega\in \Delta$. Since $x$ acts trivially on $V/W'$ and since $W'\le \omega_i$ for each $i$, we have $x\in G_{\omega_1,\ldots,\omega_\ell}$. Since $\omega_1,\ldots,\omega_\ell$ is an irredundant base for the action of $H$ on the $2$-dimensional subspaces of $V/W$, we deduce that $x$ fixes each $2$-dimensional subspace of $V/W$. However, $$\omega^x=\langle Wx,e_{k}x,e_{k+1}x\rangle=\langle W,e_{k-1}+e_{k},e_{k+1}\rangle\ne\langle e_1,\ldots,e_{k-2},e_k,e_{k+1}\rangle=\omega,$$
which is a contradiction. This concludes the proof of~\eqref{eq:number1}.

From~\eqref{equation:basta} and~\eqref{eq:number1}, we have $G_{\omega_1,\ldots,\omega_\ell}\le G_\Delta$ and $G_{\omega_1',\ldots,\omega_{\ell'}'}\le G_\Delta$. Moreover, since $\omega_1,\ldots,\omega_\ell$ and $\omega_1',\ldots,\omega_{\ell'}'$ are irredundant bases for  $H$, we get
 $$G_{\omega_1,\ldots,\omega_\ell}=G_{(\Delta)}=G_{\omega_1',\ldots,\omega_{\ell'}'},$$
where as usual $G_{(\Delta)}$ is the pointwise stabiliser of $\Delta$.  Therefore, by Lemma~\ref{lemma:1}, $G$ is not IBIS in its action on $\Omega_k$.
\end{proof}

\subsubsection{Case~\eqref{psk1}}\label{sec:psk1}
We introduce some notation that we will use in this section and in Section~\ref{sec:psk2}.

We begin by recalling the concept of polarity on the projective linear space. Let $\mathrm{PG}_{d-1}(q)$ denote the $(d-1)$-dimensional projective space over $V$. The elements of $\mathrm{PG}_{d-1}(q)$ are the subspaces of $V$. The $1$-dimensional subspaces are referred to as points, the $2$-dimensional subspaces as lines, and the $(d-1)$-dimensional subspaces as hyperplanes. Clearly, the action of $\Gamma\mathrm{L}_d(q)$ on $V$ induces a faithful action of $\mathrm{P}\Gamma\mathrm{L}_d(q)$ on $\mathrm{PG}_{d-1}(q)$. The projective space $\mathrm{PG}_{d-1}(q)$ is an incidence structure: given two subspaces $W$ and $W'$ of $V$, we say that $W$ is incident with $W'$ if $W \le W'$.

Now, $\mathrm{PG}_{d-1}(q)$ is endowed with a natural polarity $\iota$, which is an involutory mapping $\iota: \mathrm{PG}_{d-1}(q) \to \mathrm{PG}_{d-1}(q)$ that maps $k$-dimensional subspaces to $(d-k)$-dimensional subspaces and reverses the inclusions. To define this polarity, let $e_1, \ldots, e_d$ be the canonical basis of $V = \mathbb{F}_q^d$.

This basis allows us to parametrize the elements of $V$ as $d$-tuples of elements of $\mathbb{F}_q$. Given $(x_1, \ldots, x_d) \in V \setminus \{0\}$, we denote
$$\left[x_1, \ldots, x_d\right]$$
as the $1$-dimensional subspace spanned by $(x_1, \ldots, x_d)$. The elements of $\mathrm{PGL}_d(q)$ act naturally on the right. In particular, given $x \in \mathrm{GL}_d(q)$ and $W \le V$, the matrix $x$ acts on $\mathrm{PG}_{d-1}(q)$ by setting $W \mapsto Wx$, where $Wx$ denotes the image of $W$ under $x$.

This basis also allows us to parametrize the hyperplanes of $V$, and thus, by taking intersections, it allows us to describe arbitrary subspaces of $V$. Specifically, we denote a hyperplane $U$ of $\mathrm{PG}_{d-1}(q)$ as $[x_1,\dots,x_d]^T$, where $U = \ker (x_1,\dots,x_d)^T$. The polarity $\iota$ is defined by
$$[x_1, \ldots, x_d] \mapsto \left[
\begin{array}{c}
	x_1 \\
	\vdots \\
	x_d
\end{array}
\right] \quad \text{and} \quad \left[
\begin{array}{c}
	x_1 \\
	\vdots \\
	x_d
\end{array}
\right] \mapsto [x_1, \ldots, x_d].$$
This allows us to extend $\iota$ to all elements in the projective space $\mathrm{PG}_{d-1}(q)$.

Note that if we denote the elements of $V$ as row vectors and represent the corresponding hyperplane by $\langle v^T \rangle$ (thinking of $v^T$ as a column vector), then $w \in \langle v^T \rangle$ if and only if $w \cdot v^T = 0$, where $\cdot$ denotes the standard scalar product. The elements of $\mathrm{PGL}_d(q)$ act naturally on the right on the hyperplanes. In particular, given $x \in \mathrm{GL}_d(q)$ and $w \in V$, the matrix $x$ maps $\langle w^T \rangle$ to $\langle x^{-1} w^T \rangle$. It is important to note that the action we have just defined on hyperplanes coincides with the action defined in the previous paragraph and hence there is no ambiguity between the two definitions.

The mapping $\iota$ is inclusion-reversing, interchanging intersections with spans, and preserving incidence.

Finally, the action of $\iota$ on $\mathrm{PG}_{d-1}(q)$ is compatible with the action of $\mathrm{P}\Gamma\mathrm{L}_d(q)$. We have
$$v^{\iota x \iota} = (v^T)^{x \iota} = (x^{-1} v^T)^\iota = v(x^{-1})^T,$$
where $x^T$ denotes the transpose of the matrix $x$. Since this holds for every $v \in V$, we conclude that $\iota^2 = 1$ and $\iota x \iota = (x^{-1})^T$. Thus, $\iota$ is a graph automorphism of $\mathrm{PGL}_d(q)$.

\begin{lemma}\label{prop:3bispart1}
Let $G$ be a group as in~\eqref{psk1} acting on $\Omega_k^1$. Then $G$ is not $\mathrm{IBIS}$.
\end{lemma}
\begin{proof}
Let $U=\langle e_4,\ldots,e_{k+2}\rangle$ and let $W=\langle e_{k+3},\ldots,e_d\rangle$. There is a slight abuse of notation here: when $d=3$, $U=W=0$; when $k=1$, $U=0$.

Let
\begin{align*}
\omega_1 &= \omega_1' = \{\langle e_1,U \rangle, \langle e_2,e_3,W\rangle\}, \\
 \omega_2 &= \omega_2' = \{\langle e_2,U \rangle, \langle e_1, e_3,W\rangle\}, \\
 \omega_3' &= \{\langle e_1+e_2,U\rangle, \langle e_1,e_3,W\rangle \}, \\
 \omega_3 &= \omega_4' = \{ \langle e_1+e_2+e_3,U \rangle, \langle e_1, e_2,W\rangle\}.
\end{align*}
When $q\ne 2 $ and when $(d,q)\ne (3,4)$, a simple computation shows that $(\omega_1,\omega_2,\omega_3)$ and $(\omega_1',\omega_2',\omega_3',\omega_4')$ are irredundant chains for $\mathrm{PSL}_d(q)$\footnote{The condition $q\ne 2$ guarantees that $(\omega_1,\omega_2,\omega_3)$ is an irredundant chain and the condition $(d,q)\ne (3,4)$ guarantees that $(\omega_1',\omega_2',\omega_3',\omega_4')$ is also an irredundant chain.}. Moreover,
$$
\mathrm{PSL}_d(q)_{\omega_1,\omega_2,\omega_3}=
\mathrm{PSL}_d(q)_{\omega_1',\omega_2',\omega_3',\omega_4'}=
\left\{
\left[
\begin{array}{ccc}
I_3&0&0\\
0&X&0\\
0&0&Y
\end{array}
\right]\mid 
\begin{array}{c}
X\in\mathrm{GL}_{k-1}(q),\\Y\in\mathrm{GL}_{d-k-2}(q),\\ \det(X)\det(Y)=1 
\end{array}
\right\},
$$
where $I_3$ is the $3\times 3$ identity matrix. Therefore, by Lemma \ref{lemma:minus1}, $(\omega_1,\omega_2,\omega_3)$ and $(\omega_1',\omega_2',\omega_3',\omega_4')$ are irredundant chains for $G$.

	We claim that 
	\begin{equation}\label{eq:eqTRENO}
		G_{\omega_1,\omega_2,\omega_3} = G_{\omega_1',\omega_2',\omega_3',\omega_4'}.
	\end{equation}
	As $\{\omega_1,\omega_2,\omega_3\}=\{\omega_1',\omega_2',\omega_4'\}$, it suffices to show that $G_{\omega_1,\omega_2,\omega_3}$ fixes $\omega_3'$. Let $A=\mathrm{P}\Gamma\mathrm{L}_d(q)\rtimes\langle\iota\rangle$. Using the fact that $\iota$ and $\mathrm{Aut}(\mathbb{F}_q)$ fix $\omega_1$ and $\omega_2$, it is easy to verify that
$$A_{\omega_1,\omega_2}=
\left\{
\left[
\begin{array}{ccccc}
1&0&0&0&0\\
0&a&0&0&0\\
0&0&c&0&x\\
0&0&0&X&0\\
0&0&0&0&Y
\end{array}\right]\tau\iota^\varepsilon\mid 
\begin{array}{c}
a,b\in\mathbb{F}_q\setminus\{0\},x\in\mathbb{F}_q^{d-k-2}\\X\in\mathrm{GL}_{k-1}(q),Y\in \mathrm{GL}_{d-k-2}(q) ,\\
\tau\in\mathrm{Aut}(\mathbb{F}_q),\varepsilon\in\{0,1\}
\end{array}
\right\}.
$$
From this, it follows that
$$A_{\omega_1,\omega_2,\omega_3}=\left\{
\left[
\begin{array}{ccc}
I_3&0&0\\
0&X&0\\
0&0&Y
\end{array}
\right]\mid 
\begin{array}{c}
X\in\mathrm{GL}_{k-1}(q),\\Y\in\mathrm{GL}_{d-k-2}(q)
\end{array}
\right\}\rtimes
\mathrm{Aut}(\mathbb{F}_q).
$$
From this explicit calculation, it is immediate that each element of $A_{\omega_1,\omega_2,\omega_3}$ fixes $\omega_3'$ and hence, each element of $G_{\omega_1,\omega_2,\omega_3}$ fixes $\omega_3'$.
By~\eqref{eq:eqTRENO} and Lemma~\ref{lemma:1}, $G$ is not IBIS.

It remains to deal with the cases $q=2$ and $(d,q)=(3,4)$. When $(d,q)=(3,4)$, $\mathrm{PSL}_3(4)\unlhd G\le \mathrm{PSL}_3(4)\rtimes\langle\iota\rangle$ and $|\Omega|=336$. It can be verified directly, or with the auxiliary help of a computer, that $G$ is not IBIS.

Assume $q=2$. Thus $G=\mathrm{GL}_d(2)\rtimes\langle\iota\rangle$. Let $U=\langle e_3,\ldots,e_{k+1}\rangle$ and $W=\langle e_{k+2},\ldots,e_d\rangle$. As above, there is an abuse of notation here: when $d=3$, $U=0$. We define
\begin{align*}
\omega_1=\omega_1'&=\{\langle e_1,U\rangle,\langle e_2,W\rangle\},\\
\omega_2=\omega_3'&=\{\langle e_2,U\rangle,\langle e_1+e_2,W\rangle\},\\
\omega_2'&=\{\langle e_2,U\rangle,\langle e_1,W\rangle\}.
\end{align*}
We claim that
$$G_{\omega_1,\omega_2}=G_{\omega_1',\omega_2',\omega_3'},$$
from which it follows, thanks to Lemma~\ref{lemma:1}, that $G$ is not IBIS.
It  is immediate to verify that
$$G_{\omega_1',\omega_2'}=
\left\{
\left[
\begin{array}{ccc}
I_2&0&0\\
0&X&0\\
0&0&Y
\end{array}
\right]\mid 
\begin{array}{c}
X\in\mathrm{GL}_{k-1}(2), Y\in\mathrm{GL}_{d-k-1}(2)
\end{array}
\right\}
\rtimes\langle\iota\rangle
$$
and 
$$G_{\omega_1',\omega_2',\omega_3'}=
\left\{
\left[
\begin{array}{ccc}
I_2&0&0\\
0&X&0\\
0&0&Y
\end{array}
\right]\mid 
\begin{array}{c}
X\in\mathrm{GL}_{k-1}(2), Y\in\mathrm{GL}_{d-k-1}(2)
\end{array}
\right\}.
$$
It is also clear that
$$\mathrm{GL}_d(2)_{\omega_1,\omega_2}=
\left\{
\left[
\begin{array}{ccc}
I_2&0&0\\
0&X&0\\
0&0&Y
\end{array}
\right]\mid 
\begin{array}{c}
X\in\mathrm{GL}_{k-1}(2), Y\in\mathrm{GL}_{d-k-1}(2)
\end{array}
\right\}=G_{\omega_1',\omega_2',\omega_3'}.
$$
Suppose, arguing by contradiction, that $\iota g\in G_{\omega_1,\omega_2}$, for some $g\in \mathrm{GL}_d(2)$. Now,
\begin{align*}
\omega_1^\iota&=\{\langle e_1,U\rangle,\langle e_2,W\rangle\}^\iota=
\{\langle e_2,W\rangle,\langle e_1,U\rangle\}=\omega_1,\\
\omega_2^\iota&=\{\langle e_2,U\rangle,\langle e_1+e_2,W\rangle\}^\iota=
\{\langle e_1,W\rangle,\langle e_1+e_2,U\rangle\}.
\end{align*}
From this and from the fact that $\iota g$ fixes $\omega_1,\omega_2$, we deduce
\begin{align*}
\langle e_1,U\rangle^g&=\langle e_1,U\rangle,\\
\langle e_2,W\rangle^g&=\langle e_2,W\rangle,\\
\langle e_1+e_2,U\rangle^g&=\langle e_2,U\rangle,\\
\langle e_1,W\rangle^g&=\langle e_1+e_2,W\rangle.
\end{align*}
From this it follows that $g$ fixes the subspaces $\langle e_1,e_2\rangle$, $U$ and $W$. Now, the first equality implies $e_1^g=e_1$ and the forth equality implies $e_1^g=e_2$, which is a contradiction.
\end{proof}

\subsubsection{Case~\eqref{psk2}}\label{sec:psk2}
In this section, we deal with case~\eqref{psk2}. This case is  easier to handle.

\begin{lemma}\label{refdaaggiustareeepart1}
	Let $G$ be a group as in~\eqref{psk2} acting on $\Omega_k^2$. Then $G$ is not $\mathrm{IBIS}$.
\end{lemma}
\begin{proof}
Let $A=\mathrm{P}\Gamma\mathrm{L}_d(q)\rtimes \langle \iota\rangle $. Let $U=\langle e_4,\ldots,e_{k+2}\rangle$ and $W=\langle e_4,\ldots,e_{d-k+1}\rangle$. Observe that $U\le W$, because $k+2\le d-k+1$ (recall $k<d/2$). Moreover, $\dim U=k-1$ and $\dim W=d-k-2$. As in the previous lemmas, there is a slight abuse of notation here. Indeed, when $k=1$, $U=0$. Now, let
\begin{align*}
\omega_1 &= \omega_1' = \{\langle e_1,U \rangle, \langle e_1,e_2,W\rangle\}, \\
 \omega_2 &= \omega_2' = \{\langle e_1,U \rangle, \langle e_1, e_3,W\rangle\}, \\
 \omega_3&= \{\langle e_2,U\rangle, \langle e_1,e_2,W\rangle \}, \\
 \omega_4 &= \omega_3' = \{ \langle e_2,U \rangle, \langle e_2, e_3,W\rangle\}.
\end{align*}

Clearly, $G_{\omega_1,\omega_2}=G_{\omega_1',\omega_2'}\le\mathrm{P}\Gamma\mathrm{L}_d(q)$, because $\omega_1$ and $\omega_2$ share the same $k$-dimensional subspace. From $G_{\omega_1',\omega_2',\omega_3'}\le\mathrm{P}\Gamma\mathrm{L}_d(q)$, it follows that $G_{\omega_1',\omega_2',\omega_3'}$ fixes $\omega_3$ and hence 
\begin{equation}\label{29_11_2023}G_{\omega_1',\omega_2',\omega_3'}=G_{\omega_1,\omega_2,\omega_3,\omega_4}.
\end{equation} Now, consider the matrices
\[g_1=\left[
\begin{array}{cccc}
1&0&0&0\\
1&1&0&0\\
1&1&1&0\\
0&0&0&I_{d-3}
\end{array}\right],\, 
g_2=\left[
\begin{array}{cccc}
1&0&0&0\\
1&1&0&0\\
1&0&1&0\\
0&0&0&I_{d-3}
\end{array}\right],\, 
g_3=\left[
\begin{array}{cccc}
1&0&0&0\\
0&1&0&0\\
1&0&1&0\\
0&0&0&I_{d-3}
\end{array}\right].
\]
Observe $g_1,g_2,g_3\in\mathrm{PSL}_d(q)\le G$. Moreover, $g_1\in G_{\omega_1}\setminus G_{\omega_1,\omega_2}$, $g_2\in G_{\omega_1,\omega_2}\setminus G_{\omega_1,\omega_2,\omega_3}$, $g_3\in G_{\omega_1,\omega_2,\omega_3}\setminus G_{\omega_1,\omega_2,\omega_3,\omega_4}$. Therefore, $(\omega_1,\omega_2,\omega_3,\omega_4)$ and $(\omega_1',\omega_2',\omega_3')$ are irredundant chains for $G$. Now, from~\eqref{29_11_2023} and Lemma~\ref{lemma:1}, this shows that $G$ is not IBIS.
\end{proof}

We conclude this section summing up the main result.
\begin{proposition}Let $G$ be an almost simple primitive group having socle $\mathrm{PSL}_d(q)$ in a subspace action. If $G$ is $\mathrm{IBIS}$, then either
\begin{itemize}
\item $q=2$ and $G=\mathrm{GL}_d(2)$ is acting on the non-zero vectors of $\mathbb{F}_2^d$, or
\item $d=2$ and $|G:G\cap\mathrm{PGL}_2(q)|$ is either $1$ or a prime number.
\end{itemize}
\end{proposition}

\subsection{Unitary groups}\label{psu}
Let $d$ be a positive integer and $q = p^f$, with $p$ a prime number. Let $G_0 = \mathrm{PSU}_d(q)$ and let $G$ be an almost simple group having socle $G_0$. In particular, $G_0 \leq G \trianglelefteq \mathrm{Aut}(G_0) =\mathrm{PGU}_d(q)\rtimes \mathrm{Aut}(\mathbb{F}_{q^2})  $. 
In the following, we set $A = \mathrm{Aut}(\mathrm{PSU}_d(q))$.

\begin{lemma}\label{lemma:auxiliary}
There exists  $\alpha\in \mathbb{F}_{q^2}$ with $\alpha+\alpha^q+1=0$ such that the orbit of $\alpha$ under the action of $\mathrm{Aut}(\mathbb{F}_{q^2})$ has cardinality $2f$.
\end{lemma}
\begin{proof}
 Let 
	\[
		S = \{ \alpha \in \mathbb{F}_{q^2} \, : \, \alpha + \alpha^q+1=0\}.
	\]
Consider the trace map $\mathrm{Tr}:\mathbb{F}_{q^2}\to\mathbb{F}_q$ defined by $x\mapsto x+x^q$. Since $\mathrm{Tr}$ is a surjective $\mathbb{F}_q$-linear mapping, we deduce that 
\begin{equation}\label{eq:S}
|S|=q.
\end{equation}

Recall that $\mathrm{Aut}(\mathbb{F}_{q^2})$ is cyclic of order $2f$. Let $2f = r_1^{l_1} \ldots r_t^{l_t}$ be the factorization of $2f$ in prime numbers with $2=r_1<r_2<\cdots<r_t$. The elements of $S$ not lying on a regular orbit under the action of $\mathrm{Aut}(\mathbb{F}_{q^2})$ are fixed by some element of $\mathrm{Aut}(\mathbb{F}_{q^2})$ of prime order $r_i$, for some $i\in \{1,\ldots,t\}$. Actually, using the fact that this group is cyclic, we deduce that the elements of $S$ fixed by some automorphism of order $r_i$ are the elements of the set 
$$S_i=S\cap\mathbb{F}_{q^{2/r_i}}=\{\alpha\in S\mid \alpha\in \mathbb{F}_{q^{2/r_i}}\}.$$

Suppose $i=1$ and let $\alpha\in S_1$. Then $\alpha^q=\alpha$ because $\alpha\in\mathbb{F}_q$. Therefore $0=\alpha+\alpha^q+1=2\alpha+1$ and hence $|S_1|\le 1$\footnote{Actually, $|S_1|=0$ when $q$ is even and $|S_1|=1$ when $q$ is odd.}

Suppose $i>1$ and let $\alpha\in S_i$. Since $r_i$ is odd and $\alpha\in \mathbb{F}_{q^{2/r_i}}$\footnote{Since $r_i$ is odd, we may write $f=\frac{2f}{r_i}\cdot \frac{r_i-1}{2}+\frac{f}{r_i}$. Now observe that, as $\alpha\in \mathbb{F}_{q^{2/r_i}}=\mathbb{F}_{p^{2f/r_i}}$, we get $\alpha^{p^{\frac{2f}{r_i}\cdot \frac{r_i-1}{2}}}=\alpha$.}, we deduce $\alpha^q=\alpha^{p^{f/r_i}}$. Therefore
$$0=\alpha+\alpha^q+1=\alpha+\alpha^{p^{\frac{f}{r_i}}}+1.$$
Since this can be viewed as a polynomial equation in $\alpha$ of degree $p^{f/r_i}$, we deduce that it has at most $p^{f/r_i}=q^{1/r_i}$ solutions and hence $|S_i|\le q^{1/r_i}$. 

Summing up,
\begin{align}\label{eq:Sagain}
\left|\bigcup_{i=1}^tS_i\right|&\le \sum_{i=1}^t|S_i|\le 1+\sum_{i=2}^t|S_i|\le 1+\sum_{i=2}^tq^{1/r_i}\le 1+(t-1)q^{1/3}\\\nonumber
&\le 1+(\log_2q-1)q^{1/3}.
\end{align}
 Using~\eqref{eq:S} and~\eqref{eq:Sagain}, we deduce that $\bigcup_iS_i$ is a proper subset of $S$ as long as $q>1+(\log_2q-1)q^{1/3}$. Since this inequality is always satisfied, we deduce the existence of $\alpha\in S$ lying on an orbit of cardinality $2f$ under the action of $\mathrm{Aut}(\mathbb{F}_{q^2})$. 
\end{proof}

 We let $V=\mathbb{F}_{q^2}^d$ be the $d$-dimensional vector space over the finite field $\mathbb{F}_{q^2}$ of cardinality $q^2$ and let $k\in \{1,\ldots,d-1\}$. Then, for $1\le k\le d/2$, we let $\Omega_k$ be the collection of all $k$-dimensional totally singular subspaces of $V$.

\begin{lemma}\label{psu:cased3}
When $d=3$ and $q\ne 2$, there exist two positive integers $\ell,\ell'\ge 2$ with $\ell\ne\ell'$ and there exist $\omega_1,\ldots,\omega_\ell$, $\omega_1',\ldots,\omega_{\ell'}'$ elements of $\Omega_1$ such that
\begin{itemize}
\item $(\omega_1,\ldots,\omega_\ell)$ and $(\omega_1',\ldots,\omega_{\ell'}')$ are irredundant bases for the action of $G$ on $\Omega_1$,
\item $\omega_1+\cdots+\omega_\ell=V=\omega_{1}'+\cdots+\omega_{\ell'}'$, $\omega_1\cap\cdots\cap\omega_\ell=0=\omega_1'\cap\cdots\cap\omega_{\ell'}'$.
\end{itemize}
In particular, the action of $G$ on $\Omega_1$ is not $\mathrm{IBIS}$. When $(d,q)=(3,2)$,  the action of $G$ on $\Omega_1$ is $\mathrm{IBIS}$.
\end{lemma}
\begin{proof}
Let $e_1,e_2,e_3$ be an $\mathbb{F}_{q^2}$-basis of $V$ such that the Hermitian form defining $G_0$ is
\[
\begin{pmatrix}
0&0&1\\
0&1&0\\
1&0&0
\end{pmatrix}.
\]

When $(d,q)=(3,2)$, the statement can be checked directly: $G$ is IBIS for each choice of $G$ with $\mathrm{PSU}_3(2)\le G\le \mathrm{P}\Gamma\mathrm{U}_3(2)$. Therefore, for the rest of the proof we suppose $(d,q)\ne (3,2)$.
  
	The set on which $G$ acts is
	\[
		\Omega = \{\langle(1,0,0)\rangle\} \cup \{(\alpha,\beta,1) \, : \, \alpha + \alpha^q+\beta\beta^q = 0\}.
	\]
	We give an irredundant base of size $4$ for $G_0$ and an irredundant base of size $3$ for $A$, hence by Lemma~\ref{lemma:minusminus1} $G$ is not IBIS.
	
	Take $\omega_1 = \langle e_1 \rangle$, $\omega_2 = \langle e_3 \rangle$ and $\omega_3 = \langle \alpha e_1+e_3\rangle$, where $\alpha\in\mathbb{F}_{q^2}\setminus\{0\}$ satisfies $\alpha + \alpha^q= 0$, and $\omega_4 = \langle \tilde{\alpha}e_1+ \beta e_2+e_3\rangle$, where $\tilde\alpha,\beta\in\mathbb{F}_{q^2}$ satisfy $\tilde{\alpha}+\tilde{\alpha}^q+\beta^{q+1}=0$ and $\beta\ne0$. A computation shows that
$$(G_0)_{\omega_1,\omega_2}=
\left\{
\left[
\begin{array}{ccc}
a&0&0\\
0&a^{q-1}&0\\
0&0&a^{-q}
\end{array}
\right]\mid a\in\mathbb{F}_{q^2}\setminus\{0\}
\right\}.$$	
Observe that $(G_0)_{\omega_1,\omega_2}\ne 1$, because $q\ne 2$. From this description of $(G_0)_{\omega_1,\omega_2}$, we get
$$(G_0)_{\omega_1,\omega_2,\omega_3}=
\left\{
\left[
\begin{array}{ccc}
a&0&0\\
0&a^{q-1}&0\\
0&0&a^{-q}
\end{array}
\right]\mid a\in\mathbb{F}_{q^2}, a^{q+1}=1
\right\}.$$
Observe that $(G_0)_{\omega_1,\omega_2,\omega_3}$ has order $(q+1)/\gcd(3,q+1)$. Thus $(G_0)_{\omega_1,\omega_2,\omega_3}\ne 1$, because $q\ne 2$. It is now easy to check that $(G_0)_{\omega_1,\omega_2,\omega_3,\omega_4}=1$ and hence $(\omega_1,\omega_2,\omega_3,\omega_4)$ is an irredundant base for $G_0$. Finally, observe that $\omega_1 + \cdots+ \omega_4 = V$ and $\omega_1\cap\cdots\cap \omega_4=0$. We can now complete this sequence to construct an irredundant base  $(\omega_1,\dots,\omega_\ell)$ for $G$ with $\ell \geq 4$ such that $\omega_1 + \dots + \omega_\ell= V$ and $\omega_1\cap\cdots\cap\omega_\ell=0$.

Let $A=\mathrm{P}\Gamma\mathrm{U}_d(q)$.	We now construct an irredundant base of cardinality $3$ for $A$. From Lemma~\ref{lemma:auxiliary}, let $\alpha\in \mathbb{F}_{q^2}$ with $\alpha+\alpha^q+1=0$ and such that the orbit of $\alpha$ under the action of $\mathrm{Aut}(\mathbb{F}_{q^2})$ has cardinality $2f$. Now take $\omega_1^\prime = \omega_1$, $\omega_2^\prime = \omega_2$ and $\omega_3^\prime = \langle \alpha e_1+e_2+e_3\rangle$. 
	The same  computation as above shows that
	\[
	A_{\omega_1^\prime,\omega_2^\prime} = 
		\left\langle
			\begin{bmatrix}
				\gamma & 0 & 0 \\
				0 & \delta & 0 \\
				0 & 0 & \gamma^{-q}
			\end{bmatrix}, \mathrm{Aut}(\mathbb{F}_{q^2}\backslash \mathbb{F}_p) \, : \, \gamma,\delta\in\mathbb{F}_{q^2}\setminus\{0\},\delta^{q+1}=1\right\rangle.
	\]
Let $g\in A_{\omega_1',\omega_2'}$ with $g$ fixing $\omega_3'$. From the description above, we see that there exists $\gamma,\delta\in\mathbb{F}_{q^2}\setminus\{0\}$ with $\delta^{q+1}=1$ and there exists $\sigma\in\mathrm{Aut}(\mathbb{F}_{q^2})$ with
$$\alpha^\sigma \gamma e_1+\delta e_2+\gamma^{-q}e_3\in \omega_3'=\langle \alpha e_1+e_2+e_3\rangle.$$	
By checking the second and the third coordinate, we deduce $\delta=\gamma^{-q}$ and hence $\gamma^{q+1}=1$, because $\delta^{q+1}=1$. Now, by checking the first and second coordinate, we deduce $$\alpha^\sigma \gamma=\alpha\delta=\alpha \gamma^{-q}=\alpha\gamma.$$
Therefore $\alpha^\sigma=\sigma$. Since $\alpha$ lies in a regular orbit under the action of $\mathrm{Aut}(\mathbb{F}_{q^2})$, we get $\sigma=1$. This shows that $g=1$ and hence $A_{\omega_1',\omega_2',\omega_3'}=1.$ As before $\omega_1'+\omega_2'+\omega_3'=V$, $\omega_1'\cap\omega_2'\cap\omega_3'=0$ and it is also a base for $G$.
\end{proof}

We deal with the case $d=4$ and $k=2$.

\begin{lemma}\label{psu:cased=4k=2}
When $d=4$, there exist two positive integers $\ell,\ell'\ge 2$ with $\ell\ne\ell'$ and there exist $\omega_1,\ldots,\omega_\ell$, $\omega_1',\ldots,\omega_{\ell'}'$ elements of $\Omega_2$ such that
\begin{itemize}
\item $(\omega_1,\ldots,\omega_\ell)$ and $(\omega_1',\ldots,\omega_{\ell'}')$ are irredundant bases for the action of $G$ on $\Omega_2$,
\item $\omega_1+\cdots+\omega_\ell=V=\omega_{1}'+\cdots+\omega_{\ell'}'$, $\omega_1\cap\cdots\cap\omega_\ell=0=\omega_1'\cap\cdots\cap\omega_{\ell'}'$.
\end{itemize}
\end{lemma}
\begin{proof}
We use the Hermitian form given by the matrix
$$\begin{pmatrix}
0&0&0&1\\
0&0&1&0\\
0&1&0&0\\
1&0&0&0
\end{pmatrix}$$
and we use the basis $e_1,e_2,f_1,f_2$ of $V=\mathbb{F}_{q^2}^4$. Let $G_0$ be the socle of $G$.

Let $\omega_1=\langle e_1,e_2\rangle$, $\omega_2=\langle f_1,f_2\rangle$, $\omega_3=\langle e_1,f_2\rangle$, $\omega_4=\langle e_2,f_1\rangle$, $\omega_5=\langle e_1+e_2,f_1-f_2\rangle$ and $\omega_6=\langle e_1+f_2,e_2-f_1\rangle$. It is readily seen that $\omega_1,\ldots,\omega_6$ are $2$-dimensional totally singular subspaces of $V$. An easy computation shows that
$$(G_0)_{\omega_1,\omega_2,\omega_3,\omega_4}=
\left\{
\left[
\begin{array}{cccc}
a&0&0&0\\
0&b&0&0\\
0&0&a^{-q}&0\\
0&0&0&b^{-q}
\end{array}
\right]\mid a,b\in \mathbb{F}_{q^2}\setminus\{0\} \hbox{ with }ab\in\mathbb{F}_q
\right\}.$$
In particular, $|(G_0)_{\omega_1,\omega_2,\omega_3,\omega_4}|=(q^2-1)(q-1)/\gcd(4,q+1)\ne 1$. We deduce$$(G_0)_{\omega_1,\omega_2,\omega_3,\omega_4,\omega_5}=
\left\{
\left[
\begin{array}{cccc}
a&0&0&0\\
0&a&0&0\\
0&0&a^{-q}&0\\
0&0&0&a^{-q}
\end{array}
\right]\mid a\in \mathbb{F}_{q^2}\setminus\{0\} \hbox{ with }a^2\in\mathbb{F}_q
\right\}.$$
From this it follows that,
\[
|(G_0)_{\omega_1,\omega_2,\omega_3,\omega_4,\omega_5}|=
\begin{cases}
q-1&\textrm{if }q \textrm{ is even},\\
\frac{2(q-1)}{\gcd(4,q+1)}&\textrm{if }q \textrm{ is odd}.
\end{cases}
\]
Thus $|(G_0)_{\omega_1,\omega_2,\omega_3,\omega_4,\omega_5}|=1$ if and only if $q\in \{2,3\}$.
This show that $\omega_1,\ldots,\omega_5$ is an irredundant base for $G_0$ of cardinality $5$ when $q\in \{2,3\}$, whereas $\omega_1,\ldots,\omega_6$ is an irredundant base for $G_0$ of cardinality $6$ when $q\ge 4$.

The argument above shows that
$$A_{\omega_1,\omega_2,\omega_3,\omega_4}=
\left\{
\left[
\begin{array}{cccc}
1&0&0&0\\
0&b&0&0\\
0&0&1&0\\
0&0&0&b^{-q}
\end{array}
\right]\mid b\in \mathbb{F}_{q^2}\setminus\{0\}
\right\}\rtimes\mathrm{Aut}(\mathbb{F}_{q^2}).$$
Let $q=p^f$ and let $\alpha$ be a generator of the multiplicative group of the field $\mathbb{F}_{q^2}$ and let $\omega_5'=\langle e_1+\alpha f_2,e_2-\alpha f_1\rangle$. Observe that $\omega_5'$ is a $2$-dimensional totally singular  subspace of $V$. Let $g\in A_{\omega_1,\omega_2,\omega_3,\omega_4,\omega_5'}$. Then there exist $b\in\mathbb{F}_{q^2}\setminus\{0\}$ and $\eta\in\mathrm{Aut}(\mathbb{F}_{q^2})$ such that
$$\omega_5'=\omega_5'^g=\langle e_1+\alpha^\eta b^{-q} f_2,be_2-\alpha^\eta f_1\rangle.$$
This yields $\alpha=\alpha^\eta b^{-q}$ and $\alpha b=\alpha^\eta$. Therefore $b=b^q=\alpha^{q}\alpha^{-1}$ and hence $\alpha^{\eta}\alpha^{-1}\in\mathbb{F}_q$. Now, $\alpha^{\eta}=\alpha^{p^i}$, where $i$ is a divisor of $2f$. Since $\alpha^{\eta}\alpha^{-1}=\alpha^{p^i-1}\in \mathbb{F}_{q}$, we deduce that $\alpha^{(p^i-1)(p^f-1)}=1$ and, as $\alpha$ is a generator of the multiplicative group of $\mathbb{F}_{q^2}$, we get that $p^{2f}-1$ divides $(p^i-1)(p^f-1).$ If $i=2f$, then $\alpha$ is the identity and $b=1$ and hence $g=1$. Assume that $i<2f$. Then $i\le f$ and hence $(p^i-1)(p^f-1)< p^{2f}-1$, however this contradicts the fact that $p^{2f}-1$ divides $(p^i-1)(p^f-1).$ This shows that $\omega_1,\omega_2,\omega_3,\omega_4,\omega_5'$ is an irredundant base for $A$ (and hence for $G$) of cardinality $5$. This implies that $G$ is not IBIS when $q\ge 4$.  Observe also that these two irredundant bases satisfy the additional requirements in the statement of the lemma.

Suppose $q=2$. When $G=\mathrm{PSU}_4(2)$, we have found irredundant bases of cardinality $4$ and $5$ (satisfying the requirements in the statement of this lemma) and hence $G$ is not IBIS. Similarly, if $G\ne\mathrm{PSU}_4(2)$, then $G=\mathrm{P}\Gamma\mathrm{U}_4(2)$ and we have found irredundant bases of cardinality $5$ and $6$ and hence $G$ is not IBIS.

Suppose $q=3$. We have found irredundant bases of cardinality $6$ for $\mathrm{PSU}_4(3)$ and hence we may use the argument above to deduce that $G$ is not IBIS.
\end{proof}

\begin{lemma}\label{psu:cased=4k=2q=2}
When $(d,q)=(5,2)$, there exist two positive integers $\ell,\ell'\ge 2$ with $\ell\ne\ell'$ and there exist $\omega_1,\ldots,\omega_\ell$, $\omega_1',\ldots,\omega_{\ell'}'$ in $\Omega_2$ such that
\begin{itemize}
\item $(\omega_1,\ldots,\omega_\ell)$ and $(\omega_1',\ldots,\omega_{\ell'}')$ are irredundant bases for the action of $G$ on $\Omega_2$,
\item $\omega_1+\cdots+\omega_\ell=V=\omega_{1}'+\cdots+\omega_{\ell'}'$, $\omega_1\cap\cdots\cap\omega_\ell=0=\omega_1'\cap\cdots\cap\omega_{\ell'}'$.
\end{itemize}
\end{lemma}
\begin{proof}
This follows by a direct inspection in $\mathrm{P}\Gamma\mathrm{U}_5(2)$ and in $\mathrm{PSU}_5(2)$.
\end{proof}

\begin{proposition}\label{psu:cased>3}
When $(d,q)\ne (3,2)$, the action of $G$ on $\Omega_k$ is not $\mathrm{IBIS}$ and $b(G)\ge 2$. When $(d,q)=(3,2)$, the action of $G$ on $\Omega_1$ is $\mathrm{IBIS}$.
\end{proposition}
\begin{proof}
We first deal with the case $k=1$, this will set up the base case of an inductive argument.

When $d=3$, the proof follows from Lemma~\ref{psu:cased3}; therefore, for the rest of the argument, we suppose $d\ge 4$. Assume $q\ne 2$. Let $W$ be a non-degenerate $3$-dimensional subspace of $V$ and let 
$$\Delta=\{\omega\mid \omega\le W,\dim_{\mathbb{F}_{q^2}}(\omega)=1\hbox{ and }\omega \hbox{ singular}\}.$$
Let $G_{\Delta}$ be the stabilizer of $\Delta$ in $G$ and let $H$ be the permutation group induced by $G_\Delta$ on $\Delta$.

Observe that $G_{\Delta}=G_{W}$ because the span of the singular vectors of $V$ contained in $W$ is $W$ itself. Moreover, $H$ is an almost simple primitive group having socle $\mathrm{PSU}_3(q)$ acting on the singular $1$-dimensional subspaces of $W$. From Lemma~\ref{psu:cased3}, there exist two positive integers $\ell,\ell'\ge 2$ with $\ell\ne\ell'$ and there exist $\omega_1,\ldots,\omega_\ell$, $\omega_1',\ldots,\omega_{\ell'}'$ totally singular $1$-dimensional subspaces of $W$ such that
\begin{itemize}
\item $(\omega_1,\ldots,\omega_\ell)$ and $(\omega_1',\ldots,\omega_{\ell'}')$ are irredundant bases for the action of $H$ on $\Delta$,
\item $\omega_1+\cdots+\omega_\ell=W=\omega_{1}'+\cdots+\omega_{\ell'}'$.
\end{itemize}
Since $\omega_1+\cdots+\omega_\ell=W=\omega_{1}'+\cdots+\omega_{\ell'}'$, we deduce that $G_{\omega_1,\ldots,\omega_\ell}$ and $G_{\omega_1',\ldots,\omega_{\ell'}'}$ both stabilize $W$, that is, 
$G_{\omega_1,\ldots,\omega_\ell},G_{\omega_1',\ldots,\omega_{\ell'}'}\le G_W=G_\Delta$. From this, since $H_{\omega_1,\ldots,\omega_\ell}=H_{\omega_1',\ldots,\omega_{\ell'}'}=1$, we deduce
$$G_{\omega_1,\ldots,\omega_\ell}=G_{\omega_1',\ldots,\omega_{\ell'}'}.$$
As $\ell\ne\ell'$, Lemma~\ref{lemma:1} implies that $G$ is not IBIS.

Finally, assume $q=2$. Let $A=\mathrm{P}\Gamma\mathrm{U}_d(2)$. We first suppose $d=4$. We use the Hermitian form given by the matrix
$$\begin{pmatrix}
0&0&0&1\\
0&0&1&0\\
0&1&0&0\\
1&0&0&0
\end{pmatrix}.$$
Let $\omega_1=\langle e_1\rangle$, $\omega_2=\langle e_2\rangle$, $\omega_3=\langle e_3\rangle$ and $\omega_4=\langle e_4\rangle$. A computation shows that $(G_0)_{\omega_1,\omega_2,\omega_3,\omega_4}$ has order $3$ and is generated by
$$
\begin{bmatrix}
1&0&0&0\\
0&\alpha&0&0\\
0&0&\alpha^{2}&0\\
0&0&0&1
\end{bmatrix},$$
where $\alpha$ is a generator of $\mathbb{F}_4^\ast$. This shows that $G$ has an irredundant base $\omega_1,\ldots,\omega_\ell$, with $\ell= 5$ and with $V=\omega_1+\cdots+\omega_\ell$. Similarly, let $\omega_1'=\langle e_1\rangle$, $\omega_2'=\langle e_2\rangle$, $\omega_3'=\langle e_4\rangle$ and $\omega_4'=\langle e_1+e_2+\alpha e_3+\alpha e_4\rangle$. A computation shows that $A_{\omega_1',\omega_2',\omega_3'}=(G_0)_{\omega_1',\omega_2',\omega_3'}\rtimes\mathrm{Aut}(\mathbb{F}_4)$, where
$$
(G_0)_{\omega_1',\omega_2',\omega_3'}=
\left\{
\begin{bmatrix}
1&0&0&0\\
0&a&0&0\\
0&z&a^2&0\\
0&0&0&1
\end{bmatrix}\mid a,z\in\mathbb{F}_4,a\ne 0, az+a^2z^2=0
\right\}.$$
Let $g\in A_{\omega_1',\omega_2',\omega_3',\omega_4'}$. From the description of  $A_{\omega_1',\omega_2',\omega_3'}$, there exist $a,z\in\mathbb{F}_4$ (with $a\ne 0$ and $az+a^2z^2=0$) and an automorphism $\eta$ of $\mathbb{F}_4$ such that
$$\omega_4'=\langle e_1+(a+\alpha^\eta z)e_2+\alpha^\eta a^2e_3+\alpha^{\eta}e_4\rangle.$$This yields the linear system of equations
$$
\left\{
\begin{array}{rcl}
1&=&a+\alpha^\eta z,\\
\alpha&=&a^2\alpha^\eta ,\\
\alpha &=&\alpha^\eta .
\end{array}
\right.
$$
Using the fact that $\alpha$ is a generator of $\mathbb{F}_4^\ast$, this gives $a=1$, $z=0$ and $\eta=1$. 
This shows that $\omega_1',\omega_2',\omega_3',\omega_4'$ is an irredundant base for $G$  with $V=\omega_1'+\omega_2'+\omega_3'+\omega_4'$. This has established that $G$ is not IBIS. The case $d\ge 4$ can be deal with a similar argument as above (used for $q\ne 2$), by taking $W$ a non-degenerate $4$-dimensional subspace of $V$. This completes the case when $k=1$.

\medskip

Suppose that $k>1$. 
Let $U$ be a totally singular subspace of $V$ with $\dim_{\mathbb{F}_{q^2}}(U)=k-1$ and let
\begin{align*}
\Delta=\{W\mid U\le W, \mathrm{dim}_{\mathbb{F}_{q^2}}(W)=k \hbox{ and }W \hbox{ totally singular}\}.\end{align*}
Let $H$ be the permutation group induced by $G$ on $\Delta$.

Observe that $G_\Delta=G_{U}$, where $G_U$ is the setwise stabilizer of $U$ in $G$. The Hermitial form $\kappa$ defining the unitary group $G_0$ induces a form $\kappa'$ on the $\mathbb{F}_{q^2}$-vector space $U^\perp/U$. Observe that $\dim_{\mathbb{F}_{q^2}}(U^\perp/U)=d-2k+2$ and that the form induced on $U^\perp/U$ is non-degenerate and Hermitian.  Let $\bar V=U^\perp/U$ and consider $$\Omega=\{W'\le \bar V\mid\mathrm{dim}_{\mathrm{F}_{q_0}}(W')=1 \hbox{ and } W'\hbox{ is totally singular for }\kappa'\}.$$
Clearly, there exists a one-to-one correspondence between the elements of $\Delta$ and the elements of $\Omega$; moreover, this correspondence allows to identify the action of $H$ on $\Delta$ with the natural action on the $1$-dimensional totally singular subspaces of $\bar V$.

Suppose $d-2k+2\ge 3$ if $q\ge 3$  and $d-2k+2\ge 4$ when $q=2$.
From Lemma~\ref{psu:cased3}, the action of $H$ on $\Delta$ is not IBIS and hence there exist two irredundant sequences $(\delta_1,\ldots,\delta_{\ell})$ and $(\delta_1',\ldots,\delta_{\ell'}')$ of points of $\Delta$ such that $\ell,\ell'\ge 2$, $\ell\ne \ell'$ and
\begin{equation}\label{eq:nastyagain}
G_{\delta_1,\ldots,\delta_{\ell}}\cap G_{\Delta}=G_{\delta_1',\ldots,\delta_{\ell'}'}\cap G_\Delta.
\end{equation}
Since $\ell\ge 2$, $\delta_1$ and $\delta_2$ are two distinct $k$-dimensional subspaces of $V$ containing $U$ and hence 
$$\delta_1\cap \delta_2=U.$$
This shows that $G_{\delta_1,\delta_2}=G_{\delta_1}\cap G_{\delta_2}\le G_{U}=G_{\Delta}$ and hence $G_{\delta_1,\ldots,\delta_{\ell}}\le G_{\Delta}$. Since this argument applies also for $\ell'$, from~\eqref{eq:nastyagain}, we deduce 
$$G_{\delta_1,\ldots,\delta_{\ell}}=G_{\delta_1',\ldots,\delta_{\ell'}'}.$$
As $\ell\ne \ell'$, from Lemma~\ref{lemma:1}, we deduce that the action of $G$ on $\Omega_k$ is not IBIS. It remains to consider the case $d-2k+2\le 2$ (that is, $d=2k$), and $(d,q)=(2k+1,2)$.

Suppose $d=2k$, or $d=2k+1$ and $q=2$.  
Let $U$ be a totally singular subspace of $V$ with $\dim_{\mathbb{F}_{q^2}}(U)=k-2$ and let
\begin{align*}
\Delta=\{W\mid U\le W, \mathrm{dim}_{\mathbb{F}_{q^2}}(W)=k \hbox{ and }W \hbox{ totally singular}\}.\end{align*}
We may apply the argument above verbatim, with the old $U$ replaced by this new $U$. Observe that $\dim(U^\perp/U)=4$ when $d=2k$ and $\dim(U^\perp/U)=5$ when $(d,q)=(2k+1,2)$. We deduce, using Lemma~\ref{psu:cased=4k=2} when $d=2k$ and Lemma~\ref{psu:cased=4k=2q=2} when $(d,q)=(2k+1,2)$, that $G$ is not IBIS.
\end{proof}

\subsection{Symplectic groups}\label{psp}
Let $d\ge 4$ be an even positive integer, let $q=p^f$ be a prime power with $p$ a prime number, let $G_0=\mathrm{PSp}_d(q)$ and let $G$ be an almost simple group with socle $G_0$. We fix a basis $e_1,\ldots,e_d,f_1,\ldots,f_d$ of $V=\mathbb{F}_q^d$ such that the symplectic form defining $G_0$ satisfies $(e_i,e_j)=(f_i,f_j)=0$ and $(e_i,f_j)=1$ if and only $i=j$.  Let $\mu$ be a generator of the multiplicative group of the field $\mathbb{F}_q$ and let $\delta$ be the projective image of $\mathrm{diag}(\mu,\dots,\mu,1,\dots,1)$.
 In particular, $G_0\unlhd G\le \mathrm{Aut}(G_0)$  and 
$$
\mathrm{Aut}(G_0)=
\begin{cases}
\mathrm{Sp}_d(q) \rtimes \mathrm{Aut}(\mathbb{F}_q)& \text{ if } q \text{ even,} \\
(\langle\mathrm{PSp}_d(q), \delta \rangle) \rtimes \mathrm{Aut}(\mathbb{F}_q)&\text{ if } q \text{ odd,}
	\end{cases}
$$
see \cite[Section~2.4]{KL90}. In the following, we set $A = \mathrm{Aut}(\mathrm{PSp}_d(q))$.

For $1\le k\le d/2$, we let $\Omega_k$ be the collection of all $k$-dimensional totally singular subspaces of $V$. Observe that, since we are in a symplectic  space, every $1$-dimensional subspace of $V$ is singular, hence $\Omega_1$ is the collection of all $1$-dimensional subspaces.

\begin{lemma}\label{psp:cased4}
	When $d=4$ and $q\ne 2$, there exist two positive integers $\ell,\ell'\ge 2$ with $\ell\ne\ell'$ and there exist $\omega_1,\ldots,\omega_\ell$, $\omega_1',\ldots,\omega_{\ell'}'$ elements of $\Omega_1$ such that
	\begin{itemize}
\item $(\omega_1,\ldots,\omega_\ell)$ and $(\omega_1',\ldots,\omega_{\ell'}')$ are irredundant bases for the action of $G$ on $\Omega_1$,
		\item $\omega_1+\cdots+\omega_\ell=V=\omega_{1}'+\cdots+\omega_{\ell'}'$, $\omega_1\cap\cdots\cap\omega_\ell=0=\omega_1'\cap\cdots\cap\omega_{\ell'}'$.
	\end{itemize}
In particular, the action of $G$ on the set of the totally singular subspaces of dimension $1$ is not $\mathrm{IBIS}$.

When $d=4$ and $q=2$, $G$ is $\mathrm{IBIS}$.
\end{lemma}
\begin{proof} 
When $d=4$ and $q=2$, the action of $G$ is IBIS, as proved in~\cite{LeSp}. Therefore, suppose $q\ne 2$.

	For $q=3$, the veracity of the lemma can be checked with the help of a computer. So we now suppose that $q>3$. 
	We recall that  we are using the symplectic form given by
	$$\begin{pmatrix}
		0&0&1&0\\
		0&0&0&1\\
		-1&0&0&0\\
		0&-1&0&0
	\end{pmatrix}$$
	and the base $e_1,e_2,f_1,f_2$ of $V=\mathbb{F}_q^4$.
	
	We give an irredundant base of cardinality $6$ for $G_0=\mathrm{PSp}_4(q)$ and an irredundant base of cardinality $5$ for $A=\mathrm{Aut}(G_0)$. From this, the result follows immediately from Lemma~\ref{lemma:minusminus1}.  Firstly, consider $\omega_1 = \langle e_1 \rangle$, $\omega_2 = \langle e_2 \rangle$, $\omega_3 = \langle f_1 \rangle$, $\omega_4 = \langle f_2 \rangle$, $\omega_5 = \langle e_1+e_2 \rangle$, $\omega_6 = \langle e_1 + f_1 \rangle$. A computation shows that 
	$$
	(G_0)_{\omega_1,\omega_2,\omega_3,\omega_4}=
	\left\{
	\begin{bmatrix}
		a&0&0&0\\
		0&b&0&0\\
		0&0&a^{-1}&0\\
		0&0&0&b^{-1}
	\end{bmatrix}\mid a,b\in\mathbb{F}_q \setminus \{0\}
	\right\}.$$
	Now an element $g \in (G_0)_{\omega_1,\omega_2,\omega_3,\omega_4}$ fixes $\omega_5$ if and only if $ae_1+be_2 = \alpha(e_1+e_2)$ for some $\alpha \in \mathbb{F}_q$. This implies that $a=b$.\\
	Finally, an element $g \in (G_0)_{\omega_1,\omega_2,\omega_3,\omega_4,\omega_5}$ fixes $\omega_6$ if and only if $ae_1+a^{-1}f_1 = \alpha(e_1+f_1)$  for some $\alpha \in \mathbb{F}_q$. This yields $a=a^{-1}$, that is, $a=\pm 1$. In conclusion, $\omega_1,\dots,\omega_6$ is an irredundant base of cardinality $6$ for $G_0$. Moreover, the sum of these subspaces is equal to $V$ and their intersection is equal to $0$. This implies that there exists $\ell \geq 6$ and an irredundant base $(\omega_1,\dots,\omega_\ell)$ for $G$ with $\omega_1+\cdots + \omega_\ell = V$ and $\omega_1\cap\cdots\cap\omega_\ell=0$.
	
	We now take in account $A$. Let $\alpha$ be a generator of the multiplicative group of the field $\mathbb{F}_q$. Take $\omega_1' = \omega_1, \dots, \omega_4' = \omega_4$ and $\omega_5' = \langle \alpha e_1 + e_2 + \alpha f_1 + f_2 \rangle$.  We deal simultaneously with the case $q$ odd and $q$ even, by defining $\delta=1$ when $q$ is even. A computation shows that
	\[
		A_{\omega_1',\dots,\omega_4'} = \left \{ \begin{bmatrix}
			a&0&0&0\\
			0&b&0&0\\
			0&0&a^{-1}&0\\
			0&0&0&b^{-1}
		\end{bmatrix}\mid a,b\in\mathbb{F}_q \setminus \{0\}\right\}\rtimes (\langle\delta\rangle\times \mathrm{Aut}(\mathbb{F}_q)).
	\]
Let $g\in A_{\omega_1',\ldots,\omega_4'}$ be an element fixing $\omega_5'$. Then, from the description of $A_{\omega_1',\ldots,\omega_4'}$ above, we see that there exist $a,b\in\mathbb{F}_q\setminus\{0\}$, $\varepsilon\in \{0,1\}$ and $\eta\in\mathrm{Aut}(\mathbb{F}_q)$ such that
	\[
\langle \alpha e_1+e_2+\alpha f_1+f_2\rangle=\omega_5'=\omega_5'^g=\langle \mu^\varepsilon\alpha^\eta a e_1+ \mu^\varepsilon b e_2+\alpha^\eta a^{-1}f_1+b^{-1}f_2\rangle.\]
We obtain the following system
	\[\left\{
		\begin{array}{ccc}
			\alpha c &=& \mu^\varepsilon  \alpha^\eta a \\
			c &=& \mu^\varepsilon b \\
			\alpha c &=& \alpha^\eta a^{-1} \\
			c &=& b^{-1},
		\end{array}\right.
	\]
for some $c\in\mathbb{F}_q\setminus\{0\}$. The second and the forth equations give $\mu^\varepsilon=b^2$ and this implies $\varepsilon=0$, because $\mu$ is a generator of the multiplicative group $\mathbb{F}_q$.	Therefore, $b^2=1$ and hence $c=b\in \{1,-1\}$. The first and the third equations give $a=a^{-1}$ and hence $a\in \{1,-1\}$. Now, the first equation gives $ac\alpha=\alpha^\eta $ and hence $\alpha^2=(\alpha^2)^{\eta}$. Therefore, $\alpha^2$ is invariant under $\eta$. Suppose, arguing by contradiction, that $\eta\ne 1$. Then the set of elements of $\mathbb{F}_q$ invariant under $\eta$ forms a proper subfield $\mathbb{F}$ of $\mathbb{F}_q$, with $|\mathbb{F}|\le \sqrt{q}$. On the other hand, since $\alpha$ is a generator of the multiplicative group of $\mathbb{F}_q$ and since $\alpha^2\in\mathbb{F}$, we get $|\mathbb{F}|\ge (q+1)/2$. However, the inequality $\sqrt{q}\ge (q+1)/2$ is not satisfied. This contradiction has shown that $\eta=1$. Now, the linear system above gives $a=b=c\in \{1,-1\}$ and hence $ A_{\omega_1',\ldots,\omega_4',\omega_5'}=1$.
\end{proof}

We deal with the case $d=4$ and $k=2$. This is very similar to the proof of Lemma~\ref{psu:cased=4k=2}.

\begin{lemma}\label{psp:cased=4k=2}
Let $d=4$. When $q\ne 2$, there exist two positive integers $\ell,\ell'\ge 2$ with $\ell\ne \ell'$ and there exist $\omega_1,\ldots,\omega_\ell,\omega_1',\ldots,\omega_{\ell'}'$ elements of $\Omega_2$ such that
\begin{itemize}
\item $(\omega_1,\ldots,\omega_\ell)$ and $(\omega_1',\ldots,\omega_{\ell'}')$ are irredundant bases for the action of $G$ on $\Omega_2$,
		\item $\omega_1+\cdots+\omega_\ell=V=\omega_{1}'+\cdots+\omega_{\ell'}'$, $\omega_1\cap\cdots\cap\omega_\ell=0=\omega_1'\cap\cdots\cap\omega_{\ell'}'$.
\end{itemize}
In particular, $G$ is not $\mathrm{IBIS}$ and $b(G)\ge 2$. When $q=2$, the action of $G$ on the totally singular $2$-dimensional subspaces of $V$ is $\mathrm{IBIS}$ and $b(G)\ge 2$.
\end{lemma}
\begin{proof}
We use the symplectic form given by the matrix
$$\begin{pmatrix}
0&0&1&0\\
0&0&0&1\\
-1&0&0&0\\
0&-1&0&0
\end{pmatrix}$$
and we use the basis $e_1,e_2,f_1,f_2$ of $V=\mathbb{F}_{q}^4$. Let $G_0$ be the socle of $G$ and assume, for the time being, that $q\notin\{ 2,3\}$.

Let $\omega_1=\langle e_1,e_2\rangle$, $\omega_2=\langle f_1,f_2\rangle$, $\omega_3=\langle e_1,f_2\rangle$, $\omega_4=\langle e_2,f_1\rangle$, $\omega_5=\langle e_1+e_2,f_1-f_2\rangle$ and $\omega_6=\langle e_1+f_2,e_2+f_1\rangle$. It is readily seen that $\omega_1,\ldots,\omega_6$ are $2$-dimensional totally singular subspaces of $V$. An easy computation shows that
$$(G_0)_{\omega_1,\omega_2,\omega_3,\omega_4}=
\left\{
\left[
\begin{array}{cccc}
a&0&0&0\\
0&b&0&0\\
0&0&a^{-1}&0\\
0&0&0&b^{-1}
\end{array}
\right]\mid a,b\in \mathbb{F}_{q}\setminus\{0\}
\right\}.$$
In particular, $|(G_0)_{\omega_1,\omega_2,\omega_3,\omega_4}|=(q-1)^2/\gcd(2,q-1)\ne 1$, because $q\ne 2$. We deduce$$(G_0)_{\omega_1,\omega_2,\omega_3,\omega_4,\omega_5}=
\left\{
\left[
\begin{array}{cccc}
a&0&0&0\\
0&a&0&0\\
0&0&a^{-1}&0\\
0&0&0&a^{-1}
\end{array}
\right]\mid a\in \mathbb{F}_{q}\setminus\{0\}
\right\}.$$
From this it follows that
$
|(G_0)_{\omega_1,\omega_2,\omega_3,\omega_4,\omega_5}|=(q-1)/\gcd(2,q-1)\ne 1$, because $q\ne 3$.
This show that $\omega_1,\ldots,\omega_6$ is an irredundant base for $G_0$ of cardinality $6$.

The argument above shows that
$$A_{\omega_1,\omega_2}=
\left(\left\{
\left[
\begin{array}{cc}
X&0\\
0&(X^{-1})^{tr}
\end{array}
\right]\mid X\in \mathrm{GL}_{2}(q)
\right\}\rtimes \langle\delta\rangle\right)\rtimes \mathrm{Aut}(\mathbb{F}_{q^2}).$$
Here by abuse of notation we are including the case $q$ even by setting $\delta=1$. Let $\omega_3'=\langle e_1,e_2+f_2\rangle$ and observe that $\omega_3'$ is a totally singular $2$-dimensional subspace of $V$. Using the description of $A_{\omega_1,\omega_2}$ and using the fact that $\delta$ is the projective image of $\mathrm{diag}(\mu,\mu,1,1)$ with $\mu$ a generator of the multiplicative group of the field $\mathbb{F}_q$, it follows that
$$A_{\omega_1,\omega_2,\omega_3'}=
\left\{
\left[
\begin{array}{cccc}
a&0&0&0\\
0&1&0&0\\
0&0&a^{-1}&0\\
0&0&0&1
\end{array}
\right]\mid a\in \mathbb{F}_{q}\setminus\{0\}
\right\}\rtimes\mathrm{Aut}(\mathbb{F}_q).$$
Let $\omega_4'=\langle e_1+e_2,f_1+f_2\rangle$ and observe that $\omega_4'$ is a totally singular $2$-dimensional subspace of $V$. Clearly,
$$A_{\omega_1,\omega_2,\omega_3',\omega_4'}=\mathrm{Aut}(\mathbb{F}_q).$$
Since $\mathrm{Aut}(\mathbb{F}_q)$ admits a regular orbit on its action on $\mathbb{F}_q$, the irredundant chain $\omega_1,\omega_2,\omega_3',\omega_4'$ can be extended to an irredundant base of cardinality $5$.
%Let $q=p^f$ and let $\alpha$ be a generator of the multiplicative group of the field $\mathbb{F}_{q}$ and let $\omega_5'=\langle e_1+\alpha f_2,e_2+\alpha f_1\rangle$. Observe that $\omega_5'$ is a $2$-dimensional totally isotropic subspace of $V$. Let $g\in A_{\omega_1,\omega_2,\omega_3,\omega_4,\omega_5'}$. Then there exists  $\eta\in\mathrm{Aut}(\mathbb{F}_{q})$ such that
%$$\omega_5'=\omega_5'^g=\langle e_1+\alpha^\eta f_2,e_2+\alpha^\eta f_1\rangle$$
%This yields $\alpha=\alpha^\eta$. As $\alpha$ is a field generator, $\eta$ is the identity and hence $g=1$.  This shows that $\omega_1,\omega_2,\omega_3',\omega_4,\omega_5'$ is an irredundant base for $A$ (and hence for $G$) of cardinality $5$. 
This implies that $G$ is not IBIS when $q\ge 4$. 

Suppose $q=2$. Here, $G=\mathrm{Sp}_4(2)'\cong\mathrm{Alt}(6)$ or $G=\mathrm{Sp}_4(2)\cong\mathrm{Sym}(6)$, and $G$ has degree $15$. We have verified with a computer that in both cases the group is IBIS, indeed, this example is in Table~\ref{table:table}.

Suppose $q=3$. Here, $G=\mathrm{PSp}_4(3)$ or $G=\mathrm{PSp}_4(3).2$, and the degree is $40$. We have verified with a computer that in both cases the group is not IBIS and it admits bases of different cardinality satisfying the additional conditions in the statement of the lemma.
\end{proof}

\begin{lemma}\label{psp:cased=6k=2}
Let $d=6$ and $q=2$ and let $k\in \{2,3\}$. There exist two positive integers $\ell,\ell'\ge 2$ with $\ell\ne \ell'$ and there exist $\omega_1,\ldots,\omega_\ell,\omega_1',\ldots,\omega_{\ell'}'$ elements of $\Omega_k$ such that
\begin{itemize}
\item $(\omega_1,\ldots,\omega_\ell)$ and $(\omega_1',\ldots,\omega_{\ell'}')$ are irredundant bases for the action of $G$ on $\Omega_k$,
		\item $\omega_1+\cdots+\omega_\ell=V=\omega_{1}'+\cdots+\omega_{\ell'}'$, $\omega_1\cap\cdots\cap\omega_\ell=0=\omega_1'\cap\cdots\cap\omega_{\ell'}'$.
\end{itemize}
In particular, $G$ is not $\mathrm{IBIS}$ on $\Omega_k$ and $b(G)\ge 2$.
\end{lemma}
\begin{proof}
This follows with a computer computation.
\end{proof}

\begin{proposition}\label{psp:cased>4}
	When $(q,k)\ne (2,1)$ and when $(d,q,k)\ne (4,2,2)$, the action of $G$ on $\Omega_k$ is not $\mathrm{IBIS}$ and $b(G)\ge 2$. When $(q,k)=(2,1)$ and when $(d,q,k)=(4,2,2)$, the action of $G$ on $\Omega_\kappa$ is $\mathrm{IBIS}$.
\end{proposition}
\begin{proof}
We first deal with the case $k=1$, this will set up the base case of an inductive argument.

	When $d=4$, the proof follows from Lemma~\ref{psp:cased4}; therefore, for the rest of the argument, we suppose $d> 4$. Assume $q>2$. Let $W$ be a non-degenerate $4$-dimensional subspace of $V$ and let 
	$$\Delta=\{\omega\mid \omega\le W,\dim_{\mathbb{F}_{q}}(\omega)=1\hbox{ and }\omega \hbox{ singular}\}.$$
	Let $G_{\Delta}$ be the stabilizer of $\Delta$ in $G$ and let $H$ be the permutation group induced by $G_\Delta$ on $\Delta$.
	
	Observe that $G_{\Delta}=G_{W}$ because the span of the singular vectors of $V$ contained in $W$ is $W$ itself. Moreover, $H$ is an almost simple primitive group having socle $\mathrm{PSp}_4(q)$ acting on the totally singular $1$-dimensional subspaces of $W$. From Lemma~\ref{psp:cased4}, there exist two positive integers $\ell,\ell'\ge 2$ with $\ell\ne\ell'$ and $\omega_1,\ldots,\omega_\ell$, $\omega_1',\ldots,\omega_{\ell'}'$ totally singular $1$-dimensional subspaces of $V$ such that
	\begin{itemize}
		\item $\omega_1+\cdots+\omega_\ell=W=\omega_{1}'+\cdots+\omega_{\ell'}'$,
		\item $H_{\omega_1,\ldots,\omega_\ell}=H_{\omega_1',\ldots,\omega_{\ell'}'}$.
	\end{itemize}
	Since $\omega_1+\cdots+\omega_\ell=W=\omega_{1}'+\cdots+\omega_{\ell'}'$, we deduce 
	$G_{\omega_1,\ldots,\omega_\ell},G_{\omega_1',\ldots,\omega_{\ell'}'}\le G_W=G_\Delta$. Now, since $H_{\omega_1,\ldots,\omega_\ell}=H_{\omega_1',\ldots,\omega_{\ell'}'}$, we deduce
	$$G_{\omega_1,\ldots,\omega_\ell}=G_{\omega_1',\ldots,\omega_{\ell'}'}.$$
	As $\ell\ne\ell'$, Lemma~\ref{lemma:1} implies that $G$ is not IBIS.

Assume now $q=2$. Since each element of $V$ is totally singular, the action of $G=\mathrm{Sp}_d(2)$ on the totally singular $1$-dimensional subspaces of $V$ is the natural action of $G$ on the non-zero vectors of $V$. This action is IBIS, see~\cite{LeSp}, and is reported in Table~\ref{table:table}. This completes the case $k=1$.

\medskip

Suppose that $k>1$. 
Let $U$ be a totally singular subspace of $V$ with $\dim_{\mathbb{F}_{q}}(U)=k-1$ and let
\begin{align*}
\Delta=\{W\mid U\le W, \mathrm{dim}_{\mathbb{F}_{q}}(W)=k \hbox{ and }W \hbox{ totally singular}\}.\end{align*}
Let $H$ be the permutation group induced by $G$ on $\Delta$.

Observe that $G_\Delta=G_{U}$, where $G_U$ is the setwise stabilizer of $U$ in $G$. The symplectic form $\kappa$ defining the symplectic group $G_0$ induces a form $\kappa'$ on the $\mathbb{F}_{q}$-vector space $U^\perp/U$. Observe that $\dim_{\mathbb{F}_{q}}(U^\perp/U)=d-2k+2$ and that the form induced on $U^\perp/U$ is non-degenerate and symplectic.  Let $\bar V=U^\perp/U$ and consider $$\Omega=\{W'\le \bar V\mid\mathrm{dim}_{\mathbb{F}_{q}}(W')=1 \hbox{ and } W'\hbox{ is totally singular for }\kappa'\}.$$
Clearly, there exists a one-to-one correspondence between the elements of $\Delta$ and the elements of $\Omega$; moreover, this correspondence allows to identify the action of $H$ on $\Delta$ with the natural action on the $1$-dimensional totally singular subspaces of $\bar V$.

Suppose $d-2k+2\ge 4$ and $q\ne 2$.
From Lemma~\ref{psp:cased4}, the action of $H$ on $\Delta$ is not IBIS and hence there exist two irredundant chains $(\delta_1,\ldots,\delta_{\ell})$ and $(\delta_1',\ldots,\delta_{\ell'}')$ of points of $\Delta$ such that $\ell,\ell'\ge 2$, $\ell\ne \ell'$ and
\begin{equation}\label{eq:nasty1}
G_{\delta_1,\ldots,\delta_{\ell}}\cap G_{\Delta}=G_{\delta_1',\ldots,\delta_{\ell'}'}\cap G_\Delta.
\end{equation}
Since $\ell\ge 2$, $\delta_1$ and $\delta_2$ are two distinct $k$-dimensional subspaces of $V$ containing $U$ and hence 
$$\delta_1\cap \delta_2=U.$$
This shows that $G_{\delta_1,\delta_2}=G_{\delta_1}\cap G_{\delta_2}\le G_{U}=G_{\Delta}$ and hence $G_{\delta_1,\ldots,\delta_{\ell}}\le G_{\Delta}$. Since this argument applies also for $\ell'$, from~\eqref{eq:nasty1}, we deduce 
$$G_{\delta_1,\ldots,\delta_{\ell}}=G_{\delta_1',\ldots,\delta_{\ell'}'}.$$
As $\ell\ne \ell'$, from Lemma~\ref{lemma:1}, we deduce that the action of $G$ on $\Omega_k$ is not IBIS. It remains to consider the case $d-2k+2\le 2$ (that is, $d=2k$), and the case $q=2$.

Suppose $d=2k$.  
Let $U$ be a totally singular subspace of $V$ with $\dim_{\mathbb{F}_{q}}(U)=k-2$ and let
\begin{align*}
\Delta=\{W\mid U\le W, \mathrm{dim}_{\mathbb{F}_{q}}(W)=k \hbox{ and }W \hbox{ totally singular}\}.\end{align*}
We may apply the argument above verbatim, using Lemma~\ref{psp:cased=4k=2} instead of Lemma~\ref{psp:cased4}, with the old $U$ replaced by this new $U$. We deduce that $G$ is not IBIS. It remains to deal with the case $q=2$.

\medskip

Suppose  $q=2$ and $k\in \{2,3\}$. Clearly, we may suppose that $d\ge 6$. Let $W$ be a non-degenerate $6$-dimensional subspace of $V$ and let 
	$$\Delta=\{\omega\mid \omega\le W,\dim_{\mathbb{F}_{q}}(\omega)=k\hbox{ and }\omega \hbox{ singular}\}.$$
	Let $G_{\Delta}$ be the stabilizer of $\Delta$ in $G$ and let $H$ be the permutation group induced by $G_\Delta$ on $\Delta$.
	
	Observe that $G_{\Delta}=G_{W}$ because the span of the $k$-dimensional singular subspaces of  $V$ of dimension $k$ contained in $W$ is $W$ itself. Moreover, $H$ is an almost simple primitive group having socle $\mathrm{PSp}_4(q)$ acting on the totally singular $k$-dimensional subspaces of $W$. From Lemma~\ref{psp:cased=6k=2}, there exist two positive integers $\ell,\ell'\ge 2$ with $\ell\ne\ell'$ and $\omega_1,\ldots,\omega_\ell$, $\omega_1',\ldots,\omega_{\ell'}'$ totally singular $k$-dimensional subspaces of $V$ such that
	\begin{itemize}
		\item $\omega_1+\cdots+\omega_\ell=W=\omega_{1}'+\cdots+\omega_{\ell'}'$,
		\item $H_{\omega_1,\ldots,\omega_\ell}=H_{\omega_1,\ldots,\omega_{\ell'}'}$.
	\end{itemize}
	Since $\omega_1+\cdots+\omega_\ell=W=\omega_{1}'+\cdots+\omega_{\ell'}'$, we deduce 
	$G_{\omega_1,\ldots,\omega_\ell},G_{\omega_1',\ldots,\omega_{\ell'}'}\le G_W=G_\Delta$. Now, since $H_{\omega_1,\ldots,\omega_\ell}=H_{\omega_1,\ldots,\omega_{\ell'}'}$, we deduce
	$$G_{\omega_1,\ldots,\omega_\ell}=G_{\omega_1,\ldots,\omega_{\ell'}'}.$$
	As $\ell\ne\ell'$, Lemma~\ref{lemma:1} implies that $G$ is not IBIS.

Suppose now $q=2$ and $k\ge 4$.
 Again, we may let $U$ be a totally singular subspace of $V$ with $\dim_{\mathbb{F}_{q}}(U)=k-3$ and let
\begin{align*}
\Delta=\{W\mid U\le W, \mathrm{dim}_{\mathbb{F}_{q}}(W)=k \hbox{ and }W \hbox{ totally singular}\}.\end{align*}
We may apply the argument above verbatim, with the old $U$ replaced by this new $U$. We deduce that $G$ is not IBIS.
\end{proof}

\subsection{The Klein correspondence}\label{sec:klein}
In this section we make a brief digression on the \textit{\textbf{Klein correspondence}} in geometry. For this purpose, we follow~\cite[Chapter~8]{PolarSpaces}. This will allow us to use the previous results on the unitary and symplectic groups to deal with the orthogonal groups.

Let $W$ be the $\mathbb{F}_q$-vector space of all $4\times 4$ skew-symmetric matrices. Clearly, $\dim_{\mathbb{F}_q} W=6$, because the entries of an element of $W$ above the diagonal uniquely determine the matrix. Let $Q:W\to\mathbb{F}_q$ be the quadratic form defined by
$$X\mapsto Q(X)=\mathrm{Pf}(X),$$
where $\mathrm{Pf}(X)$ is the \textit{\textbf{Pfaffian}} of the matrix $X$. In particular, if
\[
\begin{pmatrix}
0&x_{12}&x_{13}&x_{14}\\
-x_{12}&0&x_{23}&x_{24}\\
-x_{13}&-x_{23}&0&x_{34}\\
-x_{14}&-x_{24}&-x_{34}&0
\end{pmatrix},
\]
then  $$\mathrm{Pf}(X)=x_{12}x_{34}-x_{13}x_{24}+x_{14}x_{23}.$$

Clearly, $Q$ has Witt defect $0$. The group $\mathrm{GL}_4(q)$ acts as a linear group on $W$ by 
$$X\mapsto P^T XP,$$
where $X\in W$, $P\in\mathrm{GL}_4(q)$ and $P^T$ is the transpose of the matrix $P$.
Using the properties of the Pfaffian, we have
$\mathrm{Pf}(P^T XP)=\det(P)\mathrm{Pf}(X)$. Therefore, $P$ preserves the quadratic form $Q$ if and only if $\det(P)=1$. Since $\mathrm{SL}_4(q)$ equals its own derived subgroup, we deduce 
$\mathrm{SL}_4(q)\le \Omega_6^+(q)$. Observe that we are identifying here $\mathrm{SL}_4(q)$ with a subgroup of $\mathrm{GL}_6(q)$ using the action of $\mathrm{SL}_4(q)$ on $W$. Moreover, from this embedding, we deduce $\mathrm{PSL}_4(q)\le \mathrm{P}\Omega_6^+(q)$. Comparing the orders of $\mathrm{PSL}_4(q)$ and of $\mathrm{P}\Omega_6^+(q)$, we get $\mathrm{PSL}_4(q)=\mathrm{P}\Omega_6^+(q)$.

We now deduce some geometric facts. Let $0\ne X\in W$ be singular with respect to $Q$. This means $Q(X)=0$. Since $\det(X)$ is the square of $\mathrm{Pf}(X)=Q(X)$, we deduce that $\det(X)=0$. Since $X$ is skew-symmetric, we deduce that the rank of $X$ is even. As $X\ne 0$, we deduce that $X$ has rank $2$. It is not hard to verify that a skew-symmetric matrix of rank $2$ is of the form
$$X=v^T w-w^T v,$$
for some vectors $v, w \in \mathbb{F}_q^4$. Observe that $v$ and $w$ are linearly independent over $\mathbb{F}_q$, because otherwise $X\ne 0$. This sets up a bijection between the points on the quadric $Q$ and the $2$-dimensional subspaces of $\mathbb{F}_q^4$.

We summarise in the following proposition a few facts that we need later. A \textit{\textbf{plain pencil}} in $\mathbb{F}_q^4$ is an incident point-plane pair, that is, a pair $\{A,B\}$ where $A$ and $B$ are subspaces of $\mathbb{F}_q^4$ with $A<B$, $\dim_{\mathbb{F}_{q}} A=1$ and $\dim_{\mathbb{F}_{q}} B=3$. Given a plain pencil $\{A,B\}$, the \textit{\textbf{plain pencil of lines}} of $\{A,B\}$ is the collection of all the $2$-dimensional subspaces of $B$ containing $A$.
\begin{proposition}\label{klein1}
\begin{description}
\item[A] There exists a natural one-to-one correspondence between the $1$-dimensional singular subspaces of $\mathbb{F}_q^6\cong W$ with respect to $Q$ and the $2$-dimensional subspaces of $\mathbb{F}_q^4$.
\item[B]  Under this correspondence, lines in $Q$, that is, $2$-dimensional subspaces of $W=\mathbb{F}_q^6$ which are totally singular for $Q$, correspond to plane pencils of lines in $\mathbb{F}_q^4$.
\item[C] Under this correspondence, planes in $Q$, that is, $3$-dimensional subspaces of $W=\mathbb{F}_q^6$ which are totally singular for $Q$, correspond to maximal families of pairwise intersecting lines in $\mathbb{F}_q^4$. There are two types: all lines through a fixed point and all lines in a fixed plane.
\end{description}
\end{proposition}

\subsection{Orthogonal groups of plus type}\label{sec:orthogonal+}

Let $d\ge 6$ be an even positive integer, let $q=p^f$ be a prime power with $p$ a prime number, let $G_0=\mathrm{P}\Omega_d^+(q)$ and let $G$ be an almost simple group with socle $G_0$.
We now describe the primitive subspace actions.

 We let $V=\mathbb{F}_q^d$ be the $d$-dimensional vector space over the finite field $\mathbb{F}_{q}$ of cardinality $q$. Then, for $1\le k\le d/2$, we let $\Omega_k$ be the collection of all $k$-dimensional totally singular subspaces of $V$, with respect to the quadratic form $Q$ preserved by $G_0$. When $k=d/2$, the collection $\Omega_{d/2}$ of all maximal totally isotropic subspaces of $V$ splits into two disjoint families, $\Omega_{d/2}=\Omega_{d/2,1}\cup\Omega_{d/2,2}$, namely the \textit{\textbf{greeks}} and the \textit{\textbf{latins}} where two elements  $X$ and $Y$ of $\Omega_{d/2}$ are in the same family if and only if $X\cap Y$ has even codimension in either $X$ or $Y$, for details we refer to~\cite{PolarSpaces}.

\begin{lemma}\label{ort+}The action of $G$ on $\Omega_k$, or $\Omega_{d/2,1}$, $\Omega_{d/2,2}$, is $\mathrm{IBIS}$ if and only if $(d,q,k)=(6,2,3)$.
\end{lemma}
\begin{proof}
When $d=6$, using the Klein correspondence, we are reduced to the primitive actions of almost simple groups with socle $\mathrm{PSL}_4(q)$ on $1$-, $2$-, $3$-dimensional subspaces of $\mathbb{F}_q^4$, and on pairs $\{A,B\}$ of subspaces with $A<B$ and $\dim B=3$, $\dim A=1$. The result in this case follows from Lemmas~\ref{prop:2},~\ref{prop:3} and~\ref{refdaaggiustareeepart1}. Therefore, we suppose $d\ge 8$. Let $W$ be a non-degenerate subspace of $V$ with $\dim W=6$ and such that the quadratic form $Q$ restricts to a quadratic form having Witt defect $0$ on $W$. In particular, the quadratic form $Q$ restricts also to $W^\perp$ to a quadratic form having Witt defect $0$.

Suppose that $k\le d/2-2$.   Let $U$ be a totally isotropic subspace of $W^\perp$ of dimension $k-1$. Let
$$\Delta=\{U'\oplus U\mid U'\le W, \dim (U'\oplus U)=k, U' \textrm{ totally isotropic}\}.$$ 
Clearly, $G_\Delta=G_{W,U}$. Using the Klein correspondence and the results on the subspace actions of $\mathrm{PSL}_4(q)$ (see Lemma~\ref{prop:3} applied with $k=2$), we  deduce that $G$ is not IBIS.

Suppose that $k= d/2-1$.   Let $U$ be a totally isotropic subspace of $W^\perp$ of dimension $k-2$, observe that $U$ is a maximal totally isotropic subspace of $W^\perp$. Let
$$\Delta=\{U'\oplus U\mid U'\le W, \dim (U'\oplus U)=k, U' \textrm{ totally isotropic}\}.$$ 
Clearly, $G_\Delta=G_{W,U}$. Before using the Klein correspondence we need to make an additional remark. Observe that the orthogonal decoposition $V=W\perp W^\perp$ gives rise to an embedding of $\mathrm{SO}_6^+(q)\times\mathrm{SO}_{d-6}^+(q)$ into $\mathrm{SO}_d^+(q)$. From~\cite{KL90}, we see that $(\mathrm{SO}_6^+(q)\times\mathrm{SO}_{d-6}^+(q))\cap \Omega_d^+(q)$ projects surjectively onto $\mathrm{SO}_6^+(q)$. In particular, the action of $G_{\Delta}$ on the $2$-dimensional totally singular subspaces of $W$ induces an almost simple group containing $\mathrm{PSO}_6^+(q)$. Now, we can apply the Klein correspondence because the associated group in the action on pairs of subspaces contains $\mathrm{PSL}_4(q)\langle\iota\rangle$, where $\iota$ is the inverse transpose automorphism. Now, using Lemma~\ref{refdaaggiustareeepart1}, we  deduce that $G$ is not IBIS.

Suppose that $k= d/2$. Without loss of generality, we may suppose that $G$ is acting on $\Omega_{d/2,1}$. Let $U$ be a totally isotropic subspace of $W^\perp$ of dimension $k-3$ and let $U''$ be a $3$-dimensional totally singular subspace of $W$ with $U'\oplus U\in\Omega_{d/2,1}$. Let
\begin{align*}
\Delta=\{U'\oplus U\mid& U'\le W, \dim (U'\oplus U)=k,\\
& U' \textrm{ totally isotropic and}\dim(U'/(U'\cap U'')) \textrm{ even}\}.
\end{align*}
Clearly, $G_\Delta=G_{W,U}$. What is more, the condition that $\dim(U'/(U'\cap U''))$ is even guarantees that every element of $\Delta$ is in $\Omega_{d/2,1}$.   Using the Klein correspondence and Lemma~\ref{prop:2}, we deduce that $G$ is not IBIS, except when $q=2$.

Assume $k=d/2$ and $q=2$. Let $U$ be a totally isotropic subspace of $V$ with $\dim U=k-4$ and consider $\Delta=\{\omega\in\Omega\mid U\le \omega\}$. We have a one-to-one correspondence between the elements of $\Delta$ and the $4$-dimensional totally isotropic subspaces of $U^\perp/U$ falling in the same family. Therefore, the result for an arbitrary value of $d$ follows from the analogous result for $d=8$. We have checked with a computer that when $d=8$, the action of $G$ on $\Omega_{4,1}$ and on $\Omega_{4,2}$ is not IBIS.
\end{proof}

\subsection{Odd dimensional orthogonal groups}\label{sec:orthogonalodd}
To deal with this case we continue our analysis on the Klein correspondence. Recall that we have an embedding of $\mathrm{SL}_4(q)$ in $\Omega_6^+(q)$.

Let $A$ be a non-singular skew-symmetric matrix. For the way that the action of $\mathrm{GL}_4(q)$ was defined on $W$, the stabilizer of $A$ in $\mathrm{SL}_4(q)$ consists of the matrices $P$ such that $P^TAP=A$. Therefore, this stabilizer is the symplectic group defined by the matrix $A$.

On the other hand, $A$ can be viewed as a vector in $W$ with $Q(A)\ne 0$ and so the stabilizer of $A$ viewed in the orthogonal group $\Omega_6^+(q)$ is the $5$-dimensional orthogonal group on $A^\perp$. Thus 
$\Omega_5(q)\cong \mathrm{PSp}_4(q)$.
Under this correspondence the orthogonal geometry and the symplectic geometry are linked together.

\begin{proposition}\label{klein2}
\begin{description}
\item[A] There exists a natural one-to-one correspondence between the $1$-dimensional singular subspaces of $A^\perp\le \mathbb{F}_q^6\cong W$ with respect to $Q$ and the $2$-dimensional subspaces of $\mathbb{F}_q^4$ which are totally singular for $\mathrm{Sp}_4(q)$.
\item[B]  Under this correspondence, lines in $Q$ contained in $A^\perp$  correspond to $1$-dimensional subspaces of  $\mathbb{F}_q^4$.
\end{description}
\end{proposition}

Let $d$ be an odd integer with $d\ge 5$, let $q$ be an odd prime power and let $k$ be an integer with $1\le k\le (d-1)/2$. We let $G$ be an almost simple group with socle $\Omega_d(q)=\mathrm{P}\Omega_d(q)$ and let $\Omega_k$ be the collection of all totally isotropic subspaces of $\mathbb{F}_q^d$ of dimension $k$, with respect to the quadradic form preserved by $\Omega_d(q)$. We consider the primitive action of $G$ on $\Omega_k$.

\begin{lemma}\label{oddorthogonal}The action of $G$ on $\Omega_k$ is not $\mathrm{IBIS}$.
\end{lemma}
\begin{proof}
We dismiss the notation established to describe the Klein correspondence and we use Proposition~\ref{klein2} directly. Therefore, as usual, we let $V=\mathbb{F}_q^d$, $G_0$ be the socle of $G$ and $Q$ be the quadratic form on $V$ preseved by $G_0$.

Let $W$ be a non-degenerate subspace of $V$ with $\dim W=5$ and such that the quadratic form $Q$  restricts to a non-degenerate quadratic form having Witt defect 0 on $W^\perp$. Now, when $k=(d-1)/2$, let $U$ be a totally isotropic subspace of $W^\perp$ of dimension $k-2$; when $k< (d-1)/2$, let $U$ be a totally isotropic subspace of $W^\perp$ of dimension $k-1$. Finally, let
$$\Delta=\{U'\oplus U\mid U'\le W, \dim (U'\oplus U)=k, U' \textrm{ totally isotropic}\}.$$ 

Clearly, $G_\Delta=G_{W,U}$. Using the fact that $q$ is odd and using the Klein correspondence and the results on the subspace actions of $\mathrm{PSp}_4(q)$ (see Proposition~\ref{psp:cased>4}), we deduce that $G$ is not IBIS.
\end{proof}

\subsection{Orthogonal groups of minus type}\label{sec:orthogonalminus}
The Klein correspondence allows to identify $\mathrm{P}\Omega_6^-(q)$ with $\mathrm{PSU}_4(q)$, details can be found in~\cite{PolarSpaces}. Under this identitication the orthogonal geometry and the Hermitian geometry are linked together. When $k\in \{1,2\}$, the $k$-dimensional totally isotropic subspaces for the orthogonal group correspond to the $k$-dimensional totally singular subspaces for the Hermitian group.

Let $d$ be an even integer with $d\ge 6$, let $q$ be a prime power and let $k$ be an integer with $1\le k\le d/2-1$. We let $G$ be an almost simple group with socle $\mathrm{P}\Omega_d^-(q)$ and let $\Omega_k$ be the collection of all totally isotropic subspaces of $\mathbb{F}_q^d$ of dimension $k$, with respect to the quadradic form $Q$ preserved by $\Omega_d(q)$. We consider the primitive action of $G$ on $\Omega_k$.

\begin{lemma}\label{evenminusorthogonal}The action of $G$ on $\Omega_k$ is not $\mathrm{IBIS}$.
\end{lemma}
\begin{proof}
Let $W$ be a non-degenerate subspace of $V$ with $\dim W=6$ and such that the quadratic form $Q$ restricts to a non-degenerate quadratic form having Witt defect 0 on $W^\perp$ and to a non-degenerate quadratic form having Witt defect 1 on $W$. Now, when $k=d/2-1$, let $U$ be a totally isotropic subspace of $W^\perp$ of dimension $k-2$; when $k< d/2-1$, let $U$ be a totally isotropic subspace of $W^\perp$ of dimension $k-1$. Finally, let
$$\Delta=\{U'\oplus U\mid U'\le W, \dim (U'\oplus U)=k, U' \textrm{ totally isotropic}\}.$$ 

Clearly, $G_\Delta=G_{W,U}$. Using the Klein correspondence and the results on the subspace actions of $\mathrm{PSU}_4(q)$ (see Proposition~\ref{psu:cased>3}), we deduce that $G$ is not IBIS.
\end{proof}

\section{Even characteristic symplectic groups: case~\eqref{HiHo4}}\label{symplecticeven}

In this section we deal with the actions listed at item~\eqref{definition6.2:2} of Definition~\ref{defintion6.2}. In particular $G$ is almost simple group with socle $G_0=\mathrm{Sp}_{2m}(q)'$ and $q$ a power of $2$, with $G \leq \mathrm{P}\Gamma\mathrm{L}_{2m}(q)$. Consulting \cite{bhr, KL90}, it is clear that if $H$ is a maximal subgroup of $G$ not containing $G_0$, and $H\cap G_0$ is a subgroup of $M=\mathrm{O}_{2m}^\pm(q)$, then $H\cap G_0=M$. In light of this, we may first consider the action of $G_0=\mathrm{Sp}_{2m}(q)$ on the set of right cosets of $M=\mathrm{O}_{2m}^\pm(q)$.\footnote{Details about the embedding of $\mathrm{O}_{2m}^\pm(q)$ in $\mathrm{Sp}_{2m}(q)$ when $q$ is even can be found in~\cite{Dye}.}

Our treatment follows~\cite{GiLoSp}, which in turn was inspired by~\cite[\S 7.7]{dixon_mortimer} where the case of $\mathrm{Sp}_{2m}(2)$ was considered. Let $e$ and $m$ be positive integers, let $q=2^e$, let $\mathbb{F}_q$ be the finite field with $q$ elements, and let $V=\mathbb{F}_q^{2m}$ be the $2m$-dimensional vector space of row vectors over $\mathbb{F}_q$. To start with we adjust notation slightly, and assume that $G=\mathrm{Sp}_{2m}(q)$ is the symplectic group defined by the symmetric matrix
\[
f=\begin{pmatrix}
0&I\\
I&0
\end{pmatrix},
\] 
where $0$ and $I$ are the zero and identity $m\times m$-matrices, respectively. In particular, $G$ is the group of invertible matrices preserving the bilinear form $\varphi:V\times V\to \mathbb{F}_q$ defined by $$\varphi(u,v)=ufv^{T},$$ for every $u,v\in V$, that is
\[ G = \left\{g \in \mathrm{GL}_{2m}(q) \mid gfg^{T}=f\right\}.
\]
Note that the bilinear form $\varphi$ is alternating, i.e. for all $u \in V$, we have
\begin{equation}\label{eq:blabla-1}
\varphi(u,u)=0.
\end{equation}
Moreover, since $\mathbb{F}_q$ is of characteristic 2, the form $\varphi$ is symmetric, i.e. for all $u,v \in V$, we have
\begin{equation}\label{eq:blabla-3}
\varphi(u,v)= \varphi(v,u).
\end{equation}
Now we let $\Omega$ be the set of quadratic forms $\theta:V\to\mathbb{F}_q$ polarising to $\varphi$. Recall that this means that $\theta:V\to\mathbb{F}_q$ is a function satisfying
\begin{itemize}
\item $\theta(u+v)-\theta(u)-\theta(v)=\varphi(u,v)$ for every $u,v\in V$, and
\item $\theta(cu)=c^2u$, for every $c\in\mathbb{F}_q$ and $u\in V$.
\end{itemize}

Next, consider the matrix 
\[
e=\begin{pmatrix}
0&I\\
0&0
\end{pmatrix}
\]
and the quadratic form $\theta_0:V\to\mathbb{F}_q$ defined by $$\theta_0(u)=ueu^T,$$
for every $u\in V$. Then, for every $u,v\in V$, a computation shows that
\begin{equation}\label{eq:blabla-2}
\begin{split}
\theta_0(u+v)-\theta_0(u)-\theta_0(v)&=\varphi(u,v). 
\end{split}
\end{equation}

In particular, \(\theta_0\) is a quadratic form whose polarisation is the symplectic form $\varphi$. This shows that $\theta_0\in \Omega$.

Let $\theta\in \Omega$ and define $\lambda=\theta-\theta_0$. We have
\begin{align*}
\lambda(u+v)=&\lambda(u)+\lambda(v),\\
\lambda(cu)=&c^2\lambda(u),
\end{align*}
for every $u,v\in V$ and for every $c\in\mathbb{F}_q$. Therefore, since $\mathbb{F}_q$ is of characteristic $2$, the function $\lambda:V\to\mathbb{F}_q$ is semilinear and hence there exists a unique $b\in V$ such that
$$\lambda(u)=(u\cdot b^T)^2,\,\, \hbox{for every }u\in V.$$ Since $f$ is an invertible matrix, there exists a unique $a\in V$ with $b=af$ and hence
$$\lambda(u)=(ufa^T)^2=\varphi(u,a)^2,$$
for every $u\in V$. Summing up, an arbitrary element of $\Omega$ is of the form
$$u\mapsto \theta_0(u)+\varphi(u,a)^2,$$
where $a\in V$. We denote this element of $\Omega$ simply by $\theta_a$. Thus
\begin{equation}\label{eq:blabla}\theta_a(u)=\theta_0(u)+\varphi(u,a)^2,\quad\textrm{for every }u\in V.\end{equation} 
In particular, the elements of $\Omega$ are parametrised by the vectors of $V$. Moreover, if $\theta_{a}=\theta_{a'}$ for some $a, a' \in V$, then $\theta_{a}(u)=\theta_{a'}(u)$ for every $u \in V$ and this implies $\varphi(u, a)= \varphi(u, a')$ for every $u \in V$. Since $\varphi$ is non-degenerate, we obtain $a=a'$. Hence, the set $\Omega$ is in one-to-one correspondence with $V$. This, in particular, yields that $$|\Omega|= q^{2m}.$$

%We choose $u = e_1$. Let $a = \sum_{i=1}^{2m}{a_i e_i}$ and $a' = \sum_{i=1}^{2m}{a'_i e_i}$, with $a_i, a'_i \in \Fq$. Then $\varphi(e_1, a)= \varphi(e_1, a')$ implies $a_m = a'_m$. By choosing $u=e_i$ for $i = 1, \dots, 2m$, we obtain $a_i = a'_i$ for $i=1, \dots, 2m$ and hence $a= a'$. This shows that the elements of $V$ are in one-to-one correspondence with the elements of $\Omega$. In particular $\left|\Omega\right|= q^{2m}$.
\begin{lemma}[Lemma~6.8,~\cite{GiLoSp}]
For every $\theta\in \Omega$ and $x\in G$, the mapping 
\begin{equation}
\begin{split}\label{eq:blabla1}
\theta^x &\colon V\to\mathbb{F}_q\\
& u \mapsto \theta(ux^{-1})
\end{split}
\end{equation}lies in $\Omega$. In particular, this defines an action of the group $G$ on the set $\Omega$.
\end{lemma}
Before continuing our discussion, we gather some information on $G$. Let $a\in V$, we define the mapping 

\begin{equation}\label{eq:blabla2}
\begin{split} 
t_a \colon &V\to V\\
&u \mapsto u+\varphi(u,a)a.
\end{split}
\end{equation}
Such a function is called a \textit{\textbf{transvection}}. For every $u,v \in V$ and $c \in \mathbb{F}_q$ we have
\begin{align*}
(u+v)t_a & = (u)t_a +(v)t_a; \\
(cu)t_a & = c(u) t_a.
\end{align*}
Hence $t_a$ is linear. Moreover, a computation shows that for every $u\in V$ we have
\begin{align*}
(u)t_a^2& = u,
\end{align*}
so that $t_a$ is an involution. Finally, for every $u,v\in V$, we have
\begin{align*}
\varphi((u)t_a,(v)t_a)&=\varphi(u,v).
\end{align*}
Therefore $t_a$ preserves $\varphi$ and hence lies in the symplectic group $G$.

We are now interested in computing the image of $\theta_a$ under the transvection $t_c$. First recall that, in a field of characteristic 2, since $x= -x$ for every $x \in \mathbb{F}_q$, the square root $\sqrt{\cdot} \colon \mathbb{F}_q \to \mathbb{F}_q$ is a well-defined map. Moreover, for every $a,b,x,y \in \mathbb{F}_q$ such that $x=a^2$ and $y=b^2$ we have $(\sqrt{x}+\sqrt{y})^2= (a+b)^2= a^2+b^2= x+y$, which implies $\sqrt{x}+\sqrt{y}=\sqrt{x+y}$. Moreover, recall that $\theta_a$ is a quadratic form polarising to $\varphi$ and that $t_a$ is an involution, in particular $t_c=t_c^{-1}$. By using these facts, it can be seen that, given $v \in V$, we have
\begin{align*}
\theta_a^{t_c}(u)&=\theta_{a+(\sqrt{\theta_a(c)}+1)c}(u).
\end{align*}
Hence
\begin{equation}\label{eq:0}
\theta_a^{t_c}=
\theta_{a+(\sqrt{\theta_a(c)}+1)c}.
\end{equation}

We now recall some facts about Galois theory. For a reference see \cite[Chapter VI]{lang}. The Frobenius mapping $\phi \colon x \mapsto x^2$ from $\mathbb{F}_q$ to itself is a generator of the Galois group of $\mathbb{F}_q$ over $\mathbb{F}_2$. There exists a well-defined $\mathbb{F}_2$-linear \textit{\textbf{trace}} mapping $\mathrm{Tr}:\mathbb{F}_q\to\mathbb{F}_2$. In what follows, we need only two basic facts about $\mathrm{Tr}$: first, $\mathrm{Tr}$ is surjective and second, from Hilbert's Theorem 90,  the kernel of $\mathrm{Tr}$ consists of the set $\{x^2+x\mid x\in\mathbb{F}_q\}$ and has cardinality $q/2$.

Define
\begin{align}\label{defineomega}
\Omega^+=\{\theta_a\mid \mathrm{Tr}(\theta_0(a))=0\},\\
\Omega^-=\{\theta_a\mid\mathrm{Tr}(\theta_0(a))=1\}.\nonumber
\end{align}

Observe that the above definition is a generalization of the definition of $\Omega^+$ and $\Omega^-$ in \cite[Corollary 7.7 A]{dixon_mortimer}. Indeed, if $q=2$, then the Galois group is the trivial group and hence the trace map is the identity.

%\color{red} 
%Can we prove that the sets $\Omega^+$ and $\Omega^-$ are the sets of quadratic forms polarizing to $\varphi$ of a certain type?
%\color{black}

\begin{lemma}[Lemma~6.9,~\cite{GiLoSp}]\label{report} The sets $\Omega^+$ and $\Omega^-$ are $G$-orbits on $\Omega$, with
$$|\Omega^+|=\frac{q^m(q^m+1)}{2},\quad |\Omega^-|=\frac{q^m(q^m-1)}{2}.$$
Moreover, the action of $G$ on $\Omega^\pm$ is permutation equivalent to the action of $G$ on the right cosets of $\mathrm{O}_{2m}^\pm(q)$.
\end{lemma}

Now, we select a distinguished element of $\Omega^-$. Let $\epsilon\in\mathbb{F}_q$ with $\mathrm{Tr}(\epsilon)=1$ and set $\varepsilon=\epsilon e_1+e_{m+1}$, where $(e_i)_{i\in \{1,\ldots,2m\}}$ is the standard basis of $V$. Since $\theta_0(\epsilon e_1)=0=\theta_0(e_{m+1})$, we have $$\theta_0(\varepsilon)=\theta_0(\epsilon e_1)+\theta_0(e_{m+1})+\varphi(\epsilon e_1,e_{m+1})=\epsilon\varphi(e_1,e_{m+1})=\epsilon,$$ and hence $\mathrm{Tr}(\theta_0(\varepsilon))=\mathrm{Tr}(\epsilon)=1$. Therefore, $\theta_\varepsilon\in \Omega^-$.

We now compute stabilizer chains in these two actions. Let $\circ\in \{+,-\}$ (we deal simultaneously with both cases). % {\color{red} \textbf{I'm worried that $\varepsilon$ was used above for a different purpose.}}
Let $\{\theta_{a_1},\theta_{a_2},\ldots,\theta_{a_k}\}$ be a subset of $\Omega^\circ$. 
Without loss of generality, we may suppose that $a_1=0$ when $\circ=+$ and $a_1=\varepsilon$ when $\circ=-$. 

Let us define
$${\bf C}_{G_{\theta_{a_1}}}(a_1+a_i)=\{x\in G_{\theta_{a_1}}\mid (a_1+a_i)x=a_1+a_i\},$$
that is the set of matrices in $G_{\theta_{a_1}}$ fixing the vector $a_1+a_i\in V$.
\begin{lemma}[Lemma~6.10,~\cite{GiLoSp}]\label{aux-11}
For every $i\in \{2,\ldots,k\}$,
\begin{equation}\label{eq:-1}G_{\theta_{a_1}}\cap G_{\theta_{a_i}}={\bf C}_{G_{\theta{a_1}}}(a_1+a_i),
\end{equation}
\end{lemma}

We are now ready to prove our main result of this section. 
\begin{proposition}\label{prop:25}Let $G$ be a permutation group as in Definition~$\ref{defintion6.2}$~$\eqref{definition6.2:2}$. Then $G$ is  $\mathrm{IBIS}$ if, and only if, $G$ is one of the following:
\begin{enumerate}
\item\label{prop:25_1} $m=1$, $G=\mathrm{Sp}_2(q)$ in its action on $\Omega^-$, $q=2^f$, $f\ge 3$,
\item\label{prop:25_2} $m=1$, $G=\mathrm{Sp}_2(q)$ or $G=\mathrm{Aut}(\mathrm{Sp}_2(q))$ in its action on $\Omega^-$, $q=4$,
\item\label{prop:25_3} $m=1$, $G=\mathrm{Aut}(\mathrm{Sp}_2(q))\cong\Gamma\mathrm{L}_2(q)$ in its action on $\Omega^+$, $q=2^f$, $f$ odd prime, 
\item\label{prop:25_4} $m=q=2$.
\end{enumerate}
\end{proposition}
\begin{proof}
Suppose first $m=1$. Here, $G_0=\mathrm{Sp}_2(q)=\mathrm{SL}_2(q)$. The stabilizer of an element of $\Omega^+$ in $G_0$ is isomorphic to the dihedral group of order $2(q-1)$. This action was discussed in Example~3.8 in~\cite{LeSp} and it was proven that it is IBIS only when $G=\Gamma\mathrm{L}_2(q)$,  $q=2^f$ and $f$ is an odd prime. In particular, we obtain the examples in~\eqref{prop:25_3}. The stabilizer of an element of $\Omega^-$ in $G_0$ is isomorphic to the dihedral group of order $2(q+1)$. This action was discussed in Example~3.7 in~\cite{LeSp} and it was proven to be IBIS only when $G=\mathrm{SL}_2(q)$,  $q=2^f$ and $f\ge 3$, or when $\mathrm{SL}_2(q)\unlhd G\le\mathrm{Aut}(\mathrm{SL}_2(q))$ with $q=4$. In particular, we obtain the examples in~\eqref{prop:25_1} and~\eqref{prop:25_2}. In conclusion, if $m=1$, $G$ is $\mathrm{IBIS}$ if and only if $G$ is as in cases \eqref{prop:25_1},\eqref{prop:25_2},\eqref{prop:25_3}.

\smallskip

We now deal with the case $m=q=2$. In this case $\mathrm{Sp}_4(2)$ is not simple, since $\mathrm{Sp}_4(2)'\cong\mathrm{Alt}(6)$. Hence $G$ is an almost simple group having socle $\mathrm{Alt}(6)$. Moreover, $|\Omega^+|=10$ and $|\Omega^-|=6$. The action of $G$ on $\Omega^\pm$ is IBIS, see~\cite[Theorem~1.2]{LeSp}. In particular, we obtain the examples in~\eqref{prop:25_4}. Therefore, for the rest of our argument, we may suppose that $(m,q)\ne (2,2)$ and $m\ge 2$.

\smallskip

As usual, set $G_0=\mathrm{Sp}_{2m}(q)$. From Lemma~\ref{report}, we may idenfity the action of $G_0$ on the right cosets of $\mathrm{O}_{2m}^\pm(q)$ with the action of $G_0$ on $\Omega^\pm$ and we may also use the notation above.

Suppose first $q=2$, so that $m > 2$. This case is rather special because in this case the action of $G_0$ on both $\Omega^+$ and $\Omega^-$ is $2$-transitive.\footnote{In general  the action of $G_0$ on $\Omega^\pm$ has rank $q$, see for instance~\cite{GuPrSp}.} We first consider $\Omega^+$ and observe that $\theta_0,\theta_{e_1}\in \Omega^+$ because $\theta_0(e_1)=0$, (that is, $e_1$ is singular for the quadratic form $\theta_0$), see~\eqref{defineomega}. From Lemma~\ref{aux-11}, $G_{\theta_0}\cap G_{\theta_{e_1}}={\bf C}_{\mathrm{O}_{2m}^+(q)}(e_1)$. Since the transitive action of $G_{\theta_0}$ on $\Omega^+\setminus\{\theta_0\}$ is permutation equivalent to the action of $G_{\theta_0}=\mathrm{O}_{2m}^+(q)$ on the right cosets of $G_{\theta_0}\cap G_{\theta_{e_1}}={\bf C}_{\mathrm{O}_{2m}^+(q)}(e_1)$, we deduce that the action of $G_{\theta_0}$ on $\Omega^+\setminus\{\theta_0\}$ is permutation equivalent to the action of $\mathrm{O}_{2m}^+(q)$ on the set of singular non-zero vectors of $V$, with respect to the quadratic form $\theta_0$. In the special case $q=2$, this implies that  the action of $G_{\theta_0}$ on $\Omega^+\setminus\{\theta_0\}$ is permutation equivalent to the action of $\mathrm{O}_{2m}^+(q)$ on the set of $1$-dimensional singular subspaces of $V$, which we have shown in Lemma~\ref{ort+} to be not IBIS. The argument for $\Omega^-$ is similar. Observe that $\theta_\varepsilon\in \Omega^-$. Hence, using~\eqref{eq:0} and using the fact that $\theta_\varepsilon(e_2)=0$,
$$\theta_{\varepsilon}^{t_{e_2}}=\theta_{\varepsilon+(\sqrt{\theta_\varepsilon(e_2)}+1)e_2}=\theta_{\varepsilon+e_2}\in \Omega^-.$$
 From Lemma~\ref{aux-11}, $G_{\theta_\varepsilon}\cap G_{\theta_{\varepsilon+e_2}}={\bf C}_{\mathrm{O}_{2m}^-(q)}(e_2)$. Since the transitive action of $G_{\theta_\varepsilon}$ on $\Omega^-\setminus\{\theta_\varepsilon\}$ is permutation equivalent to the action of $G_{\theta_\varepsilon}=\mathrm{O}_{2m}^-(q)$ on the right cosets of $G_{\theta_\varepsilon}\cap G_{\theta_{\varepsilon+e_2}}={\bf C}_{\mathrm{O}_{2m}^-(q)}(e_2)$ and since $\theta_\varepsilon(e_2)=0$, we deduce that the action of $G_{\theta_\varepsilon}$ on $\Omega^-\setminus\{\theta_\varepsilon\}$ is permutation equivalent to the action of $\mathrm{O}_{2m}^-(q)$ on the set of singular non-zero vectors of $V$, with respect to the quadratic form $\theta_\varepsilon$. In the special case $q=2$, this implies that the action of $G_{\theta_\varepsilon}$ on $\Omega^-\setminus\{\theta_\varepsilon\}$ is permutation equivalent to the action of $\mathrm{O}_{2m}^-(q)$ on the set of $1$-dimensional singular subspaces of $V$, which we have shown in Lemma~\ref{evenminusorthogonal} to be not IBIS.

\smallskip

Suppose next $q>2$. We show that in this case $G$ is not $\mathrm{IBIS}$. To start with, we recall that $G = \mathrm{Sp}_{2m}(q)\rtimes \langle \varphi \rangle$, for some field automorphism $\varphi \in \mathrm{Aut}(\mathbb{F}_q)$. We first take in account $\Omega^+$. Consider $$H=G_{\theta_0}\cap \bigcap_{i=3}^m(G_{\theta_{e_i}}\cap G_{\theta_{e_{i+m}}}).$$
By construction, $H\cong\mathrm{O}_4^+(q) \rtimes \langle \varphi \rangle $ fixes the vector space $\langle e_i,e_{i+m}\mid 3\le i\le m\rangle$ and induces the orthogonal group $\mathrm{O}_4^+(q)\rtimes \langle \varphi \rangle $ on $\langle e_1,e_2,e_{1+m},e_{2+m}\rangle$. This shows that we may deduce the fact that the action of $G$ on $\Omega^+$ is not IBIS for general values of $m$, arguing only on the case $m=2$. Therefore, in what follows, we assume $m=2$. A direct computation with the quadratic form $\theta_0$ shows that
\begin{align*}
G_{\theta_0}\cap G_{\theta_{e_2}}=&\left\{
\begin{pmatrix}
a&ay&0&0\\
0&1&0&0\\
0&a^{-1}x&a^{-1}&0\\
x&xy&y&1\\
\end{pmatrix}\varphi^k \mid a,x,y\in\mathbb{F}_q,a\ne 0, k \in \mathbb{N}
\right\}\\
&\cup
\left\{
\begin{pmatrix}
0&ax&a&0\\
0&1&0&0\\
a^{-1}&a^{-1}y&0&0\\
x&xy&y&1\\
\end{pmatrix}\varphi^k \mid a,x,y\in\mathbb{F}_q,a\ne 0, k \in \mathbb{N}
\right\}.
\end{align*}
In particular, $|G_{\theta_0}\cap G_{\theta_{e_2}}|=2(q-1)q^2|\varphi |$, where $|\varphi|$ is the order of the automorphism $\varphi$. Similarly,
$$G_{\theta_0}\cap G_{\theta_{e_2}}\cap G_{\theta_{e_1}}=\left\{
\begin{pmatrix}
1&0&0&0\\
0&1&0&0\\
0&x&1&0\\
x&0&0&1\\
\end{pmatrix}\varphi^k\mid x\in\mathbb{F}_q, k \in \mathbb{N}
\right\}.
$$
In particular, $|G_{\theta_0}\cap G_{\theta_{e_2}}\cap G_{\theta_{e_1}}|=q|\varphi|$. Moreover, from this description, it is clear that
$G_{\theta_0}\cap G_{\theta_{e_2}}\cap G_{\theta_{e_1}}\cap G_{\theta_{e_4}}=\langle \varphi \rangle$. Let $\lambda$ be a generator of the multiplicative group of the field $\mathbb{F}_q$. If we repeat the above argument with $e_4$ replaced by $\lambda e_4$, we deduce 
\begin{equation}
	\label{eqnew}G_{\theta_0}\cap G_{\theta_{e_2}}\cap G_{\theta_{e_1}}\cap G_{\theta_{\lambda e_4}}=1.
\end{equation} Therefore $\mathrm{Sp}_4(q)\rtimes \langle \phi \rangle$ on $\Omega^+$ has a base of size $4$ and a base of size at least $5$ as long as $\varphi$ is not the identity.\\
Suppose now that $\varphi = 1$, that is $G = \mathrm{Sp}_4(q)$. Observe that \eqref{eqnew} shows that $G$ has a base of cardinality $4$.
We now exhibit a base of cardinality $5$. Let $\lambda\in\mathbb{F}_q$ with $\mathrm{Tr}(\lambda)=0$ and $\lambda\ne 0$: observe that this is possible because $q\ne 2$. Let $a=e_1+\lambda e_3$ and observe that $\theta_a\in\Omega^+$ because $\theta_0(a)=0$.  A direct computation with the quadratic form $\theta_0$ shows that
\begin{align*}
G_{\theta_0}\cap G_{\theta_{e_2}}\cap G_{\theta_a}=&\left\{
\begin{pmatrix}
1&\lambda x&0&0\\
0&1&0&0\\
0&ax&1&0\\
x&\lambda x^2&\lambda x&1\\
\end{pmatrix}\mid x\in\mathbb{F}_q
\right\}\\
&\cup
\left\{
\begin{pmatrix}
0&\lambda x&\lambda &0\\
0&1&0&0\\
\lambda^{-1}&x&0&0\\
x&\lambda x^2&\lambda x&1\\
\end{pmatrix}\mid x\in\mathbb{F}_q
\right\}.
\end{align*}
In particular, $|G_{\theta_0}\cap G_{\theta_{e_2}}\cap G_{\theta_a}|=2q$. Moreover, from this description, it is clear that
$$G_{\theta_0}\cap G_{\theta_{e_2}}\cap G_{\theta_{a}}\cap G_{\theta_{e_4}}=\left\langle
\begin{pmatrix}
0&0&\lambda&0\\
0&1&0&0\\
\lambda^{-1}&0&0&0\\
0&0&0&1\\
\end{pmatrix}\right\rangle$$
has order $2$. Now, $G_{\theta_0}\cap G_{\theta_{e_2}}\cap G_{\theta_{a}}\cap G_{\theta_{e_4}}\cap G_{\theta_{e_1}}=1$. Therefore 
the action of $\mathrm{Sp}_4(q)$ on $\Omega^+$ has a base of size $5$. This concludes the proof that $G$ on $\Omega^+$ is not IBIS when $q>2$.

\smallskip

We now consider $\Omega^-$. Observe that $\theta_\varepsilon(e_i)=0$ for each $i\in \{1,\ldots,2m\}\setminus\{1,1+m\}$. In particular, for these values of $i$, we have $$\theta_\varepsilon^{t_{e_i}}=\theta_{\varepsilon+(\sqrt{\theta_\varepsilon(e_i)}+1)e_i}=\theta_{\varepsilon+e_i}.$$ Consider $$H=G_{\theta_\varepsilon}\cap \bigcap_{i=3}^m(G_{\theta_{\varepsilon+e_i}}\cap G_{\theta_{\varepsilon+e_{i+m}}})={\bf C}_{G_{\theta_\varepsilon}}(\langle e_i,e_{i+1}\mid 3\le i\le m\rangle).$$
By construction, $H\cong\mathrm{O}_4^-(q)\rtimes \langle \varphi \rangle$ fixes  the vector space $\langle e_i,e_{i+m}\mid 3\le i\le m\rangle$ and induces the orthogonal group $\mathrm{O}_4^-(q)\rtimes \langle \varphi \rangle$ on $\langle e_1,e_2,e_{1+m},e_{2+m}\rangle$. This shows that we may deduce the fact that the action of $G$ on $\Omega^-$ is not IBIS for general values of $m$, arguing only on the case $m=2$. Therefore, in what follows, we assume $m=2$. Observe that when $m=2$, the quadratic form $\theta_\varepsilon$ has matrix
\[
\begin{pmatrix}
1&0&1&0\\
0&0&0&1\\
0&0&\epsilon^2&0\\
0&0&0&0
\end{pmatrix}.
\]

 As $\theta_\varepsilon(e_2)=0$, we have
$$\theta_{\varepsilon}^{t_{e_2}}=\theta_{\varepsilon+(\sqrt{\theta_\varepsilon(e_2)}+1)e_2}=\theta_{\varepsilon+e_2}.$$ A computation gives that the group $$G_{\theta_\varepsilon}\cap G_{\theta_{\varepsilon+e_2}}={\bf C}_{G_{\theta_\varepsilon}}(e_2)$$
consists of the matrices
\[
\begin{pmatrix}
a&cx_3+ay_3&c&0\\
0&1&0&0\\
x_1&x_1y_3+x_3y_1&y_1&0\\
x_3&x_3^2+x_3y_3+\epsilon^2y_3^2&y_3&1\\
\end{pmatrix}
\hbox{ subject to }
\left\{
\begin{array}{lc}
ay_1+cx_1&=1\\
a^2+ac+\epsilon^2c^2&=1\\
x_1^2+x_1y_1+\epsilon^2 y_1^2&=\epsilon^2
\end{array}
\right.
\]
together with the group generated by the automorphism $\varphi$. Moreover, as $\theta_\varepsilon(e_4)=0$, we have
$$\theta_{\varepsilon}^{t_{e_4}}=\theta_{\varepsilon+(\sqrt{\theta_\varepsilon(e_4)}+1)e_4}=\theta_{\varepsilon+e_4}.$$ A computation gives that the group $$G_{\theta_\varepsilon}\cap G_{\theta_{\varepsilon+e_2}}\cap G_{\theta_{\varepsilon+e_4}}={\bf C}_{G_{\theta_\varepsilon}}(\langle e_2,e_4\rangle)$$
consists of the matrices
\[
\begin{pmatrix}
a&0&c&0\\
0&1&0&0\\
x_1&0&y_1&0\\
0&0&0&1\\
\end{pmatrix}
\hbox{ subject to }
\left\{
\begin{array}{lc}
ay_1+cx_1&=1\\
a^2+ac+\epsilon^2c^2&=1\\
x_1^2+x_1y_1+\epsilon^2 y_1^2&=\epsilon^2
\end{array}
\right.
\]
together with the group generated by the automorphism $\varphi$. Moreover, as $\theta_\varepsilon(e_3)=\epsilon^2$, we have
$$\theta_{\varepsilon}^{t_{e_3}}=\theta_{\varepsilon+(\sqrt{\theta_\varepsilon(e_3)}+1)e_3}=\theta_{\varepsilon+(1+\epsilon)e_3}.$$
Using Theorem~$1.1$ of \cite{reis} (applied with $a=1$ and $q=2$ with the notation therein), we may choose $\epsilon$ to be a field generator with $\mathrm{Tr}(\epsilon)=1$. A computation gives that the group $$G_{\theta_\varepsilon}\cap G_{\theta_{\varepsilon+e_2}}\cap G_{\theta_{\varepsilon+e_4}}\cap G_{\theta_{\varepsilon+(1+\epsilon)e_3}}$$
has order 2 and is generated by
\[
\begin{pmatrix}
1&0&\epsilon^{-2}&0\\
0&1&0&0\\
0&0&1&0\\
0&0&0&1\\
\end{pmatrix}.
\]
Therefore $\mathrm{Sp}_4(q)\rtimes \langle \varphi \rangle$ in its action on $\Omega^-$ has a base of size $5$. Now, fix $\alpha\in\mathbb{F}_q\setminus\{0,1\}$ and let $a=\alpha e_1$. We have
$$\theta_{\varepsilon}^{t_{a}}=\theta_{\varepsilon+(\sqrt{\theta_\varepsilon(a)+1})a}=\theta_{\varepsilon+\alpha(\alpha+1)e_1}.$$ A computation gives that the group $$G_{\theta_\varepsilon}\cap G_{\theta_{\varepsilon+e_2}}\cap G_{\theta_{\varepsilon+(1+\alpha)\alpha e_1}}$$
consists of the matrices
\[
\begin{pmatrix}
1&0&0&0\\
0&1&0&0\\
x_1&x_3&1&0\\
x_3&x_3^2&0&1\\
\end{pmatrix}
\hbox{ subject to }
x_1\in\{0,1\}
\]
together with the elements in the group generated by $\varphi$ fixing $\alpha(\alpha+1)$. 
Finally,
$$G_{\theta_\varepsilon}\cap G_{\theta_{\varepsilon+e_2}}\cap G_{\theta_{\varepsilon+(1+\alpha)\alpha e_1}}\cap G_{\theta_{\varepsilon+(1+\epsilon)e_3}}=1$$
and hence $\mathrm{Sp}_4(q)\rtimes \langle \varphi \rangle$ on $\Omega^-$ has also a base of size $4$.
\end{proof}

\section{Primitive action on non-degenerete subspaces: Case~\eqref{HiHo2}}\label{nondegenerate}

Here, we deal with the action on the non-degenerate subspaces and hence with case~\eqref{HiHo2}. Let $G$ an almost simple group with socle a classical group $G_0$ (either symplectic, unitary or orthogonal), and let $V$ be its associated module of dimension $d$. We describe the primitive action of $G$ on non-degenerate subspaces.

When $G$ is symplectic or unitary, the action of $G$ is on the collection of all non-degenerate subspaces of $V$ of a given dimension $k$ with $ k<d/2$. 
When $G$ is orthogonal, there are a few cases to consider. Indeed, the collection of non-degenerate $k$-subspaces of $V$ is not always a $G_0$-orbit. Indeed, it is a single orbit if and only if $d$ is even and $k$ is odd (see Table~$4.1.1$ in \cite{BurGiu}). In this section, we denote with $N_k$ a $G$-orbit of $k$-dimensional non-degenerate subspaces.

\begin{lemma}\label{monday1}
Let $V,d$ and $k$ be as above and let $W_1$ be a non-degenerate $k$-dimensional subspace of $V$. Except when 
\begin{itemize}
\item $(q,k)=(2,2)$ and $V$ is symplectic, or 
\item $ (q,k)=(2,1)$ and $V$ is Hermitian, or
\item  $(q,k)=(3,1)$ and $V$ is orthogonal, or
\item $(q,k)=(2,2)$, $V$ is orthogonal and $W_1$ is anisotropic,
\end{itemize} there exist two non-degenerate subspaces $W_2$ and $W_3$ of $V$ and an element $g$ of $G_0$ with
\begin{enumerate}
\item\label{eq:monday1}$W_1\perp W_2$,
\item\label{eq:monday2}$W_1\oplus W_2=W_1\oplus W_3$ and $W_1\cap W_3=W_2\cap W_3=0$,
\item\label{eq:monday3}$W_1^g=W_1$, $W_2^g=W_2$ and $W_3^g\ne W_3$.
\end{enumerate}
\end{lemma}
\begin{proof}
We give complete details with $V$ is symplectic, all other cases are similar.
Let $b$ be the symplectic form defining $V$. Fix a non-degenerate $k$-dimensional subspace $W_1$ of $V$ and fix a non-degenerate $k$-dimensional subspace $W_2$ of $W_1^\perp$. Clearly,~\eqref{eq:monday1} is satisfied. Let $f:W_1\to W_2$ be an arbitrary isomorphism such that $b(f(w_1),f(w_2))\ne b(w_1,w_2)$, for some $w_1,w_2\in W_1$. The existence of $f$ is clear when $q>2$ because we may take $f$ the scalar multiplication by a scalar different form $0$ and $1$. When $q=2$, the existence of $f$ depends on the the fact that $\mathrm{SL}_{2k}(2)=\mathrm{Sp}_{2k}(2)$ only when $k=2$. Therefore, for the rest of the argument assume $(k,q)\ne (2,2)$.

Let $W_3=\{w+f(w)\mid w\in W_1\}$. Clearly, $W_3$ satisfies~\eqref{eq:monday2}. Let $w_1,w_1'\in W_1$ and let $w_3=w_1+f(w_1)$, $w_3'=w_1'+f(w_1')$. As 
\begin{align*}
b(w_3,w_3')&=
b(w_1,w_1')+b(w_1,f(w_1'))+b(f(w_1),w_1')+b(f(w_1),f(w_1'))\\
&=
b(w_1,w_2)+b(f(w_1),f(w_2)),
\end{align*} we see that $W_3$ is non-degenerate from our assumption on $f$.

Finally, let $x:W_1\to W_1$ be any non-identity symplectic isomorphism and, using Witt's lemma, extend $x$ to a symplectic automorphism $g$ of $V$ fixing pointwise $W_2$. Clearly, $g$ satisfies~\eqref{eq:monday3}.
\end{proof}

We need an analogous result for the exceptional cases arising in Lemma~\ref{monday1}.
\begin{lemma}\label{monday3}
Assume that $V,k,W_1$ are one of the four exceptional cases in Lemma~$\ref{monday1}$. Then,
there exist three non-degenerate subspaces $W_2$, $W_3$ and $W_4$ of $V$ and an element $g$ of $G_0$ with
\begin{enumerate}
\item\label{eq:monday11}$W_1\perp W_2\perp W_3$,
\item\label{eq:monday22}$W_1\oplus W_2\oplus W_3=W_1\oplus W_2\oplus W_4$ and $W_1\cap W_4=W_2\cap W_4=W_3\cap W_4=0$,
\item\label{eq:monday33}$W_1^g=W_1$, $W_2^g=W_2$, $W_3^g=W_3$ and $W_4^g\ne W_4$.
\end{enumerate}
\end{lemma}
\begin{proof}
The proof uses the same argument as the proof of Lemma~\ref{monday1}. In this case, it suffices to take $W_4=\{w+f_2(w)+f_3(w)\mid w\in W_1\}$, where $f_2:W_1\to W_2$ and $f_3:W_1\to W_3$ are suitable isomorphisms.
\end{proof}

Using Lemmas~\ref{monday1} and~\ref{monday3} , we immediately deduce the following result, which concludes case \ref{HiHo2}.
\begin{lemma}\label{monday2}
The primitive actions of $G$ on $N_k$, for $1 \leq k < d/2$, are not $\mathrm{IBIS}$.
\end{lemma}
\begin{proof}
Suppose first that we are not in one of the exceptional cases in Lemma~\ref{monday1} and let $W_1,W_2,W_3$ be as in the statement of Lemma~\ref{monday1}. Clearly,
$$G_{W_1}>G_{W_1,W_2}>G_{W_1,W_2,W_3},$$
where the last strict inequality uses~\eqref{eq:monday3} of Lemma~\ref{monday1}. Finally, observe that $G_{W_1,W_3}\le G_{W_1,W_2}$ from~\eqref{eq:monday1} and~\eqref{eq:monday2} of Lemma~\ref{monday1}. This implies that 
\[
	G_{W_1,W_2,W_3} = G_{W_1,W_3}.
\]
Thus, $(W_1,W_2,W_3)$ and $(W_1,W_3)$ are to irredundant chain reaching the same stabilizer, hence $G$ is not $\mathrm{IBIS}$.

Next, suppose that we are in one of the exceptional cases in Lemma~\ref{monday1} and let $W_1,W_2,W_3,W_4$ be as in the statement of Lemma~\ref{monday3}. Clearly,
$$G_{W_1}>G_{W_1,W_2}>G_{W_1,W_2,W_3}>G_{W_1,W_2,W_3,W_4},$$
where the last strict inequality uses~\eqref{eq:monday33} of Lemma~\ref{monday3}. Finally, observe that $G_{W_1,W_2,W_4}\le G_{W_1,W_2,W_3}$ from~\eqref{eq:monday11} and~\eqref{eq:monday22} of Lemma~\ref{monday3}. This implies that $$G_{W_1,W_2,W_3,W_4} = G_{W_1,W_2,W_4}.$$ Thus, $(W_1,W_2,W_3,W_4)$ and $(W_1,W_2,W_4)$ are two irredundant chain reaching the same stabilizer, hence $G$ is not $\mathrm{IBIS}$.
\end{proof}

\section{Action on non-singular $1$-subspaces: Case~\eqref{HiHo3}}\label{1non-singular}
Here $G$ is an almost simple group having socle $G_0=\Omega_d^\pm(q)$  acting on the collection $\mathrm{NS}_1$ of $1$-dimensional non-singular subspaces of $V$ and $q$ is even. Since $\Omega_2^\pm(q)$ is solvable, we may suppose that $d\ge 4$.

We look closely at the case $d=4$ because this will serve as a base case of an induction.

\begin{lemma}\label{uuu1} When $d=4$, the action of $G$ on $\mathrm{NS}_1$  is $\mathrm{IBIS}$ if and only if $G=G_0=\Omega_4^\pm(q)$ and $q \geq 4$.
\end{lemma}
\begin{proof}
We first deal with the hyperbolic case.\\
In this case, $G_0=\Omega_4^+(q)=\mathrm{SL}_2(q)\times \mathrm{SL}_2(q)$. Moreover, the stabilizer of a non-singular vector in $G_0$ is isomorphic to $\mathrm{Sp}_2(q)=\mathrm{SL}_2(q)$ diagonally embedded in $\mathrm{SL}_2(q)\times \mathrm{SL}_2(q)$. Therefore, in terms of the O'Nan-Scott classification of primitive groups, this action is diagonal. Therefore, using the main result in~\cite{LuMoMo}, we deduce that the action is IBIS exactly when $G=G_0$.

Suppose now we are in the elliptic case. \\
	Let $q=2^f$ for some integer $f$ and let $\phi:\mathbb{F}_q\to\mathbb{F}_q$ be the field automorphism defined by $x^\phi=x^2$, for all $ x\in \mathbb{F}_q$.
We use the quadratic form
$$X_1X_3+X_2^2+X_2X_4+\mu X_4^2,$$
for some $\mu\in \mathbb{F}_q$.

Assume first that $G=G_0=\Omega_4^-(q) = \mathrm{PSL}_2(q^2)$. The stabilizer of the action of $G$ is $\mathrm{Sp}_2(q) = \mathrm{SL}_2(q)$, so the action is the one described in Example~$3.9$ in \cite{LeSp}, and it is $\mathrm{IBIS}$ if $q\geq 4$. Suppose then that $$G_0 < G \leq \mathrm{Aut}(G_0) = \mathrm{SO}_4^-(q) \mathrm{Aut}(\mathbb{F}_q).$$

Let $\alpha$ be a generator of the multiplicative group of the field $\mathbb{F}_q$. Using the quadratic form above, we see that the stabilizer of $\langle e_2\rangle$  and $\langle e_4\rangle$ in the whole group  $\mathrm{SO}_4^-(q) \mathrm{Aut}(\mathbb{F}_q)$ is
\[
\left\{
\begin{pmatrix}
	a&0&0&0\\
	0&1&0&0\\
	0&0&a^{-1}&0\\
	0&0&0&1
\end{pmatrix},
\begin{pmatrix}
	0&0&a&0\\
	0&1&0&0\\
	a^{-1}&0&0&0\\
	0&0&0&1
\end{pmatrix}
\left|\right. a\in\mathbb{F}_q\setminus\{0\}\right\}\rtimes\mathrm{Aut}(\mathbb{F}_q).
\]
With this, it is easy to verify that 
\begin{align*}
	&(\langle e_2\rangle,\,\langle e_4 \rangle,\,\langle  e_1+ \alpha e_2+  e_4\rangle)
\end{align*}
is an irredundant base\footnote{Observe that $\langle  e_1+ \alpha e_2+ e_4\rangle$ is non-singular since $\mu$ is chosen in such a way that  the polynomial $T^2+T+\mu$ is irreducible over $\mathbb{F}_q$.} for $\mathrm{SO}_4^-(q) \mathrm{Aut}(\mathbb{F}_q)$ of cardinality $3$, where $\alpha$ is a generator of the multiplicative group of the field $\mathbb{F}$.
To conclude that $G$ is not IBIS, we show that $G$ admits an irredundant base of cardinality at least $4$ and then we invoke Lemma~\ref{lemma:minusminus1}. 
If $G < \mathrm{SO}_4^-(q) \mathrm{Aut}(\mathbb{F}_q)$, then either 
\begin{itemize}
	\item  $\mathrm{SO}_4^-(q)\le G$, or
	\item  $G \leq \Omega_4^-(q) \rtimes \langle \phi^i \rangle$ where $i | f$, or
	\item  $G \leq \Omega_4^-(q) \rtimes \langle \iota \phi^i \rangle$ where $i | f$, $f/i$ is even, and $\iota$ is the matrix 
	$$
	\begin{pmatrix}
		0&0&1&0\\
		0&1&0&0\\
		1&0&0&0\\
		0&0&0&1
	\end{pmatrix}.
	$$
\end{itemize}
In the first case, using the description of the stabilizer of $\langle e_2\rangle$ and $\langle e_4\rangle$, we get $\iota \in G_{\langle e_2\rangle,\langle e_4\rangle,\langle e_1+e_3\rangle}$
and hence $(\langle e_2\rangle,\langle e_4\rangle,\langle e_1+e_3\rangle)$ is an irredundant chain having non-identity stabilizer.
In the second and third case, since $(\langle e_2\rangle,\langle e_4\rangle,\langle e_1+e_3\rangle)$ is an irredundant base for $\Omega_4^-(q)$, the stabilizer of these three points is $\langle \phi^i \rangle$ or $\langle \iota \phi^i \rangle$ respectively. Hence, in all cases $G$ admits a base of cardinality at least $4$, so it is not $\mathrm{IBIS}$.
\end{proof}

In the proof of Lemma~\ref{uuu1}, we have used~\cite{LuMoMo} in the hyperbolic case and~\cite{LeSp} in the elliptic case (in this latter case, only to show that $G_0$ is $\mathrm{IBIS}$). However, in both cases we could also have argued geometrically.

\begin{proposition}\label{uuuu1}
The action of $G$ on $\mathrm{NS}_1$ is $\mathrm{IBIS}$ if and only if $G=G_0=\Omega_d^\pm(q)$ and $q\ge 4$.
\end{proposition}
\begin{proof}
From Lemma~\ref{uuu1}, we may suppose that $d\ge 6$. We use the hyperbolic quadratic form given by $$X_1X_2+X_3X_4+\cdots +X_{d-1}X_d,$$ and the elliptic quadratic form given by 	$$X_1X_{\frac{d}{2}+1}+X_2X_{\frac{d}{2}+2}+\cdots +X_{\frac{d}{2}-1}X_{d-1}+X_{\frac{d}{2}}^2 + X_{\frac{d}{2}}X_d + \mu X_d^2.$$
We first show that if $G = G_0$ with $q \geq 4$, then $G$ is $\mathrm{IBIS}$.  We show that all irredundant bases of $G$ have cardinality $d-1$. Let 
$\omega_1=\langle v_1\rangle,\ldots, \omega_\ell=\langle v_\ell\rangle$ be an irredundant base for $G$. Since $G=\Omega_d^\pm(q)$, $G$ fixes the vectors $v_1,\ldots,v_\ell$ and hence it fixes pointwise the subspace generated by $v_1,\ldots,v_\ell$. This immediately shows that $\ell\le d$. Suppose $d=\ell$. Then $v_1,\ldots,v_d$ is a basis of the vector space $V$. The pointwise stabilizer of the subspace $\langle v_1,\ldots,v_{d-1}\rangle$ must be the identity group because  $G=\Omega_d^\pm(q)$ does not contain any transvection. Therefore, we contradict the fact that $\omega_1,\ldots,\omega_\ell$ is an irredundant base. Finally, suppose $\ell\le d-2$ and let $W=\langle v_1,\ldots,v_{\ell}\rangle$. Choose any $(d-2)$-dimensional subspace $U$ of $V$ containing $W$. Suppose first $U^\perp$ is non-degenerate. The pointwise stabiliser of $U$ in $G$ is either $\Omega_2^+(q)$ or $\Omega_2^-(q)$. In both cases, since $q\ge 4$, the pointwise stabiliser of $U$ in $G$ is not the identity and it does move some non-singular vector, contradicting the fact that $\omega_1,\ldots,\omega_\ell$ is a base.  Suppose next that $U^\perp$ is totally isotropic and hence $U^\perp\le U$. In this case, the pointwise stabilizer of $U$ in $G$ contains an element of order $2$ which is the product of two communting transvections\footnote{Indeed, if we are in the hyperbolic case and if $U^\perp=\langle e_1,e_d\rangle$, then we may take the transvection with respect to $e_1$ and to $e_d$ and then take their product. Recall that each element of $\Omega_d^+(q)$ is the product of an even number of transvections in $\mathrm{SO}_d^+(q)$. The same argument works for the elliptic case.} and hence again we contradict the fact that $\omega_1,\ldots,\omega_\ell$ is a base. 

We now show that, if $G$ is IBIS, then  $G=G_0$ and $q\ge 4$. Here, we divide the case depending on $G_0$. Suppose firstly we are in the hyperbolic case. Assume $q\ge 4$ and let $\alpha\in\mathbb{F}_q\setminus\{0,1\}$. Let $\omega_1=\langle e_1+e_2\rangle$ and $\omega_2=\langle e_1+\alpha e_2\rangle$. Observe that $G_{\omega_1,\omega_2}$ stabilizes the subspace $\langle e_1,e_2\rangle$ and fixes two non-singular vectors in this subspace. Therefore, $G_{\omega_1,\omega_2}$ fixes pointwise $\langle e_1,e_2\rangle$. Thus, $G_{\omega_1,\omega_2}$ is a $(d-2)-$dimensional orthogonal group acting  on the non-singular $1$-subspaces. Hence, arguing inductively and using Lemma~\ref{uuu1} as the base case of the induction, we deduce $G=G_0$. Assume now $q=2$. In this case, either $G=G_0=\Omega_d^+(2)$  or $G=\mathrm{SO}_d^+(2)$. Let $\omega_1=\langle e_1+e_2\rangle$, $\omega_2=\langle e_1+e_2+e_3\rangle$, $\omega_3=\langle e_1+e_2+e_4\rangle$ and $\omega_4=\langle e_1+e_3+e_4\rangle$. Observe that $G_{\omega_1,\omega_2,\omega_3,\omega_4}$  fixes pointwise $\langle e_1,e_2,e_3,e_4\rangle$. Therefore, $G_{\omega_1,\omega_2,\omega_3,\omega_4}$ is a $d-4$ dimensional orthogonal group acting  on the non-singular $1$-subspaces. When $d\equiv 0\pmod 4$, arguing inductively, we reduce to the case $d=8$ (observe that, when $q=2$, $\Omega_4^+(2)$ is solvable and hence it is of no concern in this work). We have verified that both $\Omega_8^+(2)$ and $\mathrm{SO}_8^+(2)$ are not IBIS.
 When $d\equiv 2\pmod 4$, arguing inductively, we reduce to the case $d=6$. We have verified that both $\Omega_6^+(2)$ and $\mathrm{SO}_6^+(2)$ are not IBIS. For instance, when $G=\mathrm{SO}_6^+(2)$, it can be verified that
\begin{align*}
&(\langle e_1+e_2\rangle,\,\langle e_1+e_2+e_6\rangle,\,\langle e_1+e_2+e_3\rangle,\,\langle e_1+e_2+e_4\rangle,\,\langle e_5+e_6\rangle,\,\langle e_2+e_3+e_4\rangle),\\
&(\langle e_1+e_2\rangle,\,\langle e_1+e_2+e_6\rangle,\,\langle e_1+e_2+e_3\rangle,\,\langle e_1+e_2+e_4\rangle,\,\langle e_2+e_3+e_4\rangle)
\end{align*}
are irredundant bases of length $6$ and $5$.

The elliptic case is analogous.
\end{proof}

\section{$\mathrm{Sp}_4(2^a)$ containing a graph automorphism: Case~\eqref{HiHo5}}\label{Sp4graph}
In this case, $G_0=\mathrm{Sp}_4(q)\le G$ and $G$ contains a graph automorphism swapping the two types of parabolic subgroups: stabilizers of $1$-dimensional subspaces and stabilizers of $2$-dimensional totally singular subspaces. The action of $G$ can be identified with the action on pairs $\{p,\ell\}$, where $p$ is a $1$-dimensional subspace and $\ell$ is a $2$-dimensional totally singular subspace with $p\in \ell$. It is not hard to verify that this action is not IBIS. To prove that, we fix the symplectic form given by the matrix
\[
\begin{pmatrix}
0&I\\
I&0
\end{pmatrix}.
\]
Where $0$ and $I$ are the zero and identity $2 \times 2$ matrices respectively. 
Let
\begin{align*}
	\omega_1 &= \omega_1' = \{\langle e_1 \rangle, \langle e_1,e_2 \rangle\}, \\
	\omega_2 &= \omega_2' = \{\langle e_2\rangle, \langle e_1, e_2 \rangle \}, \\
	\omega_3 &= \{\langle e_3 \rangle, \langle e_3,e_4\rangle \}, \\
	\omega_3'&= \{ \langle e_3 \rangle, \langle e_2, e_3\rangle\}, \\
	\omega_4'&= \{ \langle e_4 \rangle, \langle e_1, e_4\rangle\}, \\
	\omega_4 &= \omega_5' =  \{ \langle e_1+e_2 \rangle, \langle e_1+e_2,e_3+e_4 \rangle\}, \\
	\omega_5 &= \omega_6' = \{ \langle e_2 + e_3 \rangle, \langle e_2+e_3,e_1+e_4 \rangle\} .
\end{align*}
Observe that $\langle e_1, e_2 \rangle$ is common to $\omega_1$ and $\omega_2$, and hence every element of $G_{\omega_1,\omega_2}$ fixes $\langle e_1,e_2 \rangle$. Thus, $G_{\omega_1,\omega_2}$ does not contain graph automorphisms swapping the two subspaces. With this in hand, using the same argument that we have used in Lemma~\ref{prop:2}, it is not hard to verify that $(\omega_1,\dots,\omega_5)$ and $(\omega_1',\dots,\omega_6')$ are two irredundant bases for $G$, hence $G$ is not $\mathrm{IBIS}$.
\section{$\mathrm{P}\Omega_8^+(q)$ containing triality: Cases~\eqref{HiHo6} and~\eqref{HiHo7}}\label{trialies}
In this section we deal with the case $G_0=\mathrm{P}\Omega_8^+(q)$ and $G$ contains a triality.
\begin{lemma}If $G$ is as in Case~$\eqref{HiHo6}$, then $G$ is not $\mathrm{IBIS}$.
\end{lemma}
\begin{proof}
We have a geometric description of the domain $\Omega$ of $G$. We briefly recall a few basic facts about trialities, see~\cite[Chapter~8]{PolarSpaces}. Let $V=\mathbb{F}_q^8$ be the orthogonal spaces associated to $G$. The $4$-dimensional totally singular subspaces fall into two families (the greeks and the latins), where two subspaces $W_1,W_2$ are in the same class if $\dim(W_1)-\dim(W_1\cap W_2)$ is even. Now, $\Omega$ can be identified with the collection of all triples $\{p,\pi_1,\pi_2\}$, where $p$ is a point (that is, a $1$-dimensional totally singular subspace) and $\pi_1$ and $\pi_2$ are solids containing $p$ and of different type. Without loss of generality we may suppose that the hyperbolic quadric is $$X_1X_5+X_2X_6+X_3X_7+X_4X_8.$$

Let 
\begin{align*}
\omega_1&=\{\langle e_1\rangle,\langle e_1,e_2,e_3,e_4\rangle,\langle e_1,e_2,e_3,e_8\rangle\},\\
\omega_2&=\{\langle e_2\rangle,\langle e_1,e_2,e_3,e_4\rangle,\langle e_1,e_2,e_3,e_8\rangle\},\\
\omega_3&=\{\langle e_3\rangle,\langle e_1,e_2,e_3,e_4\rangle,\langle e_1,e_2,e_3,e_8\rangle\},\\
\omega_4&=\{\langle e_1+e_2+e_3\rangle,\langle e_1,e_2,e_3,e_4\rangle,\langle e_1,e_2,e_3,e_8\rangle\},\\
\omega_3'&= \{\langle e_1+e_2\rangle,\langle e_1,e_2,e_3,e_4\rangle,\langle e_1,e_2,e_3,e_8\rangle\},\\
\omega_4'&= \omega_3, \\
\omega_5'&= \omega_4.
\end{align*}
As in Section~\ref{Sp4graph}, since $\langle e_1,e_2,e_3,e_4\rangle$  is common to $\omega_1$ and $\omega_2$, every element of $G_{\omega_1,\omega_2}$ fixes $\langle e_1,e_2,e_3,e_4\rangle$. This implies that $G_{\omega_1,\omega_2}\le\mathrm{P}\Gamma\mathrm{L}_8(q)$, that is, $G$ does not contain triality. With this remark, it is elementary to show that, when $q>2$,
$(\omega_1,\omega_2,\omega_3,\omega_4)$ and $(\omega_1,\omega_2, \omega_3',\omega_4',\omega_5')$ are irredundant chains with
$$G_{\omega_1,\omega_2,\omega_3,\omega_4}=G_{\omega_1,\omega_2, \omega_3',\omega_4',\omega_5'}.$$
Therefore $G$ is not IBIS by Lemma~\ref{lemma:1}.

When $q=2$, we have verified with a computer that $b(G)=3$ and that $G$ also admits bases of cardinality $4$.
\end{proof}

Dealing with case~\ref{HiHo7} is much more delicate and in fact we need to start with a lemma dealing with certain non-subspace actions. Our proof is quite geometric and we believe that it is interesting in its own right.
\begin{lemma}\label{lemma:non-subspace}
Let $q$ be even and $G_0=\mathrm{Sp}_6(q)$ or let $q$ be odd and $G_0=\Omega_7(q)$. Let $G$ be an almost simple primitive group with socle $G_0$ such that the stabilizer of a point in $G_0$ is a maximal subgroup in the Aschbacher class $\mathcal{C}_9$ isomorphic to $G_2(q)$. Then $G$ is not $\mathrm{IBIS}$.
\end{lemma}
\begin{proof}
We denote with $\Omega$ the coset space $[G_0 : G_2(q)]$ on which $G$ acts. Thus, $G = G_0.\langle \varphi \rangle$, for some field automorphism $\varphi$. Lemma~3.1 and Proposition~3.5 in~\cite{Bur07} show that $G$ has a base of size $4$ and hence it suffices to show that $G$ has a base of cardinality at least $5$. Consider an irreducible spin representation $\rho:\mathrm{P}\Omega_8^+(q) \to \mathrm{P}\Omega_8^+(q)$. The proof of Lemma~$3.1$ and Proposition~3.2 in \cite{Bur07} show that the action of $G$ on $\Omega$ is equivalent to the action of $\rho(G_0).\langle \varphi \rangle$ on the set of $1$-dimensional non-singular subspace of the natural module of $\mathrm{P}\Omega_8^+(q)$. In particular, let $Q$ be the quadratic form
$$X_1X_8+X_2X_7+X_3X_6+X_4X_5,$$
which defines the hyperbolic quadric preserved by $\mathrm{P}\Omega_8^+(q)$. We 
may identify $\Omega$ with the set of $1$-dimensional subspaces of $\mathbb{F}_q^8$ with respect to $Q$.

Suppose that $q$ is even, so that $G_0 = \mathrm{Sp}_6(q)$ and let $G_0(\mathbb{F}_2)=\mathrm{Sp}_6(2)\le G_0$. Observe that, if $\rho(G_0(\mathbb{F}_2))$ in its action on the elements of $\Omega$ lying in $\mathbb{F}_2^8$ has a base of cardinality at least $5$, then so does $G$. We have constructed $\rho(G_0(\mathbb{F}_2))$ with \texttt{magma}. In particular, using the command ClassicalMaximals,
we see that $\rho(G_0(\mathbb{F}_2))$ is generated by the matrices
\[
\begin{pmatrix}
	1 & 1 & 0 & 0 & 0 & 0 & 0 & 0 \\
	0 & 1 & 0 & 0 & 0 & 0 & 0 & 0 \\
	1 & 1 & 1 & 1 & 0 & 0 & 0 & 0 \\
	0 & 1 & 0 & 1 & 0 & 0 & 0 & 0 \\
	0 & 0 & 0 & 0 & 1 & 1 & 0 & 0 \\
	0 & 0 & 0 & 0 & 0 & 1 & 0 & 0 \\
	0 & 0 & 0 & 0 & 1 & 1 & 1 & 1 \\
	0 & 0 & 0 & 0 & 0 & 1 & 0 & 1
\end{pmatrix},
\begin{pmatrix}
	1 & 1 & 1 & 0 & 0 & 0 & 0 & 0 \\
	0 & 0 & 0 & 1 & 0 & 0 & 0 & 0 \\
	0 & 1 & 1 & 1 & 0 & 0 & 0 & 0 \\
	0 & 1 & 0 & 0 & 0 & 0 & 0 & 0 \\
	0 & 0 & 0 & 0 & 0 & 1 & 1 & 0 \\
	0 & 0 & 0 & 0 & 0 & 1 & 0 & 1 \\
	0 & 0 & 0 & 0 & 1 & 1 & 0 & 1 \\
	0 & 0 & 0 & 0 & 0 & 0 & 0 & 1
\end{pmatrix},
\]
\[
\begin{pmatrix}
	1 & 0 & 0 & 0 & 0 & 0 & 0 & 0 \\
	1 & 1 & 1 & 0 & 0 & 0 & 0 & 0 \\
	0 & 0 & 1 & 0 & 0 & 0 & 0 & 0 \\
	0 & 0 & 0 & 1 & 0 & 0 & 0 & 0 \\
	0 & 0 & 0 & 0 & 1 & 0 & 0 & 0 \\
	0 & 0 & 0 & 0 & 0 & 1 & 1 & 0 \\
	0 & 0 & 0 & 0 & 0 & 0 & 1 & 0 \\
	0 & 0 & 0 & 0 & 0 & 0 & 1 & 1
\end{pmatrix},
\begin{pmatrix}
	0 & 0 & 0 & 0 & 0 & 0 & 0 & 1 \\
	0 & 1 & 0 & 1 & 1 & 1 & 1 & 0 \\
	0 & 1 & 0 & 1 & 0 & 0 & 0 & 1 \\
	0 & 1 & 1 & 0 & 0 & 0 & 0 & 1 \\
	0 & 1 & 0 & 1 & 0 & 1 & 0 & 0 \\
	0 & 0 & 1 & 0 & 1 & 0 & 0 & 1 \\
	0 & 1 & 0 & 0 & 0 & 0 & 0 & 0 \\
	1 & 0 & 1 & 0 & 1 & 1 & 0 & 1
\end{pmatrix}.
\]
Thus, we have verified that
$$(\langle e_1+e_8\rangle,\,\langle e_2+e_7\rangle,\langle e_3+e_6\rangle,\langle e_3+e_4+e_5\rangle,\langle e_4+e_5\rangle)$$
is an irredundant base for  $\rho(G_0(\mathbb{F}_2))$ in its action on $\Omega$.

Suppose that $q$ is odd, so that $G_0 = \Omega_7(q)$. Thus, $G_2(q) = (\Omega_7(q).\langle \varphi \rangle)_{\omega_1}$, for some $1$-dimensional non-singular subspace $\omega_1 \in \Omega$. A set of generators for $G_2(q)$ in $\Omega_7(q)$ is defined explicitly in~\cite[page~172]{Malle}. In particular, using the notation from~\cite{Malle}, we may assume that $G_2(q)$ is generated by the root elements
\[
x_\alpha(t)=
\begin{pmatrix}
1&t&0&0&0&0&0\\
0&1&0&0&0&0&0\\
0&0&1&t&-t^2&0&0\\
0&0&0&1&-2t&0&0\\
0&0&0&0&1&0&0\\
0&0&0&0&0&1&-t\\
0&0&0&0&0&0&1
\end{pmatrix},
x_\beta(t)=
\begin{pmatrix}
1&0&0&0&0&0&0\\
0&1&t&0&0&0&0\\
0&0&1&0&0&0&0\\
0&0&0&1&0&0&0\\
0&0&0&0&1&-t&0\\
0&0&0&0&0&1&0\\
0&0&0&0&0&0&1
\end{pmatrix},\]
with $t\in \mathbb{F}_q$, and the Coxeter element \[
w=
\begin{pmatrix}
0&0&0&0&0&0&1\\
0&0&0&0&0&-1&0\\
0&0&0&0&1&0&0\\
0&0&0&-1&0&0&0\\
0&0&1&0&0&0&0\\
0&-1&0&0&0&0&0\\
1&0&0&0&0&0&0\\
\end{pmatrix}.
\]
Moreover, from~\cite[Section~2]{Malle}, we see that the quadratic form preserved by $G_2(q)$ is
$$X_1X_7+X_2X_6+X_3X_5+X_4^2,$$
and is the restriction of $Q$ to a suitable $7$-dimensional non-degenerate subspace of $\mathbb{F}_q^7$. Let $\omega_2=\langle e_1+e_7\rangle$, $\omega_3=\langle e_1-e_7\rangle$, $\omega_{4}=\langle e_3+e_4\rangle $ and observe that $\omega_2,\omega_3,\omega_4$ are non-degenerate $1$-subspaces. Moreover,
\begin{align*}
x_\alpha(1)&\in G_2(q)\setminus (G_2(q))_{\omega_2}, &&w\in (G_2(q))_{\omega_2}\setminus (G_2(q))_{\omega_2,\omega_3},\\
wx_\beta(1)w&\in (G_2(q))_{\omega_2,\omega_3}\setminus (G_2(q))_{\omega_2,\omega_3,\omega_4}, &&x_\beta(1)\in (G_2(q))_{\omega_2,\omega_3,\omega_4}.
\end{align*}
Therefore $(\omega_1,\omega_2,\omega_3,\omega_4)$ is an irredundant chain of length $4$ that can be extended to a base of size at least $5$.
\end{proof}
Using Lemma~\ref{lemma:non-subspace} we may now deal with~\eqref{HiHo7}.
\begin{lemma}\label{lemma:caseHiHo7}
If $G$ is as in Case~$\eqref{HiHo7}$ then $G$ is not $\mathrm{IBIS}$.
\end{lemma}
\begin{proof}
Here $G_0=\mathrm{P}\Omega_8^+(q)$ and the stabilizer of a point in $G_0$ is isomorphic to $G_2(q)$: the centralizer of a graph automorphism of order $3$. 

Let $\Omega$ be the domain of $G$ and fix $\alpha\in \Omega$. Thus $(G_0)_\alpha={\bf C}_{G_0}(\tau)\cong G_2(q)$, for some graph automorphism $\tau$ of order $3$. From~\cite[Table~$8.50$]{bhr}, we deduce that the lattice of subgroups of $G_0$ containing $(G_0)_\alpha$ consists of five subgroups: $G_0$, $(G_0)_\alpha$ and three maximal subgroups $M_1,M_2,M_3$ of $G_0$ with $M_1\cong M_2\cong M_3\cong\Omega_7(q)$ when $q$ is odd and $M_1\cong M_2\cong M_3\cong\mathrm{Sp}_6(q)$ when $q$ is even. These embeddings arise from the spin representations of orthogonal groups. 

The group $G_0$ is imprimitive on $\Omega$ and each subgroup $M_i$ gives rise to a system of imprimitivity $\Sigma_i$ for the action of $G_0$ on $\Omega$. Clearly, $G$ does not preserve $\Sigma_i$ and in fact $G$ acts transitively on the three systems of imprimitivity $\Sigma_1,\Sigma_2,\Sigma_3$ making the action of $G$ primitive. Let $\theta_i:\Omega\to \Sigma_i$ be the natural projection, that is, $\theta_i(\beta)$ is the unique part in the partition $\Sigma_i$ containing the point $\beta$. We now consider the map
\begin{align*}
\theta:&\,\Omega\to \Sigma_1\times\Sigma_2\times\Sigma_3,\\
&\beta\mapsto\theta(\beta)=(\theta_1(\beta),\theta_2(\beta),\theta_3(\beta)).
\end{align*} Since $G$ is primitive on $\Omega$, $\theta$ is injective. Therefore, when necessary, we identify $\Omega$ with its image under $\theta$.

Let $i\in\{1,2,3\}$ and let $\{1,2,3\}\setminus\{i\}=\{j,k\}$. We claim that the mapping $\Omega\to \Sigma_j\times\Sigma_k$ defined by $\beta\mapsto (\theta_j(\beta),\theta_k(\beta))$ is a bijection. In other words, in the function $\theta$, each coordinate is uniquely determined by the remaining two coordinates. Firstly, note that
\begin{align*}
|\Omega|&=|\mathrm{P}\Omega_8^+(q):G_2(q)|=q^6(q^4-1)^2,\\
|\Sigma_i|&=\begin{cases}
|\mathrm{P}\Omega_8^+(q):\Omega_7(q)|=q^3(q^4-1),&\textrm{when }q\textrm{ is odd},\\
|\mathrm{P}\Omega_8^+(q):\mathrm{Sp}_6(q)|=q^3(q^4-1),&\textrm{when }q\textrm{ is even}.
\end{cases}
\end{align*}
Therefore, $|\Omega|=|\Sigma_j\times\Sigma_k|$. Hence, it suffices to show that $\beta\mapsto (\theta_j(\beta),\theta_k(\beta))$ is injective. Let $\beta,\beta'\in \Omega$ with  $(\theta_j(\beta),\theta_k(\beta))=(\theta_j(\beta'),\theta_k(\beta'))$. Set $\lambda=\theta_j(\beta)$ and $\lambda'=\theta_k(\beta)$. From the classification of the maximal factorizations of the simple group $\mathrm{P}\Omega_8^+(q)$ (see first line of Table~4 in~\cite{LPS}), we deduce that $$G_0=(G_0)_{\lambda}(G_0)_{\lambda'}$$
and hence simply checking the cardinality, we deduce that
$$(G_0)_\beta=(G_0)_{\lambda}\cap(G_0)_{\lambda'}=(G_0)_{\beta'}.$$
The primitivity of $G$ implies that $(G_0)_\beta$ can fix at most one point and hence $\beta=\beta'$.

Let $\theta(\alpha)=(\lambda_1,\lambda_2,\lambda_3)$ and let $X$ be the setwise stabilizer of $\lambda_1$ in $G$. Then, $X$ is an almost simple group with socle either 
$\mathrm{Sp}_6(q)$ or $\Omega_7(q)$ and the action of $X$ on $\lambda_1$ is the one investigated in Lemma~\ref{lemma:non-subspace}. Therefore, there exist 
$\alpha=\alpha_1,\alpha_2,\alpha_3,\alpha_4\in \lambda_1$ and $\alpha=\alpha_1',\alpha_2',\alpha_3',\alpha_4',\alpha_5'\in \lambda_1$ forming irredundant bases of cardinality $4$ and $5$ for the action of $X$ on $\lambda_1$. We claim that these are also irredundant chains for the action of $G$ on $\Omega$ with 
$$G_{\alpha_1,\alpha_2,\alpha_3,\alpha_4}=
G_{\alpha_1',\alpha_2',\alpha_3',\alpha_4',\alpha_5'}=G_{(\lambda_1)},
$$
where $G_{(\lambda_1)}$ is the pointwise stabilizer of $\lambda_1$ in $G$. To prove this it suffices to show that $G_{\alpha_1,\alpha_2,\alpha_3,\alpha_4}$ and
$G_{\alpha_1',\alpha_2',\alpha_3',\alpha_4',\alpha_5'}$ fix setwise $\lambda_1$ and then use the fact these points form two irredundant bases for the action of $G_{\{\lambda_1\}}$ on $\lambda_1$.

Clearly,
$$G>G_{\alpha_1}>G_{\alpha_1,\alpha_2}>G_{\alpha_1,\alpha_2,\alpha_3}>G_{\alpha_1,\alpha_2,\alpha_3,\alpha_4},$$
because the analogous chain of stabilizers in $X$ is strictly decreasing. Let $Y=G_{\alpha_1,\alpha_2,\alpha_3,\alpha_4}$. Observe that, for each $i$, the first coordinate of $\theta(\alpha_i)$ is $\lambda_1$, that is, 
$$\theta_1(\alpha_1)=\theta_1(\alpha_2)=\theta_1(\alpha_3)=\theta_1(\alpha_4).$$
Assume that $Y$ contains some permutation $g$ moving $\Sigma_1$, say $\Sigma_1^g=\Sigma_2$. Then, since $g$ fixes $\alpha_i$ for each $i$, we deduce
 $$\theta_2(\alpha_1)=\theta_2(\alpha_2)=\theta_2(\alpha_3)=\theta_2(\alpha_4).$$
However, since the mapping $\beta\mapsto (\theta_1(\beta),\theta_2(\beta))$ is bijective, this implies $\alpha_1=\alpha_2=\alpha_3=\alpha_4$, which is clearly a contradiction. This has shown that every element of $Y$ fixes $\Sigma_1$ (and eventually permutes $\Sigma_2$ with $\Sigma_3$). Therefore, $\Sigma_1$ is a system of imprimitivity for the action of $Y$. Since $\alpha_1\in\lambda_1$ and since  $Y$ fixes $\alpha_1$, we deduce that $Y$ fixes setwise $\lambda_1$. With the same argument, we can show that $G_{\alpha_1',\alpha_2',\alpha_3',\alpha_4',\alpha_5'}$ fixes setwise $\lambda_1$. Therefore, $G$ has two irredundant bases of cardinality $4$ and $5$ respectively, and hence it is not $\mathrm{IBIS}$. 
\end{proof}

\section{Almost simple groups having socle $E_6(q)$ or $E_7(q)$: case~\eqref{HiHo8} and~\eqref{HiHo9}}\label{e6e7}
In this case, $G_0 = E_6(q) \leq G$ or $G_0 = E_7(q) \leq G$, and the action is one of those described in cases~\eqref{HiHo8} and~\eqref{HiHo9}.
\begin{proposition}\label{prop:e6}Let $G$ be an almost simple group on $\Omega$ with socle $G_0=E_6(q)$, where the action of $G_0$ on $\Omega$ is permutation equivalent to the action on the right cosets of a parabolic subgroup labeled $P_1$ or $P_6$. Then $G$ is not IBIS.
\end{proposition}
\begin{proof}
Observe that a graph automorphism of $G_0=E_6(q)$ fuses the $G_0$-conjugacy classes of parabolic subgroups labeled $P_1$ and $P_6$.

By~\cite{Bur2018}, $b(G)=6$ and hence, by Lemma~\ref{lemma:minus1}, it suffices to show that $G_0$ has an irredundant base of length at least $7$. 

Let $V$ be the $27$-dimensional natural module for the covering group of $G_0$ over the finite field $\mathbb{F}_q$ and let $\mathcal{P}$ be the collection of all $1$-dimensional subspaces of $V$. From~\cite{CoCo}, the $1$-dimensional subspaces of $V$ are partitioned into three parts $$\mathcal{P}=\mathcal{W}\cup\mathcal{G}\cup \mathcal{B},$$ namely the set of white points $\mathcal{W}$, the set of gray points $\mathcal{G}$ and the set of black points $\mathcal{B}$. In~\cite{CoCo}, it  is noticed that the action of $G_0$ on the right cosets of a maximal parabolic subgroup of type $P_1$ (or $P_6$) is permutation equivalent to the action of $G_0$ on the set of white points $\mathcal{W}$ of $V$.\footnote{Incidentally, this fact was first observed by Freudenthal in~\cite{6}.} In particular, without loss of generality, we may suppose that $\Omega=\mathcal{W}$. 

Let $\mathcal{W}_6$ be the collection of all $6$-dimensional subspaces $W$ of $V$ consisting only of white points, that is, each $1$-dimensional subspace of $W$ is white. From~\cite[(P.3), page~470]{CoCo}, we see  that $\mathcal{W}_6\ne \emptyset$ and that $G_0$ acts transitively on $\mathcal{W}_6$. Let $W\in\mathcal{W}_6$ and let $H=(G_0)_{\{\Delta\}}$ be the setwise stabilizer of $$\Delta=\{\omega\in \Omega\mid \omega\le W\},$$ in other words, $\Delta\subseteq \Omega$ consists of the white points contained in $W$. Again from~\cite[(P.3), page~470]{CoCo}, we deduce that the permutation group induced by $H$ on $\Delta$ is permutation equivalent to an almost simple group $X$ having socle $Y=\mathrm{PSL}_6(q)$ in its natural action on the projective points of a $6$-dimensional vector space $\mathbb{F}_q^6$ over $\mathbb{F}_q$. In particular, we may identify the action of $H=G_{\{\Delta\}}$ on $\Delta$ with the natural action of $X$ on the projective points. We use this identification in what follows.

Now, let $\delta_1=\langle v_1\rangle,\ldots,\delta_6=\langle v_6\rangle$ be elements of $\Delta$ such that the vectors $v_1,\ldots,v_6$ are linearly independent in $\mathbb{F}_q^6$. When $q>2$, let 
$\delta_7=\langle v_1+\cdots+v_6\rangle$. Moreover, let $\ell=6$ when $q=2$ and let $\ell=7$ when $q>2$. Now,
$$G_0>(G_0)_{\delta_1}>(G_0)_{\delta_1,\delta_2}>\cdots>(G_0)_{\delta_1,\delta_2,\ldots,\delta_\ell}.$$
In particular, when $q>2$, we may extend the sequence $(\delta_1,\ldots,\delta_\ell)$ to an irredundant base for $G_0$ of length at least $7$. When $q=2$, by definition $\ell=6$ and, moreover, 
$(G_0)_{\delta_1,\delta_2,\ldots,\delta_\ell}$ is the pointwise stabilizer of $\Delta$. Since this subgroup contains the unipotent radical of the parabolic subgroup $(G_0)_{\delta_1}$\footnote{See~\cite{CoCo} or~\cite{6}.}, we deduce that $(G_0)_{\delta_1,\delta_2,\ldots,\delta_\ell}\ne1 $, hence we may extend the sequence $(\delta_1,\ldots,\delta_\ell)$ to an irredundant base for $G_0$ of length at least $7$.\footnote{We observe that the argument when $q=2$ also applies to the case $q>2$.}
\end{proof}

\begin{proposition}\label{prop:e7}Let $G$ be an almost simple group on $\Omega$ with socle $G_0=E_7(q)$, where the action of $G_0$ on $\Omega$ is permutation equivalent to the action on the right cosets of a parabolic subgroup labeled $P_7$. Then $G$ is not $\mathrm{IBIS}$.
\end{proposition}
\begin{proof}
From~\cite{Bur2018}, $G$ has a base of cardinality $6$. Therefore, from Lemma~\ref{lemma:minus1}, it suffices to prove that $G_0$ has irredundant bases of cardinality larger than $6$.

Let $\Omega$ be the set of cosets of $P_7$ in $G_0$ and let $\omega\in \Omega$. From~\cite{Vasilev}\footnote{Vasilev~\cite{Vasilev} employs a distinct labeling system for the parabolic subgroups compared to that of~\cite{Bourbaki}. Specifically, parabolic subgroups denoted as $P_7$ in~\cite{Bourbaki} are identified as $P_1$ in~\cite{Vasilev}.}, we have $$|\Omega|=\frac{(q^{14}-1)(q^9+1)(q^5+1)}{q-1}.$$ Moreover, again from~\cite{Vasilev}, we see that $G_0$ has (permutational) rank $4$ in its action on $\Omega$, and that $G_0$ has a suborbit of cardinality
\begin{equation}\label{eq:n2}
n_2=\frac{q(q^9-1)(q^8+q^4+1)}{q-1}.
\end{equation} Therefore, $(G_0)_\omega$ has an orbit,  $\Delta$ say, having cardinality $n_2$. Moreover, from~\cite{Vasilev}, we deduce that the action of $(G_0)_\omega$ on $\Delta$ is faithful. From this, we have that $G_0$ admits an irredundant base $(\omega_0=\omega,\omega_1,\ldots,\omega_{\ell-1})$, where $\omega_1,\ldots,\omega_{\ell-1}\in \Delta$. This implies
\begin{align*}
n_2^{\ell-1}&=|\Delta|^{\ell-1}>|(G_0)_\omega|=|P_7|=\frac{|G_0|}{|\Omega|}\\
&=\frac{1}{d}q^{63}(q^9-1)(q^{12}-1)(q^5-1)(q^8-1)(q^6-1)(q^2-1)(q-1),
\end{align*}
where $d=\gcd(2,q-1)$. Using~\eqref{eq:n2}, it can be verified that this implies $\ell\ge 7$.
\end{proof}

\section*{GAP code}
\begin{lstlisting}[language=GAP]
	randomVect := function(A,numIn,size)
		#Create a random vector of given size, by selecting points ,
		#from A, with the starting point numIn
 		local RndVect,x;	
		RndVect := [];
		Add(RndVect,numIn);
		while Size(RndVect) < size do
			x:=Random(A);
			if not x in RndVect then
				Add(RndVect, x);
			fi;
		od;
		return RndVect;
	end;
	isBase := function(G,base)
		#Return true if base is a base for the permutation group G.
		return Size(Stabilizer(G,base,OnTuples))=1;
	end;
	
	IsIrredundant := function(G,B)
		#Return true if the sequence B is irredundant.
		local irr,size,i,stab,newstab;
		irr:=true;
		size:=Size(B);
		stab:=Stabilizer(G,B[1]);
		for i in [2..size] do
			newstab:=Stabilizer(stab,B[i])
			if Size(newstab)=Size(stab) then
				irr:=false;
				break;
			fi;
		stab:=newstab;
		od;
		return irr;
	end;
	FindRandBase := function(G, dim)
		#Find a (random) basis of a given dimension. 
		#If the base is not found after 1000 tries, return false.
		local base, a,Omega,found,i,max;
		found:=false;
		Omega:=MovedPoints(G);
		a:= Random(Omega);
		i:=0;
		base:=[];
		max:=1000;
		while not found and i<max do
			base:=randomVect(Omega,a,dim);
			if IsBase(G,base) and IsIrredundant(G,base) then
				found:=true;
			fi;
			i:=i+1;
		od;
		if not found then 
			base:=[]; 
		fi;
		return [base,found];
	end;
	
	### Codes used for computation with FinInG package
	
	FiningPointwiseStabiliser := function(G, base)
		#Compute the pointwise stabilizer of the
		#sequence base using the package FinInG
		local stab,x;
		stab:=G;
		for x in base do
			stab := Intersection(stab,FiningStabiliser(G,x));
		od;
		return stab;
	end;
	
	FiningIsIrrBase := function(G,base)
		return Size(FiningPointwiseStabiliser(G,base))=1 
			and FiningIsIrredundant(G,base);
	end;
	
	FiningIsIrredundant := function(G,base)
		local irr,H,Htemp,x;
		irr:=true;
		H:=G;
		for x in base do
			Htemp:=Stabiliser(H,x);
			if Size(Htemp) = Size(H) then
				irr:=false;
				break;
			fi;
			H:=Htemp;
		od;
		return irr;
	end;
	
	randomVettMat := function(O,numIn,size)
		#Create a random vector by selecting points from O,
		#starting with the point numIn.
		#O is an orbit of subspace for some classical group G 
		local RndVect, x,diff;
		RndVect := [];
		Add(RndVect,numIn);
		while Size(RndVect) < size do
			x:=Random(O);
			if not x in RndVect then
				Add(RndVect, x);
			fi;
		od;
		return RndVect;
	end;
	
	
	#Examples of use
	
	#Lemma 3.14. We use GAP to show that PSp_4(3) in its action of degree 40 is not IBIS. 
	
	gap> G:=PSp(4,3);;
	gap> FindRandBase(G,4);
	[ [ 15, 13, 4, 9 ], true ]
	gap> FindRandBase(G,5);
	[ [ 14, 13, 11, 1, 34 ], true ]
	
	#We deduce that G has an irredundant base of size 4 and an irredundant base of size 5,
	#and hence is not IBIS.
	
	gap> FindRandBase(G,6);
	[ [], false ];
	
	#In this case, we may not deduce that G has not a base of cardinality 6.
	
	#Lemma 3.15. We use GAP to show that PSp_6(2), in its action on
	#the set of totally singular 2-dimensional subspaces, has two
	#irredundant bases satisfying the conditions of Lemma 3.15.
	#We need two auxiliary functions.
	
	MeetSubspaces :=  function(V,list)
		#compute the intersection of the subspaces in list
		local x, meet,n;
		n:=Size(list);
		meet:=Meet(list[1],list[2]);
		for x in [3..n] do
			meet := Meet(meet,list[x]);
		od;
		return meet;
	end;
	
	SumSubspaces := function(V,list)
		#compute the sum of the subspaces in list
		local x;
		base:=[];
		for x in list do
			base:=Union(base,Unpack(UnderlyingObject(x)));
		od;
		return Subspace(V,base);
	end;
	
	gap> LoadPackage("Fining");
	true
	gap> ps:=SymplecticSpace(5,2);
	W(5, 2)
	gap> G:=IsometryGroup(ps);
	PSp(6,2)
	gap> spaces:=Lines(ps);
	<lines of W(5, 2)>
	gap> V:=GF(2)^6;
	( GF(2)^6 )
	gap> found:=false;
	false;
	gap> b1:=[];
	[ ];
	gap> while not found do b1:=randomVettMat(spaces,Random(spaces),4);
	if FiningIsIrrBase(G,b1) and V = SumSubspaces(V,b1)
	and IsEmptySubspace(MeetSubspaces(V,b1))
	then found:=true; fi; od;
	gap> b1;
	[ <a line in W(5, 2)>, <a line in W(5, 2)>,
	<a line in W(5, 2)>, <a line in W(5, 2)> ]
	gap> found:=false;
	false
	gap> while not found do b2:=randomVettMat(spaces,Random(spaces),5);
	if FiningIsIrrBase(G,b2) and V = SumSubspaces(V,b2)
	and IsEmptySubspace(MeetSubspaces(V,b2)) then 
	ound:=true; fi; od;
	gap> b2;
	[ <a line in W(5, 2)>, <a line in W(5, 2)>, <a line in W(5, 2)>,
	<a line in W(5, 2)>, <a line in W(5, 2)> ]
\end{lstlisting}


\thebibliography{20}

	
\bibitem{Robert}R.~F.~Bailey,
Uncoverings-by-bases for base-transitive permutation 
groups, \textit{Des. Codes Cryptogr.} \textbf{41} (2006), no. 2, 153--176. 

\bibitem{FINING}  J.~Bamberg, A.~Betten, P.~Cara,  J.~De Beule,  M.~Lavrauw, M.~Neunhoeffer,  and  M.~Horn, FinInG, Finite Incidence Geometry, Version 1.5.6 (2023)

\bibitem{bhr}J.~N.~Bray, D.~F.~Holt, C.~M.~Roney-Dougal, \textit{The maximal subgroups of the low-dimensional finite classical groups}, Cambridge: Cambridge University Press, 2013.

\bibitem{Bourbaki}N.~Bourbaki, \textit{Groupes et Algebr\`es de Lie (Chapters 4, 5 and 6)}, Hermann, Paris, 1968.
   
\bibitem{B1}T.~C.~Burness, Fixed point ratios in actions of finite classical groups. I, \textit{J. Algebra} \textbf{309} (2007), 69--79. 


\bibitem{B2}T.~C.~Burness, Fixed point ratios in actions of finite classical groups. II, \textit{J. Algebra} \textbf{309} (2007), 80--138. 


\bibitem{B3}T.~C.~Burness, Fixed point ratios in actions of finite classical groups. III, \textit{J. Algebra} \textbf{309} (2007), 693--748. 

\bibitem{B4}T.~C.~Burness, Fixed point ratios in actions of finite classical groups. IV, \textit{J. Algebra} \textbf{309} (2007), 749--788. 

\bibitem{Bur07} T.~C.~Burness, On base sizes for actions of finite classical groups, \textit{J. Lond. Math. Soc. (2)} \textbf{75} (2007), 545--562.
\bibitem{Bur2018}T.~C.~Burness, On base sizes for almost simple primitive groups, \textit{J. Algebra} \textbf{516} (2018), 38--74.

\bibitem{BurGiu} T.~C.~Burness, M.~Giudici, Classical Groups, Derangements and Primes, 	\textit{Australian Mathematical Society Lecture Series}, Cambridge University Press, 2016.
 
\bibitem{BGS11}T.~C.~Burness, R.~M.~Guralnick, J.~Saxl, On base sizes for symmetric groups, \textit{Bull. Lond. Math. Soc. 43} (2011), 386--391.
\bibitem{BLS09} T.~C.~Burness, M.~W.~Liebeck, A.~Shalev, Base sizes for simple groups and a conjecture of Cameron, \textit{Proc. Lond. Math. Soc. (3)} \textbf{98} (2009), 116--162.
\bibitem{BOW10} T.~C.~Burness, E.~A.~O’Brien, R.~A.~Wilson, Base sizes for sporadic simple groups, \textit{Israel J. Math.} \textbf{177} (2010), 307--333.


\bibitem{PolarSpaces}P.~J.~Cameron, \textit{Projective and Polar Spaces}, University of London, Queen Mary and Westfield college, 1992.
\bibitem{cameron}P.~J.~Cameron, W.~M.~Kantor, Random permutations: some group-theoretic aspects, \textit{Combin. Probab. Comput.} \textbf{2} (1993) 257--262. 
 
\bibitem{Peter}P.~J.~Cameron, \textit{Permutation Groups}, London Mathematical Society Student Texts, Cambridge University Press, 1999. 
\bibitem{CF}P.~J.~Cameron, D.~G.~Fon-Der-Flaass, Bases for permutation groups and matroids, \textit{Eur. J. Comb.} \textbf{16} (1995), 537--544.
\bibitem{Cam92}P.~J.~Cameron,``Some open problems on permutation group'' in Groups, Combinatorics and
Geometry. Proceedings of the L.M.S. Durham Symposium, Held July 5--15, 1990 in Durham, UK,
Cambridge University Press, Cambridge, MA, 1992, 340--350.
\bibitem{CoCo}A.~M.~Cohen, B.~N.~Cooperstein, The $2$-spaces of the standard $E_6(q)$-module, \textit{Geom. Dedicata} \textbf{25} (1988), 467--480. 
\bibitem{DMS}F.~Dalla Volta, F.~Mastrogiacomo, P.~Spiga, On the cardinality of irredundant and minimal  bases of finite permutation groups, \textit{J. Algebr. Comb.}, to appear. 
\bibitem{dixon_mortimer}J.~D.~Dixon,  B.~Mortimer, \textit{Permutation groups}, Graduate Texts in Mathematics \textbf{163}, Springer-Verlag, New York, 1996.
\bibitem{Dye}R.~H.~Dye, Interrelations of Symplectic and Orthogonal Groups in Characteristic Two, \textit{J. Algebra} \textbf{59} 1979, 202--221. 
\bibitem{6}H.~Freudenthal, Beziehungen der $E_7$ und $E_8$ zur Oktavenebene I.,  \textit{Indagationes Math.} \textbf{16} 1954,
218--230.
\bibitem{GAP}The GAP Group, GAP -- Groups, Algorithms, and Programming, Version 4.13.1; 2024. (https://www.gap-system.org)

\bibitem{GiLoSp}N.~Gill, B.~Lod\'a, P.~Spiga, On the height and relational complexity of a finite permutation group, \textit{Nagoya Math. J.} \textbf{246} (2022), 372--411.
\bibitem{GuPrSp}S.~Guest, A.~Previtali, P.~Spiga, A remark on the permutation representations afforded by the embeddings of $\mathrm{O}^\pm_{2m}(2^f)$ in $\mathrm{Sp}_{2m}(2^f)$,
\textit{Bull. Aust. Math. Soc.} \textbf{89} (2014),  331--336. 

\bibitem{jordan}C.~Jordan, \textit{Trait\'e des Substitutions et des \'Equations Alg\'ebriques}, Gauthier-Villars, Paris, 1870.
\bibitem{KRD} V.~Kelsey, M.~C.~Roney-Dougal, On relational complexity and base size of finite primitive groups, \textit{Pacific Journal of Mathematics} (2021), 89--108.
\bibitem{KL}P.~Kleidman, The maximal subgroups of the finite 8-dimensional orthogonal groups $\mathrm{P}\Omega^+_8(q)$ and of their automorphism groups, \textit{J. Algebra} \textbf{110} (1987), 173--242.

\bibitem{KL90}P.~Kleidman, M.~Liebeck, \textit{The Subgroup Structure of the Finite Classical Groups}, Lond. Math. Soc. Lect. Note Ser. \textbf{129}, Cambridge University Press, Cambridge, MA, 1990.


\bibitem{lang}S.~Lang, \textit{Algebra. 3rd revised ed.}, Graduate Texts in Mathematics \textbf{211}, 2002, New York, NY: Springer.
\bibitem{Lee}M.~Lee, Primitive almost simple IBIS groups with sporadic socle,  \href{https://arxiv.org/abs/2302.01521}{arXiv:2302.01521}. 
\bibitem{LeSp}M.~Lee, P.~Spiga, A classification of finite primitive IBIS groups with alternating socle, \textit{J. Group Theory} \textbf{26} (2023),  915--930.

\bibitem{LPS}
M.~W. Liebeck, C.~E. Praeger and J. Saxl, The maximal factorizations of the finite simple groups and their automorphism groups, \emph{Mem. Amer. Math. Soc.} 86 (1990).


\bibitem{LS99} M.~W.~Liebeck, A.~Shalev, Simple groups, permutation groups, and probability, \textit{J. Amer. Math. Soc.} \textbf{12} (1999), 497--520.
\bibitem{LuMoMo}A.~Lucchini, M.~Morigi, M.~Moscatiello, Primitive permutation IBIS groups, \textit{J. Combin. Theory Ser. A} \textbf{184} (2021), Paper No. 105516. 
\bibitem{LuMa}A.~Lucchini, D.~Malinin, IBIS soluble linear groups, \textit{Comm. Algebra} \textbf{52} (2024), 457--465.
\bibitem{Malle}G.~Malle, Explicit realization of the Dickson groups $G_2(q)$ as Galois groups, \textit{Pacific J. Math.} \textbf{212} (2003), 157--167.
\bibitem{Pra90}C.~E.~Praeger, The inclusion problem for ﬁnite primitive permutation groups, \textit{Proc. Lond. Math. Soc. (3)} \textbf{60} (1990), 68--88.
\bibitem{reis} L.~Reis, S.~Ribas, Generators of finite fields with prescribed traces, \textit{Journal of the Australian Mathematical Society}, 112 (2022), pp. 355-366.

\bibitem{Vasilev}A.~V.~Vasilyev, Minimal permutation representations of finite simple exceptional groups of type $E_6$, $E_7$ and $E_8$, \textit{Algebra and Logic} \textbf{36} (1997), 302--3010.
\end{document}